\documentclass[11pt]{article}

\usepackage{amssymb,amsfonts,amsmath,amsthm}
\usepackage{graphicx}
\usepackage{graphics}

\newtheorem{theorem}{Theorem}[section]
\newtheorem{lemma}[theorem]{Lemma}

\newtheorem{remark}{Remark}[section]
\newtheorem{definition}{Definition}[section]

\usepackage{color}
\usepackage[colorlinks]{hyperref}

\numberwithin{equation}{section}

\begin{document}
	
		\title{\sc Long-time Dynamics for a Cahn-Hilliard Tumor Growth Model with Chemotaxis}
	
	\author{
		{\sc Harald Garcke}\thanks{Email: \href{mailto:harald.garcke@ur.de}{harald.garcke@ur.de}} \\
		\small  Faculty of Mathematics, Regensburg University, \\
		\small 93040  Regensburg.
		\smallskip	\\	
		{\sc Sema Yayla}\thanks{Email: \href{mailto:semasimsek@hacettepe.edu.tr}{semasimsek@hacettepe.edu.tr}} \\
		\small Department of Mathematics, Faculty of Science, Hacettepe University,	\\
		\small  Beytepe 06800, Ankara, Turkey.
		\smallskip}

	\date{}
	\maketitle

\begin{abstract}
	Mathematical models that describe the tumor growth process have been formulated by several authors in order to understand how cancer develops and to develop new treatment approaches. In this study, it is aimed to investigate the long-time behavior of the two-phase diffuse-interface model, which was proposed in \cite{hawkins-daarud} to model a tumor tissue as a mixture of cancerous and healthy cells. Up to now, studies on the long-time behavior of this model have neglected chemotaxis and active transport, which have a significant effect on tumor growth. In this research, our main aim is to study this model with chemotaxis and active transport. We prove an asymptotic compactness result for the weak solutions of the problem in the whole phase-space $ H^{1}\left( \Omega\right) \times L^{2}\left( \Omega\right)$. We establish the existence of the global attractor in a phase space defined via mass conservation. We also prove that the global attractor equals to the unstable manifold emanating from the set of stationary points. Moreover, we obtain that the global attractor has a finite fractal dimension.
	
\end{abstract}

\smallskip


\section{Introduction}

In this paper, we will investigate the long-time behavior of the following
tumor growth model:%
\begin{align}
&\varphi _{t}=div\left( m\left( \varphi \right) \nabla \mu \right)
+p\left( \varphi \right) \left( \chi _{\sigma }\sigma +\chi _{\varphi
}\left( 1-\varphi \right) -\mu \right) &\text{in }\left( 0,T\right) \times \Omega ,  \tag{1.1}\\
&\mu =-\Delta \varphi +\Psi ^{\prime }\left( \varphi \right) -\chi _{\varphi
}\sigma &\text{in }\left( 0,T\right) \times \Omega ,  \tag{1.2}\\
&\sigma _{t}=div\left( n\left( \varphi \right) \left( \chi _{\sigma
}\nabla \sigma -\chi _{\varphi }\nabla \varphi \right) \right) -p\left(
\varphi \right) \left( \chi _{\sigma }\sigma +\chi _{\varphi }\left(
1-\varphi \right) -\mu \right) \ \ &\text{in }\left( 0,T\right) \times \Omega ,\bigskip  \tag{1.3}\\
&\partial _{\nu }\mu =\partial _{\nu }\varphi =\partial _{\nu }\sigma =0&\text{on }\left( 0,T\right) \times\Gamma,  \tag{1.4}
\end{align}
with initial conditions%
\begin{equation}
\varphi \left( 0\right) =\varphi _{0}\text{, }\sigma \left( 0\right) =\sigma
_{0},  \tag{1.5}
\end{equation}%
where $\Omega \subset 
%TCIMACRO{\U{211d} }%
%BeginExpansion
\mathbb{R}
%EndExpansion
^{3}$ is a bounded domain with smooth boundary $\Gamma$ and $%
\partial _{\nu }$ in (1.4) stands for the normal derivative where $ \nu $ is the outer unit normal to $\Gamma$.

This model was introduced by A. Hawkins-Daarud, Van der Zee and Oden in \cite{hawkins-daarud} as a continuum-mixture model for tumor growth. The system (1.1) is a Cahn-Hilliard type equation in which the order parameter $\varphi \in \left[-1,1\right] $ is equal to $1$ in the tumorous phase and equal to $-1$ in
the healthy phase, and $\mu $ denotes the chemical potential for $\varphi $. Equation (1.3) is a reaction-diffusion equation where $\sigma $
represents the chemical concentration that serves as nutrient for the tumor. Here, $m\left( \varphi \right) $ and  $n\left( \varphi \right) $
are positive mobilities which denote the diffusivity of the binary mixture
and the chemicals. Furthermore, $\Psi$ is a potential with two equal minima at $\ \pm 1\ $. Finally, $\chi _{\varphi }\geq 0$ is
a constant which represents of the chemotaxis and active transport\textbf{,} and also the constant $\chi_{\sigma }\geq 0$ stands for the chemical mobility\textbf{. }

A similar system to (1.1)-(1.3) was studied in \cite{garcke1}- \cite{garcke4}%
.\textit{\ } In \cite{garcke1}, by using \ principal thermodynamic rules, a
Cahn-Hilliard-Darcy model was proposed for tumor growth with chemotaxis and
active transport. With the help of asymptotic analysis,
the authors developed sharp interface models of tumor growth, and then stability and instability of radial solutions of a specific sharp interface model was investigated.
Some numerical calculations were also carried out to see the effect of the
active transport term on the tumor growth phenomena. In \cite{garcke2} also
a Cahn-Hilliard-Darcy system was investigated. In that paper, in the two and
three dimensional case, existence of global weak solutions was obtained by
means of Galerkin method. In \cite{garcke3}, the following diffuse interface
model coupled with Neumann boundary conditions was considered:%
\begin{align}
&\varphi _{t}=div\left( m\left( \varphi \right) \nabla \mu \right)
+(\lambda _{p}\sigma -\lambda _{a})h\left( \varphi \right) &\text{in }\left( 0,T\right) \times\Omega ,  \tag{1.5}\\
&\mu =A\Psi ^{\prime }\left( \varphi \right) -B\Delta \varphi -\chi _{\varphi
}\sigma &\text{in }\left( 0,T\right) \times\Omega,  \tag{1.6}\\
&\kappa \sigma _{t}=div\left( n\left( \varphi \right) \left( \chi
_{\sigma }\nabla \sigma -\chi _{\varphi }\nabla \varphi \right) \right)
-\lambda _{c}\sigma h\left( \varphi \right)&\text{in }\left( 0,T\right) \times\Omega,   \tag{1.7}\\
&\partial _{\nu } \mu =\partial _{\nu }\varphi =0\text{ and }n\left( \varphi
\right) \chi _{\sigma }\partial _{\nu } \sigma =K\left( \sigma _{\infty
}-\sigma \right) &\text{on } \left( 0,T\right)\times\Gamma  .  \tag{1.8}
\end{align}
Here, $\kappa =1$, $A$, $B$ and $K$ are positive constants, $\lambda
_{p},\lambda _{a}$ and $\lambda _{c}$ are also nonnegative constants which stand
for the proliferation rate, the apoptosis rate of the tumor cells, and the
consumption rate of the nutrient, respectively. The interpolation function $h\left( \varphi\right) $ is chosen such that $h\left( 1\right) =0$ and $%
h\left( -1\right) =0$, and $\sigma _{\infty }$ represents the nutrient
amount on the boundary. Additionally, $m\left( \varphi \right) $ and $%
n\left( \varphi \right) $ as well as $\chi _{\varphi }$ and $\chi _{\sigma
} $ represent the same quantities as in the system (1.1)-(1.3). In that paper,
by using Galerkin approximation well-posedness of the model (1.5)-(1.8) and
its quasi-static version ($\kappa =0$) was established for
regular potentials with quadratic growth. On the other hand, in \cite%
{garcke4}, well-posedness of the problem (1.5)-(1.7) with Dirichlet boundary
conditions was proved for regular potentials with higher polynomial growth
and even for singular potentials. However, in \cite{garcke3}, the authors
were not able to determine whether the weak solutions of the system (1.5)-(1.8)
converges to weak solutions of the quasi-static model as $\kappa \rightarrow
0$. In that case, well-posedness of the quasi-static model was shown
independently from (1.5)-(1.8). Conversely, if (1.5)-(1.7) is equipped with
Dirichlet boundary conditions, it was proven in \cite{garcke4} that a quasistatic model can be obtained from (1.5)-(1.7) as a result of the limit $\kappa \rightarrow 0$. Moreover, long-time behavior of the model (1.5)-(1.8) without chemotaxis has recently been studied in \cite{MRS}.

In the case $\chi _{\varphi }=0$,
the model (1.1)-(1.4) with $m\left( \varphi \right) =n\left( \varphi \right)
=\chi _{\sigma }=1$ was investigated in \cite{frigieri}. In that paper, well-posedness of the problem in  $ H^{1}\left( \Omega\right) \times L^{2}\left( \Omega\right)$ and the attractor in a suitably defined subspace are established. In the
paper \cite{colli1}-\cite{colli2}, problem was considered with additional viscosity regularization term. In \cite{colli1}, the authors proved the existence and
uniqueness of strong solutions of the problem under the conditions $P:\mathbb{R}\rightarrow \mathbb{R}$ is nonnegative, bounded and Lipschitz continuous and $\Psi =\Psi _{1}+\Psi_{2}$ where $\Psi _{1}:\mathbb{R}\rightarrow \left[ 0,+\infty \right] $ is convex, proper lower semicontinuous, $\Psi _{2}\in C^{1}\left(\mathbb{R}\right) $ is nonnegative and $\Psi _{2}^{\prime }$ is Lipschitz continuous.
Moreover, the authors analyzed the long-time dynamics of the problem in terms of
an $ \omega $-limit set. In \cite{colli2}, \ some of the
results obtained in \cite{colli1} are extended to the case that the viscosity coefficients are independent from each other. Additionally, some assumptions on
the initial data were weakened while assuming the same conditions
on $P$ and $\Psi $. The formally matched asymptotic limit of a quasi-static variant of the diffuse-interface tumor-growth model (1.1)-(1.4) has been studied in \cite{HKNZ}.

In this paper, different from \cite{frigieri} and other previous works, we consider the problem in the case $ \chi _{\varphi }>0 $. Our main novelties are as the following:
\begin{itemize}
	\item We obtain an asymptotic compactness result in whole phase-space $ H^{1}\left( \Omega\right) \times L^{2}\left( \Omega\right)$. This result yields that $ \omega$-limit sets of bounded sets are nonempty, compact, invariant (see Lemma 5.6). Comparing with previous works, this will be the first result on the long-time behavior of the problem (1.1)-(1.4) in the whole phase-space.
	\item We analyze the corresponding stationary problem in $ Z_{M} $, which is a closed subspace of $  H^{1}\left( \Omega\right) \times L^{2}\left( \Omega\right) $ defined with the help of mass conservation (see (5.32)). We obtain that the set of stationary points is non empty and bounded in $ Z_{M} $.
	\item  Considering the asymptotic compactness of the system, boundedness of the set of stationary points and the existence of a strict Lyapunov functional (gradient property), we establish the existence of the global attractor in $ Z_{M} $. Moreover, these properties yield that the global attractor equals the unstable manifold (see Definition 5.3) emanating from the set of stationary points.
	\item In this paper, we also obtain that the global attractor is finite dimensional. To the best of our knowledge, this is the first result regarding the dimension of global attractor for the tumor growth model (1.1)-(1.4) even for the case  $ \chi _{\varphi }=0 $.
\end{itemize}
The plan of the paper is as follows. In Section 2, we give some notation and preliminaries. In Section 3, we introduce the assumptions and state the well-posedness results. In Section 4, we prove the results given in Section 3. In Section 5, we established existence of the smooth global attractor. Finally, in the last section we show that the fractal dimension of the global attractor is bounded.
 
\section{Notation and Preliminaries}
At first, we want to fix some notation.\newline
For a Banach space $X$, we use the symbol $\left\Vert .\right\Vert _{X}$ for
its norm, $\overset{\ast }{X}$ for its dual space and $\left\langle
\left\langle .,.\right\rangle \right\rangle $ for the duality pairing between $%
\overset{\ast }{X}$ and $X$. Also, we will use the notation $\left\langle
.,.\right\rangle _{X}$ for the inner product in $X$. The inner product in $L^{2}\left( \Omega \right) $ will be denoted $\left\langle
.,.\right\rangle $.\newline
For every $f\in \overset{\ast }{H^{1}}\left( \Omega \right)$, we introduce the mean value by%
\begin{equation*}
\overline{f}:=\frac{1}{\left| \Omega\right|   }\left\langle \left\langle f,1\right\rangle
\right\rangle \text{ for every }f\in \overset{\ast }{H^{1}}\left( \Omega \right)\text{.}
\end{equation*}%
Moreover, we introduce the operator $A:H^{1}\left( \Omega \right)
\rightarrow \overset{\ast }{H^{1}}\left( \Omega \right) $ such that%
\begin{equation*}
Au=-\Delta u+u\text{ and }D\left( A\right) =\left\{ \varphi \in H^{2}\left(
\Omega \right) :\partial _{\nu }\varphi =0\text{ on }\Gamma \right\} \text{.}
\end{equation*}%
It is worth noting that the restriction of $A$ on $D\left( A\right) $ is an
isomorphism from $D\left( A\right) $ onto $L^{2}\left( \Omega \right) $, i.e. $D(A)=A^{-1}(L^{2}(\Omega))$.
In addition, following identities hold:%
\begin{equation*}
\begin{array}{ll}
\left\langle \left\langle Au,A^{-1}v^{\ast }\right\rangle \right\rangle
=\left\langle \left\langle v^{\ast },u\right\rangle \right\rangle &\text{ for
	every }u\in H^{1}\left( \Omega \right) \text{, }v^{\ast }\in \overset{\ast }{%
	H^{1}}\left( \Omega \right),\\
\left\langle \left\langle u^{\ast },A^{-1}v^{\ast }\right\rangle
\right\rangle =\left\langle u^{\ast },v^{\ast }\right\rangle _{\overset{\ast 
	}{H^{1}}\left( \Omega \right) }&\text{ for every }u^{\ast }\text{, }v^{\ast
}\in \overset{\ast }{H^{1}}\left( \Omega \right),
\end{array}
\end{equation*}%
where $\left\langle .,.\right\rangle _{\overset{\ast }{H^{1}}\left( \Omega
\right) }$ is the dual inner product in $\overset{\ast }{H^{1}}\left( \Omega
\right) $ corresponding to the usual inner product in $H^{1}\left( \Omega
\right) $. Domain of the inverse operator $ A^{-1} $ is defined as $ D\left( A^{-1} \right)=  \overset{\ast }{D\left(A \right) }  $ (see \cite{miranville}). Also, we have%
\begin{equation*}
\left\langle \left\langle v^{\ast },u\right\rangle \right\rangle
=\int\limits_{\Omega }v^{\ast }u\text{ if }v^{\ast }\in L^{2}\left( \Omega
\right) \text{, }
\end{equation*}%
and%
\begin{equation*}
\frac{d}{dt}\left\Vert v^{\ast }\right\Vert _{\overset{\ast }{H^{1}}\left(
\Omega \right) }^{2}=2\left\langle \left\langle \partial _{t}v^{\ast
},A^{-1}v^{\ast }\right\rangle \right\rangle \text{ for every }v^{\ast }\in
H^{1}\left( 0,T;\overset{\ast }{H^{1}}\left( \Omega \right) \right) .
\end{equation*}%
Furthermore, as a consequence of classical spectral theory the operator $A$
has eigenvalues $\lambda _{j}$ with $0<\lambda _{1}\leq \lambda _{2}\leq ...$
and $\lambda _{j}\rightarrow \infty $, and a family of eigenfunctions $%
w_{j}\in D\left( A\right) $ such that $Aw_{j}=\lambda _{j}w_{j}$. The
sequence $\left\{ w_{j}\right\} $ generates an orthonormal basis in $%
L^{2}\left( \Omega \right) $ and also it is orthogonal in $H^{1}\left(
\Omega \right) $ and $D\left( A\right) $.

\section{Statement of the Problem and Well-posedness Results}

In this section we fix our assumptions and state the well-posedness results.\newline

\textbf{Assumptions}

\begin{description}
 \item[(A1)]  The coefficients $\chi _{\varphi }$ and $\chi _{\sigma }$ are
assumed to be fixed, positive constants in the whole paper.

\item[(M)] The mobilities $m$, $n$ \ are continuous on $%
%TCIMACRO{\U{211d} }%
%BeginExpansion
\mathbb{R}
%EndExpansion
$ satisfying either one of the following,

\begin{description}
\item[(M1)] For some positive constants $m_{0},$ $m_{1},$ $n_{0}$ and $n_{1}$%
\begin{equation}
m_{0}\leq m\left( s\right) \leq m_{1}\text{, \ \ }n_{0}\leq n\left( s\right)
\leq n_{1}\text{\ }\ \forall s\in 
%TCIMACRO{\U{211d} }%
%BeginExpansion
\mathbb{R}
%EndExpansion
;\text{\ }  \tag{3.1}
\end{equation}

\item[(M2)] 
\begin{equation}
m\left( s\right) =n\left( s\right) =1.  \tag{3.2}
\end{equation}
\end{description}

\item[($\mathbf{\Psi }$)] The potential $\Psi \in C^{2}\left( 
%TCIMACRO{\U{211d} }%
%BeginExpansion
\mathbb{R}
%EndExpansion
\right) $ can be written as%
\begin{equation}
\Psi \left( s\right) =\Psi _{0}\left( s\right) +\lambda \left( s\right) 
\tag{3.3}
\end{equation}%
where $\Psi _{0}\in C^{2}\left( 
%TCIMACRO{\U{211d} }%
%BeginExpansion
\mathbb{R}
%EndExpansion
\right) $ and $\lambda \in C^{2}\left( 
%TCIMACRO{\U{211d} }%
%BeginExpansion
\mathbb{R}
%EndExpansion
\right) $ satisfies $\left\vert \lambda ^{\prime \prime }\left( s\right)
\right\vert \leq \alpha $, for all $s\in 
%TCIMACRO{\U{211d} }%
%BeginExpansion
\mathbb{R}
%EndExpansion
$, and for some constant $\alpha \geq 0$. Moreover, we assume%
\begin{equation}
c_{1}\left( 1+\left\vert s\right\vert ^{\rho -2}\right) \leq \Psi
_{0}^{\prime \prime }\left( s\right) \leq c_{2}\left( 1+\left\vert
s\right\vert ^{\rho -2}\right),  \tag{3.4}
\end{equation}%
\begin{equation}
\Psi \left( s\right) \geq R_{1}\left\vert s\right\vert ^{2}-R_{2}  \tag{3.5}
\end{equation}%
for all $s\in 
%TCIMACRO{\U{211d} }%
%BeginExpansion
\mathbb{R}
%EndExpansion
$, with $c_{1}$, $c_{2}$, $R_{1}>\frac{2\chi _{\varphi }^{2}}{\chi _{\sigma }%
},$ $R_{2}\in 
%TCIMACRO{\U{211d} }%
%BeginExpansion
\mathbb{R}
%EndExpansion
$ and with $\rho \in \left[ 2,6\right) $.

\item[(P)] The interpolation function $p\in C_{loc}^{0,1}\left( 
%TCIMACRO{\U{211d} }%
%BeginExpansion
\mathbb{R}
%EndExpansion
\right) $ satisfies either one of the following,

\begin{description}
\item[(P1)] 
\begin{equation}
0\leq p\left( s\right) \leq c_{3}\left( 1+\left\vert s\right\vert ^{q}\right)
\tag{3.6}
\end{equation}
for all $s\in 
%TCIMACRO{\U{211d} }%
%BeginExpansion
\mathbb{R}
%EndExpansion
$, with $c_{3}>0$ and $q\in \left[ 1,9\right) $\thinspace ;

\item[(P2)] $p>0$ and%
\begin{equation}
\left\vert p^{\prime }\left( s\right) \right\vert \leq c_{4}\left(
1+\left\vert s\right\vert ^{q-1}\right)  \tag{3.7}
\end{equation}%
for all $s\in 
%TCIMACRO{\U{211d} }%
%BeginExpansion
\mathbb{R}
%EndExpansion
$, with $c_{4}>0$ and with $q\in \left[ 1,4\right] .$
\end{description}
\end{description}
\begin{remark}
	Assumption (3.5) is only needed if $ \rho=2 $.
\end{remark}
Before stating the existence theorems we will specify the definition of weak
and strong solutions.

\begin{definition}
Let $\left( \varphi _{0},\sigma _{0}\right) \in H^{1}\left( \Omega \right)
\times L^{2}\left( \Omega \right) $ and $T\in \left( 0,\infty \right) $ be
given. A pair $\left( \varphi ,\sigma \right) $, satisfying the properties%
\begin{align}
\left( \varphi ,\sigma \right) \in L^{\infty }\left( 0,T;H^{1}\left( \Omega
\right) \times L^{2}\left( \Omega \right) \right) \cap L^{2}\left(
0,T;H^{2}\left( \Omega \right) \times H^{1}\left( \Omega \right) \right) , 
\tag{3.8}\\
\left( \varphi _{t},\sigma _{t}\right) \in L^{r}\left( 0,T;D\left(
A^{-1}\right) \times D\left( A^{-1}\right) \right),   \tag{3.9}\\
\mu :=-\Delta \varphi +\Psi ^{\prime }\left( \varphi \right) -\chi _{\varphi
}\sigma \in L^{2}\left( 0,T;H^{1}\left( \Omega \right) \right) ,  \tag{3.10}\\
\left( \varphi \left( 0\right) ,\sigma \left( 0\right) \right) =\left(
\varphi _{0},\sigma _{0}\right) ,  \tag{3.11}
\end{align}%
for some $r>1$, is called a weak solution of the problem (1.1)-(1.5) on $%
\left[ 0,T\right] \times \Omega ,$ iff%
\begin{align}
\left\langle \left\langle \varphi _{t},\eta \right\rangle \right\rangle
+\left\langle m\left( \varphi \right) \nabla \mu ,\nabla \eta \right\rangle
&=\left\langle p\left( \varphi \right) \left( \chi _{\sigma }\sigma +\chi
_{\varphi }\left( 1-\varphi \right) -\mu \right) ,\eta \right\rangle  
\tag{3.12}\\
\left\langle \left\langle \sigma _{t},\xi \right\rangle \right\rangle
+\left\langle n\left( \varphi \right) \left( \chi _{\sigma }\nabla \sigma
-\chi _{\varphi }\nabla \varphi \right) ,\nabla \xi \right\rangle
&=-\left\langle p\left( \varphi \right) \left( \chi _{\sigma }\sigma +\chi
_{\varphi }\left( 1-\varphi \right) -\mu \right) ,\xi \right\rangle  
\tag{3.13}
\end{align}%
holds in $\left( 0,T\right) \times \Omega $, for every $\eta ,$ $\xi \in
D\left( A\right) .$ \newline
If the pair $\left( \varphi ,\sigma \right) $ also satisfies the
properties
\begin{align}
& \varphi \in L^{\infty }\left(0,T;H^{3}\left( \Omega
\right)\right), &   \varphi _{t} \in L^{2}\left( 0,T;H^{1}\left(
\Omega \right) \right) \tag{3.14}\\
& \sigma \in L^{\infty }\left(0,T;H^{1}\left( \Omega
\right)\right)\cap L^{2}\left( 0,T;H^{2}\left( \Omega \right) \right), &\sigma _{t} \in L^{2}\left( 0,T; L^{2}\left( \Omega \right) \right)   \tag{3.15}\\
&\mu \in L^{\infty }\left( 0,T;H^{1}\left( \Omega \right) \right)\cap L^{2}\left( 0,T;H^{3}\left( \Omega \right) \right) , \tag{3.16}
\end{align}
\ then, it is called a strong solution of the problem (1.1)-(1.5) on $\left[
0,T\right] \times \Omega $.
\end{definition}

\begin{remark}
The initial condition (3.11) is meaningful because from the regularity of
weak solution it follows that%
\begin{equation*}
\left( \varphi ,\sigma \right) \in C_{w}\left( \left[ 0,T\right]
;H^{1}\left( \Omega \right) \times L^{2}\left( \Omega \right) \right) .
\end{equation*}
\end{remark}

Our result on the existence of weak solutions is stated in the following
theorem:

\begin{theorem}
Assume that the conditions (M1), ($\Psi $) and (P1) are satisfied. Then for
every $\left( \varphi _{0},\sigma _{0}\right) \in H^{1}\left( \Omega \right)
\times L^{2}\left( \Omega \right) $, the problem (1.1)-(1.5) has a weak
solution such that%
\begin{align}
&\varphi \in L^{2}\left( 0,T;H^{3}\left( \Omega \right) \right) \text{, }&\Psi
\left( \varphi \right) \in L^{\infty }\left( 0,T;L^{1}\left( \Omega \right)
\right) ,  \tag{3.17}\\
&\nabla N_{\sigma }\in L^{2}\left( 0,T;L^{2}\left( \Omega \right) \right) 
\text{, }&\sqrt{p\left( \varphi \right) }\left( N_{\sigma }-\mu \right) \in
L^{2}\left( 0,T;L^{2}\left( \Omega \right) \right) .  \tag{3.18}
\end{align}%
Moreover, following energy inequality holds for the weak solutions,%
\begin{equation}
E\left( \varphi \left( t\right) ,\sigma \left( t\right) \right)
+\int\limits_{0}^{t}\int\limits_{\Omega }\left( m\left( \varphi \right)
\left\vert \nabla \mu \right\vert ^{2}+n\left( \varphi \right) \left\vert
\nabla N_{\sigma }\right\vert ^{2}+p\left( \varphi \right) \left( N_{\sigma
}-\mu \right) ^{2}\right) dxd\tau \leq E\left( \varphi _{0},\sigma
_{0}\right)  \tag{3.19}
\end{equation}%
where $E\left( \varphi ,\sigma \right) =\frac{1}{2}\left\Vert \nabla \varphi
\right\Vert _{L^{2}\left( \Omega \right) }^{2}+\int \Psi \left( \varphi
\right) +\frac{\chi _{\sigma }}{2}\left\Vert \sigma \right\Vert
_{L^{2}\left( \Omega \right) }^{2}+\chi _{\varphi }\int \sigma \left(
1-\varphi \right) $ and $N_{\sigma }=\chi _{\sigma }\sigma +\chi _{\varphi
}\left( 1-\varphi \right) $.\newline
Furthermore, if (P2) is satisfied instead of (P1) , it follows that%
\begin{equation}
\varphi _{t}\text{, }\sigma _{t}\in L^{2}\left( 0,T;\left( H^{1}\left(
\Omega \right) \right) ^{\ast }\right)  \tag{3.20}
\end{equation}%
and in (3.19) we have "$=$" instead of "$\leq $". Moreover,in this case the
weak formulation (3.12), (3.13) is sat\i sf\i ed also for all $\eta ,\xi \in
H^{1}\left( \Omega \right) $.
\end{theorem}

\begin{theorem}
Let the conditions (M2), ($\Psi $) and (P2) hold. Then, for every initial
data $\left( \varphi _{0},\sigma _{0}\right) \in H^{1}\left( \Omega \right)
\times L^{2}\left( \Omega \right) $ and for every $T>0$, the weak solution
of problem (1.1)-(1.5) specified by Theorem 3.1 is unique. Moreover, if $%
\left( \varphi _{i},\sigma _{i}\right) $, $i=1,2$ are weak solutions of the
problem (1.1)-(1.5) with initial data $\left( \varphi _{0i},\sigma
_{0i}\right) \in H^{1}\left( \Omega \right) \times L^{2}\left( \Omega
\right) $,respectively, then%
\begin{equation*}
\left\Vert \varphi _{2}\left( t\right) -\varphi _{1}\left( t\right)
\right\Vert _{\left( H^{1}\left( \Omega \right) \right) ^{\ast }}+\left\Vert
\sigma _{2}\left( t\right) -\sigma _{1}\left( t\right) \right\Vert _{\left(
H^{1}\left( \Omega \right) \right) ^{\ast }}+\left\Vert \varphi _{2}\left(
t\right) -\varphi _{1}\left( t\right) \right\Vert _{L^{2}\left(
0,t;H^{1}\left( \Omega \right) \right) }
\end{equation*}%
\begin{equation*}
+\left\Vert \sigma _{2}\left( t\right) -\sigma _{1}\left( t\right)
\right\Vert _{L^{2}\left( 0,t;L^{2}\left( \Omega \right) \right) }\leq
\Lambda \left( t\right) \left( \left\Vert \varphi _{02}-\varphi
_{01}\right\Vert _{\left( H^{1}\left( \Omega \right) \right) ^{\ast
}}+\left\Vert \sigma _{02}-\sigma _{01}\right\Vert _{\left( H^{1}\left(
\Omega \right) \right) ^{\ast }}\right)
\end{equation*}%
where $\Lambda $ is a continuous positive \ function which depends on the
norms of the initial data and $\Psi $, $p$, $\Omega $ and $T$.
\end{theorem}

\begin{theorem}
Suppose (M2), ($\Psi $) and (P2) hold. Then, for every initial data $\left(
\varphi _{0},\sigma _{0}\right) \in H^{3}\left( \Omega \right) \times
H^{1}\left( \Omega \right) $ with $\partial _{\nu }\varphi _{0}=0$ on $%
\Gamma $ and for every $T>0$, problem (1.1)-(1.5) has a strong solution.
\end{theorem}

\section{Proof of the well-posedness}

In this section we will prove well-posedness results stated in the third
section.

To start with, we state the following lemma which is used in the proof of
Theorem 3.1 (for the proof see \cite[Lemma 2]{frigieri}).

\begin{lemma}
Assume that $\Psi $ satisfies ($\Psi $) with $\rho >2$. Then, there exist a
sequence of $\Psi _{m}\in C^{2}\left( 
%TCIMACRO{\U{211d} }%
%BeginExpansion
\mathbb{R}
%EndExpansion
\right) $ satisfying%
\begin{equation*}
\left\vert \Psi _{m}\left( s\right) \right\vert \leq \alpha _{m}\left(
1+s^{2}\right) \text{ \ }\forall s\in 
%TCIMACRO{\U{211d} }%
%BeginExpansion
\mathbb{R}
%EndExpansion
,
\end{equation*}%
for some constant $\alpha _{m}\geq 0$ such that $\Psi _{m}\left( s\right)
\rightarrow \Psi \left( s\right) $ pointwise for all $s\in 
%TCIMACRO{\U{211d} }%
%BeginExpansion
\mathbb{R}
%EndExpansion
$ as $m\rightarrow \infty $ and fulfilling, for every $m\in 
%TCIMACRO{\U{2115} }%
%BeginExpansion
\mathbb{N}
%EndExpansion
$, the bounds%
\begin{equation*}
\left\vert \Psi _{m}\left( s\right) \right\vert \leq k_{0}\left\vert \Psi
\left( s\right) \right\vert \text{, \ }\left\vert \Psi _{m}^{\prime }\left(
s\right) \right\vert \leq k_{1}\left\vert \Psi ^{\prime }\left( s\right)
\right\vert \text{, \ }\left\vert \Psi _{m}^{\prime \prime }\left( s\right)
\right\vert \leq k_{2}\left\vert \Psi ^{\prime \prime }\left( s\right)
\right\vert \text{ \ }\forall s\in 
%TCIMACRO{\U{211d} }%
%BeginExpansion
\mathbb{R}
%EndExpansion
,
\end{equation*}%
and the equi-coercivity conditions%
\begin{equation*}
\Psi _{m}\left( s\right) \geq R_{1}s^{2}-k_{4}\text{, }\Psi _{m}^{\prime
\prime }\left( s\right) \geq -k_{5}\text{ \ }\forall s\in 
%TCIMACRO{\U{211d} }%
%BeginExpansion
\mathbb{R}
%EndExpansion
,
\end{equation*}%
where $k_{i}$, $i=0,...,5$ are positive constants depending on $\Psi $ and $%
\rho $ only.
\end{lemma}

\subsection{Existence of weak Solutions}

\begin{proof}[Proof of Theorem 3.1]
The proof follows the line of the proof of Theorem 1 in \cite{frigieri} and we skip some details and highlight the changes arising due to the fact that we include chemotaxis. First of all we will prove the theorem for $\rho\leq 4$, and then we extend the
	result to $\rho \in \left( 4,6\right) $ with the help of Lemma 4.1.\\
	\textbf{Step 1: (case }$\rho \leq 4$\textbf{)}\\
	
	\textbf{Approximating problem:} We prove the theorem by Galerkin approximation. Let $\left\{ w_{j}\right\}
_{j\geq 1}$ be the eigenfunctions of $A$. We take this sequence as a
Galerkin basis in $H^{1}\left( \Omega \right) $. Define $P_{n}$ as the
orthogonal projector in $L^{2}\left( \Omega \right) $ onto the $n-$%
dimensional subspace $W_{n}:=\left\langle w_{1},...,w_{n}\right\rangle $
spanned by the first $n$ eigenfunctions. We choose $ w_{1} $ to be constant. We consider the following Galerkin
ansatz for (1.1)-(1.5)%
\begin{equation*}
\varphi _{n}\left( t\right) =\sum\limits_{k=1}^{n}a_{k}^{n}\left( t\right)
w_{k}\text{, }\sigma _{n}\left( t\right)
=\sum\limits_{k=1}^{n}b_{k}^{n}\left( t\right) w_{k}\text{, }\mu _{n}\left(
t\right) =\sum\limits_{k=1}^{n}c_{k}^{n}\left( t\right) w_{k}
\end{equation*}%
which solves the following Cauchy problem%
\begin{align}
\left\langle \varphi _{n}^{\prime },w_{j}\right\rangle +\left\langle m\left(
\varphi _{n}\right) \nabla \mu _{n},\nabla w_{j}\right\rangle =\left\langle
p\left( \varphi _{n}\right) \left( \chi _{\sigma }\sigma _{n}+\chi _{\varphi
}\left( 1-\varphi _{n}\right) -\mu _{n}\right) ,w_{j}\right\rangle,\tag{4.1}\\
\left\langle \mu _{n},w_{j}\right\rangle =\left\langle \nabla \varphi
_{n},\nabla w_{j}\right\rangle +\left\langle \Psi ^{\prime }\left( \varphi
_{n}\right) ,w_{j}\right\rangle -\left\langle \chi _{\varphi }\sigma
_{n},w_{j}\right\rangle,\tag{4.2}\\
\left\langle \sigma _{n}^{\prime },w_{j}\right\rangle +\left\langle n\left(
\varphi _{n}\right) \left( \chi _{\sigma }\nabla \sigma _{n}-\chi _{\varphi
}\nabla \varphi _{n}\right) ,\nabla w_{j}\right\rangle =-\left\langle
p(\varphi _{n}) (\chi _{\sigma }\sigma _{n}+\chi _{\varphi}(1-\varphi _{n}) -\mu _{n}),w_{j}\right\rangle,\tag{4.3}\\
\varphi _{n}\left( 0\right) =\varphi _{0n}\text{, }\sigma _{n}\left(
0\right) =\sigma _{0n},\tag{4.4}
\end{align}%
for $j=1,...,n,$ where $\varphi _{0n}=P_{n}\varphi _{0}$ and $\sigma
_{0n}=P_{n}\sigma _{0}$.
\\
This problem (4.1)--(4.4) is an initial value problem for a system of
ordinary differential equations in the unknowns $\left\{
a_{1}^{n},...,a_{n}^{n}\right\} $ and $\left\{
b_{1}^{n},...,b_{n}^{n}\right\} $. Since $\Psi ^{\prime }\in C^{1}$ and $%
p\in C_{loc}^{0,1}$, from the Cauchy-Lipschitz theorem it follows that there
exists $T_{n}^{\ast }\in \left( 0,\infty \right] $ such that this system has
a unique maximal solution $a^{n}:=\left( a_{1}^{n},...,a_{n}^{n}\right) $, $%
b^{n}:=\left( b_{1}^{n},...,b_{n}^{n}\right) $ on $\left[ 0,T_{n}^{\ast
}\right) $ with $a^{n}$, $b^{n}\in C^{1}\left( \left[ 0,T_{n}^{\ast }\right)
;%
%TCIMACRO{\U{211d} }%
%BeginExpansion
\mathbb{R}
%EndExpansion
^{n}\right) $. Thus, the approximate problem (4.1)--(4.4) has a unique
solution $\varphi _{n}$, $\sigma _{n}$, $\mu _{n}\in $ $C^{1}\left( \left[
0,T_{n}^{\ast }\right) ;W_{n}\right) .$\\

\textbf{A priori estimates for the sequence of }$\varphi _{n}$\textbf{\ and }$%
\sigma _{n}$\textbf{:}
We prove the a priori estimates for the approximating solutions. As a result
of these estimates, we deduce that $T_{n}^{\ast }=\infty $ for every $n\in 
%TCIMACRO{\U{2115} }%
%BeginExpansion
\mathbb{N}
%EndExpansion
$. Now, multiplying (4.1) by $c_{j}^{n}$, (4.2) by $\left( a_{j}^{n}\right)^{\prime} $
and (4.3) by $\chi _{\sigma }b_{j}^{n}+\chi _{\varphi }\left(
1-a_{j}^{n}\right) $ and summing the obtained equalities over $j=1,...,n$,
we have%
\begin{align}
&\frac{d}{dt}\left( \frac{1}{2}\left\Vert \nabla \varphi _{n}\right\Vert
_{L^{2}\left( \Omega \right) }^{2}+\int\limits_{\Omega } \Psi \left( \varphi _{n}\right)dx +%
\frac{\chi _{\sigma }}{2}\left\Vert \sigma _{n}\right\Vert _{L^{2}\left(
	\Omega \right) }^{2}+\int\limits_{\Omega } \chi _{\varphi }\sigma _{n}\left( 1-\varphi
_{n}\right) dx\right) \nonumber\\
&+\int\limits_{\Omega }\left( m\left( \varphi \right)
\left\vert \nabla \mu \right\vert ^{2}+n\left( \varphi \right) \left\vert
\nabla N_{\sigma }\right\vert ^{2}+p\left( \varphi \right) \left( N_{\sigma
}-\mu \right) ^{2}\right) dx =0.  \tag{4.5}
\end{align}
Integrating (4.5) on $\left( 0,t\right) $, we obtain%
\begin{align}
&E\left( \varphi _{n}\left( t\right) ,\sigma _{n}\left( t\right) \right)
+\int\limits_{0}^{t}\int\limits_{\Omega }\left( m\left( \varphi \right)
\left\vert \nabla \mu \right\vert ^{2}+n\left( \varphi \right) \left\vert
\nabla N_{\sigma }\right\vert ^{2}+p\left( \varphi \right) \left( N_{\sigma
}-\mu \right) ^{2}\right) dxd\tau\nonumber\\
 &=E\left( \varphi _{0n},\sigma _{0n}\right)
.  \tag{4.6}
\end{align}
On the other hand, by (3.5), it follows that%
\begin{align*}
\left\vert \int\limits_{\Omega }\chi _{\varphi }\sigma _{n}\left( 1-\varphi
_{n}\right) \right\vert &\leq \left\vert \Omega \right\vert ^{1/2}\chi
_{\varphi }\left\Vert \sigma _{n}\right\Vert _{L^{2}\left( \Omega \right)
}+\chi _{\varphi }\left\Vert \sigma _{n}\right\Vert _{L^{2}\left( \Omega
\right) }\left\Vert \varphi _{n}\right\Vert _{L^{2}\left( \Omega \right) }\\
&\leq \frac{2\chi _{\varphi }^{2}}{\chi _{\sigma }}\left\vert \Omega
\right\vert +\frac{\chi _{\sigma }}{8}\left\Vert \sigma _{n}\right\Vert ^{2}+%
\frac{\chi _{\sigma }}{8}\left\Vert \sigma _{n}\right\Vert ^{2}+\frac{2\chi
	_{\varphi }^{2}}{\chi _{\sigma }}\left\Vert \varphi _{n}\right\Vert ^{2}\\
&\leq \frac{2\chi _{\varphi }^{2}}{\chi _{\sigma }}\left\vert \Omega
\right\vert +\frac{\chi _{\sigma }}{4}\left\Vert \sigma _{n}\right\Vert ^{2}+%
\frac{R_{2}}{R_{1}}\frac{2\chi _{\varphi }^{2}}{\chi _{\sigma }}+\frac{1}{%
	R_{1}}\frac{2\chi _{\varphi }^{2}}{\chi _{\sigma }}\int\limits_{\Omega } \Psi \left( \varphi
_{n}\right).
\end{align*}%
Then,since $R_{1}>\frac{2\chi _{\varphi }^{2}}{\chi _{\sigma }}$,
considering the previous estimate in (4.6), we deduce%
\begin{align}
&\left\Vert \nabla \varphi _{n}\left( t\right) \right\Vert
^{2}+\int\limits_{\Omega }\Psi \left( \varphi _{n}\left( t\right) \right)dx
+\left\Vert \sigma _{n}\left( t\right) \right\Vert
^{2}\nonumber\\
&+\int\limits_{0}^{t}\left( \left\Vert \nabla \mu _{n}\right\Vert
^{2}+\left\Vert \nabla N_{\sigma _{n}}\right\Vert ^{2}\right) d\tau
+\int\limits_{0}^{t}\int\limits_{\Omega }p\left( \varphi _{n}\right)
\left( N_{\sigma _{n}}-\mu _{n}\right) ^{2}dxd\tau\leq C . \tag{4.7}
\end{align}
Together with the assumptions of the theorem, this yields%
\begin{align}
\left\Vert \varphi _{n}\right\Vert _{L^{\infty }\left( 0,T;H^{1}\left(\Omega \right) \right) }&\leq C\text{, }&\left\Vert \sigma _{n}\right\Vert_{L^{\infty }\left( 0,T;L^{2}\left( \Omega \right) \right) \cap L^{2}\left(0,T;H^{1}\left( \Omega \right) \right) }\leq C, \tag{4.8}\\
\left\Vert \nabla \mu _{n}\right\Vert _{L^{2}\left( 0,T;L^{2}\left( \Omega\right) \right) }&\leq C\text{, }&\left\Vert N_{\sigma _{n}}\right\Vert_{L^{2}\left( 0,T;H^{1}\left( \Omega \right) \right) }\leq C,\tag{4.9}\\
\left\Vert \Psi \left( \varphi _{n}\right) \right\Vert _{L^{\infty }\left(	0,T;L^{1}\left( \Omega \right) \right) }&\leq C, &\left\Vert \sqrt{p\left( \varphi _{n}\right) }\left( N_{\sigma _{n}}-\mu _{n}\right)\right\Vert _{L^{2}\left( 0,T;L^{2}\left( \Omega \right) \right) }\leq C,\tag{4.10}
\end{align}%
where $C=C\left( \left\Vert \varphi _{0}\right\Vert ,\left\Vert \sigma
_{0}\right\Vert \right) $ is a nonnegative constant depending on the norms
of the initial data (and on $\Psi $ and $\Omega )$. \newline
On the other hand, recalling the assumptions on $\Psi $, from (4.2) it follows taking $ j=1 $ that%
\begin{equation*}
\left| \int\limits_{\Omega }\mu _{n}\right| =\left| \int\limits_{\Omega }\Psi ^{\prime }\left(
\varphi _{n}\right) -\chi _{\varphi }\int\limits_{\Omega }\sigma _{n}\right| \leq
c_{1}+c_{2}\int\limits_{\Omega }\Psi \left( \varphi _{n}\right)
+c_{3}\left\Vert \sigma _{n}\right\Vert _{L^{2}\left( \Omega \right) },
\end{equation*}%
where $c_{1}$, $c_{2}$ are two nonnegative constants depending only on $\Psi 
$, $\Omega $. Taking into account (4.8)$_{1}$ and (4.10) in the above
equality, we infer 
\begin{equation}
\left\vert \int\limits_{\Omega }\mu _{n}\right\vert \leq C.  \tag{4.11}
\end{equation}%
Hence, exploiting Poincare-Wirtinger inequality from (4.9)$_{1}$ and (4.11),
we obtain%
\begin{equation}
\left\Vert \mu _{n}\right\Vert _{L^{2}\left( 0,T;H^{1}\left( \Omega \right)
\right) }\leq C.  \tag{4.12}
\end{equation}%
Now, we will establish that $\varphi _{n}$ are
also uniformly bounded in $L^{2}\left( 0,T;H^{3}\left( \Omega \right)
\right) $. By using the projection onto $W_{n}$, (4.2) can be written as%
\begin{equation}
\mu _{n}=-\Delta \varphi _{n}+P_{n}\Psi ^{\prime }\left( \varphi _{n}\right)
-\chi _{\varphi }\sigma _{n}.  \tag{4.13}
\end{equation}%
Since $H^{1}\left( \Omega \right) \hookrightarrow L^{6}\left( \Omega \right) 
$, (4.8)$_{1}$ yields that the sequence of $\varphi _{n}$ is bounded in $%
L^{\infty }\left( 0,T;L^{6}\left( \Omega \right) \right) $ which gives us%
\begin{equation}
\left\Vert \Psi ^{\prime }\left( \varphi _{n}\right) \right\Vert _{L^{\infty
}\left( 0,T;L^{2}\left( \Omega \right) \right) }\leq C.  \tag{4.14}
\end{equation}%
Hence, recalling that $\left\Vert P_{n}\Psi ^{\prime }\left( \varphi
_{n}\right) \right\Vert _{L^{2}\left( \Omega \right) }\leq \left\Vert \Psi
^{\prime }\left( \varphi _{n}\right) \right\Vert _{L^{2}\left( \Omega
\right) }$ and considering (4.8)$_{2}$, (4.12) and (4.14) in (4.13), we get%
\begin{equation*}
\left\Vert \Delta  \varphi _{n} \right\Vert _{L^{2}\left(
0,T;L^{2}\left( \Omega \right) \right) }\leq C,
\end{equation*}%
which together with (4.8)$_{1}$ and elliptic regularity theory yields%
\begin{equation}
\left\Vert \varphi _{n}\right\Vert _{L^{\infty }\left( 0,T;H^{1}\left(
\Omega \right) \right) \cap L^{2}\left( 0,T;H^{2}\left( \Omega \right)
\right) }\leq C.  \tag{4.15}
\end{equation}%
Using the Gagliardo-Nirenberg inequality, from (4.15) it follows that 
\begin{equation*}
\left\{ 
\begin{array}{l}
L^{\infty }\left( 0,T;H^{1}\left( \Omega \right) \right) \cap L^{2}\left(
0,T;H^{2}\left( \Omega \right) \right) \hookrightarrow L^{10}\left( Q\right)
, \\ 
L^{\infty }\left( 0,T;L^{2}\left( \Omega \right) \right) \cap L^{2}\left(
0,T;H^{1}\left( \Omega \right) \right) \hookrightarrow L^{10/3}\left(
Q\right) ,%
\end{array}%
\right.
\end{equation*}%
which yields%
\begin{equation}
\left\{ 
\begin{array}{r}
\left\Vert \varphi _{n}\right\Vert _{L^{10}\left( Q\right) }\leq C, \\ 
\left\Vert \nabla \varphi _{n}\right\Vert _{L^{10/3}\left( Q\right) }\leq C,%
\end{array}%
\right.  \tag{4.16}
\end{equation}%
where $Q=\left( 0,T\right)\times\Omega$. On the other hand, since $%
\left( left\Vert \nabla u\right\Vert _{L^{2}\left( \Omega \right) }\leq \left\Vert
A^{1/2}u\right\Vert _{L^{2}\left( \Omega \right) }$ for all $u\in
H^{1}\left( \Omega \right) $, we have%
\begin{eqnarray*}
\left\Vert \nabla \left( P_{n}\Psi ^{\prime }\left( \varphi _{n}\right)
\right) \right\Vert _{L^{2}\left( \Omega \right) } &\leq &\left\Vert
A^{1/2}P_{n}\Psi ^{\prime }\left( \varphi _{n}\right) \right\Vert
_{L^{2}\left( \Omega \right) } \\
&=&\left\Vert P_{n}A^{1/2}\Psi ^{\prime }\left( \varphi _{n}\right)
\right\Vert _{L^{2}\left( \Omega \right) } \\
&\leq &\left\Vert \nabla \Psi ^{\prime }\left( \varphi _{n}\right)
\right\Vert _{L^{2}\left( \Omega \right) }+\left\Vert \Psi ^{\prime }\left(
\varphi _{n}\right) \right\Vert _{L^{2}\left( \Omega \right) }.
\end{eqnarray*}%
With the help of above estimate,from (3.2),(4.14) and (4.16), we infer that%
\begin{eqnarray*}
\left\Vert P_{n}\Psi ^{\prime }\left( \varphi _{n}\right) \right\Vert
_{L^{2}\left( 0,T;H^{1}\left( \Omega \right) \right) } &\leq &\left\Vert
\Psi ^{\prime \prime }\left( \varphi _{n}\right) \nabla \varphi
_{n}\right\Vert _{L^{2}\left( Q\right) }+2\left\Vert \Psi ^{\prime }\left(
\varphi _{n}\right) \right\Vert _{L^{2}\left( Q\right) } \\
&\leq &\left\Vert \Psi ^{\prime \prime }\left( \varphi _{n}\right)
\right\Vert _{L^{5}\left( Q\right) }\left\Vert \nabla \varphi
_{n}\right\Vert _{L^{10/3}\left( Q\right) }+2\left\Vert \Psi ^{\prime
}\left( \varphi _{n}\right) \right\Vert _{L^{2}\left( Q\right) } \\
&\leq &c_{3}\left( 1+\left\Vert \varphi _{n}\right\Vert _{L^{10}\left(
Q\right) }\right) \left\Vert \nabla \varphi _{n}\right\Vert _{L^{10/3}\left(
Q\right) }+2\left\Vert \Psi ^{\prime }\left( \varphi _{n}\right) \right\Vert
_{L^{2}\left( Q\right) } \\
&\leq &C.
\end{eqnarray*}%
Considering above estimate with (4.8)$_{2}$ and (4.12) in (4.13), we obtain%
\begin{equation}
\left\Vert \varphi _{n}\right\Vert _{L^{2}\left( 0,T;H^{3}\left( \Omega
\right) \right) }\leq C.  \tag{4.17}
\end{equation}%

\textbf{A priori estimates for the sequence of time derivatives }$\varphi
_{n}^{\prime }$\textbf{\ and }$\sigma _{n}^{\prime }$\textbf{:} Let $\chi \in D\left( A\right) \hookrightarrow L^{\infty }\left( \Omega
\right) $ then it can be decomposed as $\chi =\chi _{1}+\chi _{2}$, where $%
\chi _{1}=P_{n}\chi \in W_{n}$ and $\chi _{2}\in \left( I-P_{n}\right) \chi
\in W_{n}^{\bot }$ \ (recall that $\chi _{1}$ and $\chi _{2}$ are orthogonal
in $L^{2}\left( \Omega \right) $, $H^{1}\left( \Omega \right) $ and $D\left(
A\right) $). Then, since $\chi _{1}$ and $\chi _{2}$ are orthogonal in $%
L^{2}\left( \Omega \right) $ we obtain from (4.1) and (4.3) that%
\begin{equation}
\left\{ 
\begin{array}{c}
\left\langle \left\langle \varphi _{n}^{\prime },\chi \right\rangle
\right\rangle =\left\langle \left\langle \varphi _{n}^{\prime },\chi
_{1}\right\rangle \right\rangle =-\left\langle m\left( \varphi _{n}\right)
\nabla \mu _{n},\nabla \chi _{1}\right\rangle +\left\langle p\left( \varphi
_{n}\right) \left( N_{\sigma _{n}}-\mu _{n}\right) ,\chi _{1}\right\rangle ,
\\ 
\left\langle \left\langle \sigma _{n}^{\prime },\chi \right\rangle
\right\rangle =\left\langle \left\langle \sigma _{n}^{\prime },\chi
_{1}\right\rangle \right\rangle =-\left\langle n\left( \varphi _{n}\right)
\nabla N_{\sigma _{n}},\nabla \chi _{1}\right\rangle -\left\langle p\left(
\varphi _{n}\right) \left( N_{\sigma _{n}}-\mu _{n}\right) ,\chi
_{1}\right\rangle. 
\end{array}%
\right.   \tag{4.18}
\end{equation}%
For the last terms in the above identities, recalling $D\left( A\right)
\hookrightarrow L^{\infty }\left( \Omega \right) $, we deduce that%
\begin{equation}
\left\vert \left\langle p\left( \varphi _{n}\right) \left( N_{\sigma
_{n}}-\mu _{n}\right) ,\chi _{1}\right\rangle \right\vert \leq \left\Vert
p\left( \varphi _{n}\right) \right\Vert _{L^{6/5}\left( \Omega \right)
}\left\Vert N_{\sigma _{n}}-\mu _{n}\right\Vert _{L^{6}\left( \Omega \right)
}\left\Vert \chi _{1}\right\Vert _{D\left( A\right) }.  \tag{4.19}
\end{equation}%
From (4.9)$_{2}$ and (4.12), the sequence of $\left( N_{\sigma _{n}}-\mu
_{n}\right) $ is bounded in $L^{2}\left( 0,T;L^{6}\left( \Omega \right)
\right) $. Thus, to obtain the boundedness of the sequences of $\varphi
_{n}^{\prime }$ and $\sigma _{n}^{\prime }$ in $L^{r}\left( 0,T;D\left(
A^{-1}\right) \right) $ with some $r>1$, boundedness of $p\left( \varphi
_{n}\right) $ in $L^{\theta }\left( 0,T;L^{6/5}\left( \Omega \right) \right) 
$ is required for some $\theta >2$. As in \cite{frigieri}, we can use embedding theorems and Hölder as well as Gagliardo-Nirenberg inequalities to show that this is possible. Hence, another application of Hölder's inequality gives that (4.19) implies
\begin{equation}
\int\limits_{0}^{T}\left\vert \left\langle p\left( \varphi _{n}\right)
\left( N_{\sigma _{n}}-\mu _{n}\right) ,\chi _{1}\right\rangle \right\vert
dt\leq c_{5}\left\Vert \chi _{1}\right\Vert _{L^{\frac{2\theta }{\theta -2}%
}\left( 0,T;D\left( A\right) \right) }\text{ for some }\theta >2. 
\tag{4.20}
\end{equation}%
Moreover, one can easily deduce%
\begin{equation}
\left\{ 
\begin{array}{r}
\int\limits_{0}^{T}\left\vert \left\langle m\left( \varphi _{n}\right)
\nabla \mu _{n},\nabla \chi _{1}\right\rangle \right\vert dt\leq
c_{6}\left\Vert \chi _{1}\right\Vert _{L^{2}\left( 0,T;D\left( A\right)
\right) ,} \\ 
\int\limits_{0}^{T}\left\vert \left\langle n\left( \varphi _{n}\right)
\nabla N_{\sigma _{n}},\nabla \chi _{1}\right\rangle \right\vert dt\leq
c_{6}\left\Vert \chi _{1}\right\Vert _{L^{2}\left( 0,T;D\left( A\right)
\right) .}%
\end{array}%
\right.   \tag{4.21}
\end{equation}%
Taking into account (4.20) and (4.21) in (4.18), we infer the following
controls%
\begin{equation}
\left\Vert \varphi _{n}^{\prime }\right\Vert _{L^{r}\left( 0,T;D\left(
A^{-1}\right) \right) }\leq C\text{, \ \ }\left\Vert \sigma _{n}^{\prime
}\right\Vert _{L^{r}\left( 0,T;D\left( A^{-1}\right) \right) }\leq C\text{
for }r=\frac{2\theta }{\theta +2},  \tag{4.22}
\end{equation}
where $r>1$ (notice that $\theta >2$).\\

\textbf{Passing to limit:} From the uniform bounds (4.8), (4.12), (4.17) and (4.22) we obtain that the
sequences of $\varphi _{n}$, $\sigma _{n}$, $\mu _{n}$, $\varphi
_{n}^{\prime }$ and $\sigma _{n}^{\prime }$ admit weakly convergent
subsequences with weak limits by $\varphi $, $\sigma $, $\mu $, $\varphi
^{\prime }$ and $\sigma ^{\prime }$, respectively, such that%
\begin{equation*}
\left\{ 
\begin{array}{l}
\varphi \in L^{\infty }\left( 0,T;H^{1}\left( \Omega \right) \right) \cap
L^{2}\left( 0,T;H^{3}\left( \Omega \right) \right) , \\ 
\sigma \in L^{\infty }\left( 0,T;L^{2}\left( \Omega \right) \right) \cap
L^{2}\left( 0,T;H^{1}\left( \Omega \right) \right) , \\ 
\mu \in L^{2}\left( 0,T;H^{1}\left( \Omega \right) \right) , \\ 
\varphi _{t},\sigma _{t}\in L^{r}\left( 0,T;D\left( A^{-1}\right) \right) .%
\end{array}%
\right.
\end{equation*}%
From Aubin-Lions lemma, we have the compact embedding%
\begin{equation*}
L^{\infty }\left( 0,T;H^{1}\left( \Omega \right) \right) \cap W^{1,r}\left(
0,T;D\left( A^{-1}\right) \right) \hookrightarrow C\left( \left[ 0,T\right]
;L^{\kappa }\left( \Omega \right) \right) \text{, \ }2\leq \kappa <6
\end{equation*}
which yields, up to a subsequence, 
\begin{equation*}
\varphi _{n}\rightarrow \varphi \text{ \ \ a.e. in \ }Q\text{ }.
\end{equation*}%
Then, we have%
\begin{equation}
p\left( \varphi _{n}\right) \rightarrow p\left( \varphi \right) \text{ \ \
a.e. in \ }Q\text{ }.  \tag{4.23}
\end{equation}
On the other hand, as $ p(\varphi_{n}) $ is for suitable $ \theta>2, \ \epsilon>0 $ uniformly bounded in $ L^{\theta }\left( 0,T;L^{6/5+\varepsilon }\left( \Omega \right)
\right) $ (cf. \cite{frigieri}), we obtain that%
\begin{equation*}
p\left( \varphi _{n}\right) \rightarrow p\left( \varphi \right) \text{
weakly in }L^{\theta }\left( 0,T;L^{6/5+\varepsilon }\left( \Omega \right)
\right) .
\end{equation*}%
This together with (4.23) gives that up to a subsequence%
\begin{equation}
p\left( \varphi _{n}\right) \rightarrow p\left( \varphi \right) \text{ 
strongly in }L^{2}\left( 0,T;L^{6/5}\left( \Omega \right) \right) \text{.} 
\tag{4.24}
\end{equation}%
Then, since%
\begin{equation*}
\left( N_{\sigma _{n}}-\mu _{n}\right) \rightarrow \left( N_{\sigma }-\mu
\right) \text{ weakly in }L^{2}\left( 0,T;L^{6}\left( \Omega \right) \right) 
\text{,}
\end{equation*}
we can now pass to the limit in the term $\left( p\left( \varphi _{n}\right)
\left( N_{\sigma _{n}}-\mu _{n}\right) ,w_{j}\right) $ with the help of
(4.24). Consequently, thanks to all of the convergences established above we
can pass to the limit in the approximate problem (4.1)-(4.4) and deduce that 
$\varphi $, $\sigma $, $\mu $ satisfy (3.12)-(3.13) together with (3.10).\\

\textbf{Energy inequality:} The energy inequality can be obtained by passing to limit in (4.6). The right hand side is bounded due to the fact that $ \left\| \varphi_{0n} \right\|_{L^{6}(\Omega)}\leq C \left\| \varphi_{0n} \right\|_{H^{1}(\Omega)}\leq C \left\| \varphi_{0} \right\|_{H^{1}(\Omega)} $ which follows from the fact that the eigenfunctions $ {\left\lbrace w_{j}\right\rbrace }_{j\geq1} $ are orthogonal in $L^{2}(\Omega) $ and $ H^{1}(\Omega) $. In addition, we need to show that%
\begin{equation*}
\int\limits_{0}^{t}\int\limits_{\Omega }p\left( \varphi \right) \left(
N_{\sigma }-\mu \right) ^{2}\leq \underset{n\rightarrow \infty }{\lim \inf }%
\int\limits_{0}^{t}\int\limits_{\Omega }p\left( \varphi _{n}\right) \left(
N_{\sigma _{n}}-\mu _{n}\right) ^{2}.
\end{equation*}%
However, here we can argue as in \cite{frigieri}. \\

\textbf{Further regularity with respect to time in the case }$q\leq 4:$ If $q\leq 4$, since $\varphi \in L^{\infty }\left( 0,\infty ;L^{6}\left(
\Omega \right) \right) $, from assumption (P) it follows that $\sqrt{p\left(
\varphi \right) }\in L^{\infty }\left( 0,\infty ;L^{3}\left( \Omega \right)
\right) $. Then, we can estimate the term $p\left( \varphi \right) \left(
N_{\sigma }-\mu \right) $ in $H^{-1}\left( \Omega \right) $ as follows%
\begin{equation}
\left\Vert p\left( \varphi \right) \left( N_{\sigma }-\mu \right)
\right\Vert _{H^{-1}\left( \Omega \right) }\leq c\left\Vert \sqrt{p\left(
\varphi \right) }\right\Vert _{L^{3}\left( \Omega \right) }\left\Vert \sqrt{%
p\left( \varphi \right) }\left( N_{\sigma }-\mu \right) \right\Vert
_{L^{2}\left( \Omega \right) }.  \tag{4.25}
\end{equation}%
Moreover, we can easily obtain that%
\begin{equation}
\left\Vert \Delta \mu \right\Vert _{H^{-1}\left( \Omega \right) }\leq
\left\Vert \nabla \mu \right\Vert _{L^{2}\left( \Omega \right) }\text{ and }%
\left\Vert \Delta N_{\sigma }\right\Vert _{H^{-1}\left( \Omega \right) }\leq
\left\Vert \nabla N_{\sigma }\right\Vert _{L^{2}\left( \Omega \right) } .
\tag{4.26}
\end{equation}%
Therefore, taking the energy inequality (3.19) in (4.25) and (4.26) into account we have%
\begin{equation*}
\left\{ 
\begin{array}{r}
p\left( \varphi \right) \left( N_{\sigma }-\mu \right) \in L^{2}\left(
0,\infty ;H^{-1}\left( \Omega \right) \right) , \\ 
\Delta \mu \in L^{2}\left( 0,\infty ;H^{-1}\left( \Omega \right) \right) ,
\\ 
\Delta N_{\sigma }\in L^{2}\left( 0,\infty ;H^{-1}\left( \Omega \right)
\right) .%
\end{array}%
\right. 
\end{equation*}%
Considering the above arguments in the variational formulation (3.12)-(3.13), we
deduce 
\begin{equation}
\left\{ 
\begin{array}{r}
\varphi _{t}\in L^{2}\left( 0,\infty ;H^{-1}\left( \Omega \right) \right) ,
\\ 
\sigma _{t}\in L^{2}\left( 0,\infty ;H^{-1}\left( \Omega \right) \right) .%
\end{array}%
\right.  \tag{4.27}
\end{equation}\\

\textbf{Energy identity in the case }$q\leq 4:$ If we choose $\eta =\mu $ and $\xi =N_{\sigma }$ as test functions in the
variational formulation (3.12)-(3.13), we can infer (for the details see \cite{frigieri})%
\begin{equation*}
\frac{d}{dt}\left( \frac{1}{2}\left\Vert \nabla \varphi \right\Vert
_{L^{2}\left( \Omega \right) }^{2}+\frac{\chi _{\sigma }}{2}\left\Vert
\sigma \right\Vert _{L^{2}\left( \Omega \right) }^{2}+\int\limits_{\Omega
}\Psi \left( \varphi \right) dx+\chi _{\varphi }\int\limits_{\Omega }\sigma
\left( 1-\varphi \right) dx\right)
\end{equation*}%
\begin{equation*}
+\int\limits_{\Omega }m\left( \varphi \right) \left\vert \nabla \mu
\right\vert ^{2}dx+\int\limits_{\Omega }n\left( \varphi \right) \left\vert
\nabla N_{\sigma }\right\vert ^{2}dx+\int\limits_{\Omega }p\left( \varphi
\right) \left( N_{\sigma }-\mu \right) ^{2}dx=0.
\end{equation*}%
Hence integrating the above identity with respect to $t$, we obtain (3.19)
with an equality sign.\\
\\
\textbf{Step 2: (case }$4<\rho <6$\textbf{)}\\

In this case, potential $\Psi $ is approximated by a sequence of potentials $%
\Psi _{m}$ constructed in Lemma 4.1. Now, let us consider the system
(1.1)-(1.5) with $\Psi _{m}$ instead of $\Psi $ and called the new system $%
P_{m}.$ Since $\Psi _{m}$ satisfies the assumption ($\Psi $), \ from Step 1
we deuce that there exist a weak solution $\left( \varphi _{m},\sigma
_{m}\right) $ of the system $P_{m}$ for each $m\in 
%TCIMACRO{\U{2115} }%
%BeginExpansion
\mathbb{N}
%EndExpansion
$. Moreover, this solution satisfied the following regularities%
\begin{equation*}
\left\{ 
\begin{array}{l}
\varphi _{m}\in L^{\infty }\left( 0,T;H^{1}\left( \Omega \right) \right)
\cap L^{2}\left( 0,T;H^{3}\left( \Omega \right) \right) , \\ 
\sigma _{m}\in L^{\infty }\left( 0,T;L^{2}\left( \Omega \right) \right) \cap
L^{2}\left( 0,T;H^{1}\left( \Omega \right) \right) , \\ 
\mu _{m}\in L^{2}\left( 0,T;H^{1}\left( \Omega \right) \right) , \\ 
\varphi _{mt},\sigma _{mt}\in L^{r}\left( 0,T;D\left( A^{-1}\right) \right) .%
\end{array}%
\right.
\end{equation*}%
and the energy inequality%
\begin{align*}
&E\left( \varphi _{m}\left( t\right) ,\sigma _{m}\left( t\right) \right)
+\int\limits_{0}^{t}\int\limits_{\Omega }\left( m\left( \varphi
_{m}\right) \left\vert \nabla \mu _{m}\right\vert ^{2}+n\left( \varphi
_{m}\right) \left\vert \nabla N_{\sigma _{m}}\right\vert ^{2}+p\left(
\varphi _{m}\right) \left( N_{\sigma _{m}}-\mu _{m}\right) ^{2}\right)
dxd\tau 
\\
&\leq E\left( \varphi _{0},\sigma _{0}\right) .
\end{align*}%
Recalling the assumptions ($\Psi $), (P) and the properties of the sequence
of $\Psi _{m}$, we can obtain the following controls%
\begin{equation}
\left\{
\begin{array}{r}
\left\Vert \varphi _{m}\right\Vert _{L^{\infty }\left( 0,T;H^{1}\left(
\Omega \right) \right) }\leq C, \\ 
\left\Vert \sigma _{m}\right\Vert _{L^{\infty }\left( 0,T;L^{2}\left( \Omega
\right) \right) \cap L^{2}\left( 0,T;H^{1}\left( \Omega \right) \right)
}\leq C,\\ 
\left\Vert \mu _{m}\right\Vert _{L^{2}\left( 0,T;H^{1}\left( \Omega \right)
\right) }\leq C.
\end{array} 
\right.  \tag{4.28}
\end{equation}%
On the other hand, if we multiply the identity $\mu _{m}=-\Delta \varphi
_{m}+\Psi _{m}^{\prime }\left( \varphi _{m}\right) -\chi _{\varphi }\sigma
_{m}$ with $\Delta \varphi _{m}$, from the properties of $\Psi _{m}$, we
obtain%
\begin{equation*}
\left\Vert \Delta \varphi _{m}\right\Vert _{L^{2}\left( \Omega \right)
}^{2}\leq \left\Vert \nabla \mu _{m}\right\Vert _{L^{2}\left( \Omega \right)
}^{2}+c\left\Vert \nabla \varphi _{m}\right\Vert _{L^{2}\left( \Omega
\right) }^{2}+c\left\Vert \nabla \sigma _{m}\right\Vert _{L^{2}\left( \Omega
\right) }^{2}
\end{equation*}%
together with (4.28),which yields the control of the sequence $\varphi _{m}$
in $L^{\infty }\left( 0,T;H^{1}\left( \Omega \right) \right) \cap
L^{2}\left( 0,T;H^{2}\left( \Omega \right) \right) $. Furthermore, by using a bootstrapping method (see \cite{frigieri} for details) we deduce the control of
the sequence $\varphi _{m}$ in $L^{\infty }\left( 0,T;H^{1}\left( \Omega
\right) \right) \cap L^{2}\left( 0,T;H^{3}\left( \Omega \right) \right) $.
Consequently, noting that%
\begin{equation*}
\Psi _{m}^{\prime }\left( \varphi _{m}\right) \rightarrow \Psi ^{\prime
}\left( \varphi \right) \text{ almost everywhere in }\Omega ,
\end{equation*}%
as in the first step we can pass to the limit in the problem (P$_{m}$) and
obtain the existence of weak solutions in the case $\rho \in \left(
4,6\right) $. For $q\in \left[ 1.9\right) $, energy inequality (3.19) and the regularity of time derivatives as in (4.22) can be established with similar arguments done in the first step. Moreover, in the case $q\leq 4$ we obtain an equality sign in (3.19) and the further regularity for time derivatives as in (4.27) (cf. \cite{frigieri}).
\end{proof}

Now, we prove the uniqueness of weak solutions and their continuous dependence to the initial data. At this point, we need to
note that, uniqueness is established in the case mobilities $m\left( \varphi
\right) $ and $n\left( \varphi \right) $ are constant. For simplicity, we
assume that they are equal to one.
\subsection{Uniqueness and continuous dependence to the initial data}

\begin{proof}[Proof of Theorem 3.2]
Firstly, observe that we can rewrite variational formulation (3.12)-(3.13) as
follows%
\begin{align}
\left\langle\left\langle \varphi _{t},\chi \right\rangle\right\rangle +\left\langle A \mu
, \chi \right\rangle&=\left\langle p\left( \varphi \right) N_{\sigma }-\left( p\left( \varphi
\right) -1\right) \mu ,\chi \right\rangle  \tag{4.29}\\
\mu &=A\varphi +G^{\prime }\left( \varphi \right) -\chi _{\varphi }\sigma 
\tag{4.30}\\
\left\langle\left\langle \sigma _{t},\xi \right\rangle\right\rangle +\chi _{\sigma }\left\langle
A \sigma , \xi \right\rangle  &=\chi _{\varphi }\left\langle A \varphi ,\xi
\right\rangle -\left\langle p\left( \varphi \right) N_{\sigma }-p\left( \varphi \right)
\mu -\chi _{\sigma }\sigma +\chi _{\varphi }\varphi ,\xi \right\rangle 
\tag{4.31}
\end{align}%
for all $\chi $, $\xi \in H^{1}\left( \Omega \right) $, where $G\left(
s\right) :=\Psi \left( s\right) -\frac{1}{2}s^{2}.$\newline
Let $\left[ \varphi _{i},\sigma _{i}\right] $, $i=1,2$ be two weak solutions
of the problem (1.1)-(1.5) with initial data $\left[ \varphi _{i}\left(
0\right) ,\sigma _{i}\left( 0\right) \right] =\left[ \varphi _{0i},\sigma
_{0i}\right] $, $i=1,2$. Defining $\varphi :=\varphi _{2}-\varphi _{1}$, $%
\sigma :=\sigma _{2}-\sigma _{1}$ and $\mu :=\mu _{2}-\mu _{1}$, from
(4.29)-(4.31), we obtain%
\begin{align}
\left\langle\left\langle \varphi _{t},\chi \right\rangle\right\rangle +\left\langle A \mu
,\chi \right\rangle=\left\langle \left( p\left( \varphi _{2}\right) -p\left( \varphi
_{1}\right) \right) \left( N_{\sigma _{2}}-\mu _{2}\right) +p\left( \varphi
_{1}\right) N_{\sigma }-\left( p\left( \varphi _{1}\right) -1\right) \mu
,\chi \right\rangle  \tag{4.32}\\
\mu =A\varphi +G^{\prime }\left( \varphi _{2}\right) -G^{\prime }\left(
\varphi _{1}\right) -\chi _{\varphi }\sigma  \tag{4.33}\\
\left\langle\left\langle \sigma _{t},\xi \right\rangle \right\rangle +\chi _{\sigma }\left\langle A \sigma , \xi \right\rangle  =\chi _{\varphi }\left\langle A \varphi , \xi
\right\rangle\nonumber \\
-\left\langle \left( p\left( \varphi _{2}\right) -p\left( \varphi
_{1}\right) \right) \left( N_{\sigma _{2}}-\mu _{2}\right) -p\left( \varphi
_{1}\right) \left( N_{\sigma }-\mu \right) -\chi _{\sigma }\sigma +\chi
_{\varphi }\varphi ,\xi \right\rangle  \tag{4.34}
 \end{align}%
for all $\chi $, $\xi \in H^{1}\left( \Omega \right) $. Choosing $\chi
=A^{-1}\varphi $ in (4.32) and $\xi =A^{-1}\sigma $ in (4.34) and summing
the obtained identities, we infer that%
\begin{align*}
&\frac{1}{2}\frac{d}{dt}\left\Vert \varphi \right\Vert
_{\overset{\ast }{H^{1}}\left( \Omega \right)}^{2}+\left\Vert \varphi \right\Vert _{H^{1}}^{2}+\left\langle  G^{\prime
}\left( \varphi _{2}\right) -G^{\prime }\left( \varphi _{1}\right) ,\varphi
\right\rangle  +\frac{1}{2}\frac{d}{dt}\left\Vert \sigma \right\Vert _{\overset{\ast }{H^{1}}\left( \Omega \right)}^{2}+\chi _{\sigma }\left\Vert \sigma \right\Vert _{{L^{2}}\left( \Omega \right)}^{2}\\
&=\int\limits_{\Omega } \chi
_{\varphi }\sigma \varphi dx+\left\langle  \left( p\left( \varphi _{2}\right) -p\left( \varphi _{1}\right)
\right) \left( N_{\sigma _{2}}-\mu _{2}\right) +p\left( \varphi _{1}\right)
N_{\sigma }-\left( p\left( \varphi _{1}\right) -1\right) \mu ,A^{-1}\varphi\right\rangle \\
&-\left\langle  \left( p\left( \varphi _{2}\right) -p\left( \varphi _{1}\right)
\right) \left( N_{\sigma _{2}}-\mu _{2}\right) -p\left( \varphi _{1}\right)
\left( N_{\sigma }-\mu \right) -\chi _{\sigma }\sigma +\chi _{\varphi
}\varphi ,A^{-1}\sigma \right\rangle.  \tag{4.35}
\end{align*}%
Using the same arguments used \cite{frigieri} to obtain [3.65], we get%
\begin{equation}
\left\vert \left\langle  p\left( \varphi _{1}\right) N_{\sigma }-\left( p\left(
\varphi _{1}\right) -1\right) \mu ,A^{-1}\varphi \right\rangle  \right\vert \leq 
\frac{1}{12}\left\Vert \varphi \right\Vert _{H^{1}\left( \Omega \right)}^{2}+\Lambda \alpha
_{1}^{2}\left( t\right) \left( \left\Vert \sigma \right\Vert
_{\overset{\ast }{H^{1}}\left( \Omega \right)}^{2}+\left\Vert \varphi \right\Vert _{\overset{\ast }{H^{1}}\left( \Omega \right)}^{2}\right) , 
\tag{4.36}
\end{equation}%
where%
\begin{equation*}
\alpha _{1}\left( t\right) :=c\left( \left\Vert p\left( \varphi _{1}\left(
t\right) \right) \right\Vert _{L^{3}\left( \Omega \right) }+\left\Vert
\varphi _{1}\left( t\right) \right\Vert _{L^{\infty }\left( \Omega \right)
}^{q}+\left\Vert p^{\prime }\left( \varphi _{1}\left( t\right) \right)
\nabla \varphi _{1}\left( t\right) \right\Vert _{L^{3}\left( \Omega \right)
}+1\right)
\end{equation*}%
and $\Lambda $ denotes a positive constant that depends on the norms of the
initial data of the two solutions. Also, using the same arguments used in \cite{frigieri}
to obtain [3.72], we have%
\begin{equation}
\left\vert \left\langle  \left( p\left( \varphi _{2}\right) -p\left( \varphi
_{1}\right) \right) \left( N_{\sigma _{2}}-\mu _{2}\right) ,A^{-1}\varphi
\right\rangle  \right\vert \leq \frac{1}{12}\left\Vert \varphi \right\Vert
_{H^{1}\left( \Omega \right)}^{2}+\alpha _{2}^{4/3}\left( t\right) \left\Vert \varphi \right\Vert
_{\overset{\ast }{H^{1}}\left( \Omega \right)}^{2},  \tag{4.37}
\end{equation}%
where%
\begin{equation*}
\alpha _{2}\left( t\right) :=c\left( 1+\left\Vert \varphi _{1}\left(
t\right) \right\Vert _{L^{\infty }\left( \Omega \right) }^{q-1}+\left\Vert
\varphi _{2}\left( t\right) \right\Vert _{L^{\infty }\left( \Omega \right)
}^{q-1}\right) \left( \left\Vert N_{\sigma _{2}}\left( t\right) \right\Vert
_{L^{3}\left( \Omega \right) }+\Lambda \left\Vert \mu _{2}\left( t\right)
\right\Vert _{H^{1}\left( \Omega \right)}^{3/4}\right) .
\end{equation*}%
Moreover, using the same arguments used in \cite{frigieri} to obtain [3.73] and [3.74], we
obtain%
\begin{align*}
&\left\vert \left\langle  \left( p\left( \varphi _{2}\right) -p\left( \varphi
_{1}\right) \right) \left( N_{\sigma _{2}}-\mu _{2}\right) +\chi _{\varphi
}\varphi ,A^{-1}\sigma \right\rangle  \right\vert\\
 &\leq \frac{1}{12}\left\Vert
\varphi \right\Vert _{H^{1}\left( \Omega \right)}^{2}+\alpha _{2}^{4/3}\left( t\right) \left(
\left\Vert \varphi \right\Vert _{\overset{\ast }{H^{1}}\left( \Omega \right)}^{2}+\left\Vert \sigma \right\Vert
_{\overset{\ast }{H^{1}}\left( \Omega \right)}^{2}\right) ,  \tag{4.38}
\end{align*}%
and
\begin{align*}
&\left\vert \left\langle  p\left( \varphi _{1}\right) \left( N_{\sigma }-\mu \right)
-\chi _{\sigma }\sigma ,A^{-1}\sigma \right\rangle  \right\vert\\
 &\leq \frac{1}{12}%
\left\Vert \varphi \right\Vert _{H^{1}\left( \Omega \right)}^{2}+\Gamma \left( 1+\alpha
_{1}^{2}\left( t\right) \right) \left( \left\Vert \varphi \right\Vert
_{\overset{\ast }{H^{1}}\left( \Omega \right)}^{2}+\left\Vert \sigma \right\Vert _{\overset{\ast }{H^{1}}\left( \Omega \right)}^{2}\right) . 
\tag{4.39}
\end{align*}%
On the other hand, because of the chemotaxis we have an additional term $%
\int\limits_{\Omega } \chi _{\varphi }\sigma \varphi dx$ on the right hand side of (4.39).
For this term, we have%
\begin{equation}
\left\vert \int\limits_{\Omega } \chi _{\varphi }\sigma \varphi dx\right\vert \leq \chi
_{\varphi }^{2}3\left\Vert \sigma \right\Vert _{\overset{\ast }{H^{1}}\left( \Omega \right)}^{2}+\frac{1}{12}%
\left\Vert \varphi \right\Vert _{H^{1}\left( \Omega \right)}^{2}.  \tag{4.40}
\end{equation}%
Moreover, choosing $\beta :=\alpha +1-c_{1}$, we get
\begin{equation}
\left\langle  G^{\prime }\left( \varphi _{2}\right) -G^{\prime }\left( \varphi
_{1}\right) ,\varphi \right\rangle  \geq -\beta \left\Vert \varphi \right\Vert
_{L^{2}\left( \Omega \right)}^{2}\geq -\frac{1}{12}\left\Vert \varphi \right\Vert
_{H^{1}\left( \Omega \right)}^{2}-c\left\Vert \varphi \right\Vert _{\overset{\ast }{H^{1}}\left( \Omega \right)}^{2}  \tag{4.41}
\end{equation}%
Finally considering the estimates (4.36)-(4.41) in (4.35), we have%
\begin{equation*}
\frac{d}{dt}\left( \left\Vert \varphi \right\Vert _{\overset{\ast }{H^{1}}\left( \Omega \right)}^{2}+\left\Vert
\sigma \right\Vert _{\overset{\ast }{H^{1}}\left( \Omega \right)}^{2}\right) +\left\Vert \varphi \right\Vert
_{H^{1}\left( \Omega \right)}^{2}+\chi _{\sigma }\left\Vert \sigma \right\Vert _{L^{2}\left( \Omega \right)}^{2}\leq
\gamma \left( \left\Vert \varphi \right\Vert _{\overset{\ast }{H^{1}}\left( \Omega \right)}^{2}+\left\Vert \sigma
\right\Vert _{\overset{\ast }{H^{1}}\left( \Omega \right)}^{2}\right)
\end{equation*}%
where $\gamma :=\Lambda \left( \alpha _{1}^{2}+\alpha _{2}^{4/3}+1\right) \in
L^{1}\left( 0,T\right) $. Hence, applying Gronwall's lemma we obtained the
desired result.
\end{proof}

In the next theorem, proof of the existence of strong solutions is stated.

\subsection{Strong solutions}
\begin{proof}[Proof of Theorem 3.3]
We will prove the theorem by obtaining formal estimates which can be justified by Galerkin scheme.\newline
Firstly, testing (1.1) by $\mu _{t}$ in $L^{2}\left( \Omega \right) $ and using
(1.2), we obtain%
\begin{align*}
&\frac{1}{2}\frac{d}{dt}\left\Vert \nabla \mu \right\Vert_{L^{2}\left( \Omega \right)} ^{2}+\left\Vert
\nabla \varphi _{t}\right\Vert _{L^{2}\left( \Omega \right)}^{2}+\int\limits_{\Omega } \Psi ^{\prime \prime }\left(
\varphi \right) \varphi _{t}^{2}dx+\frac{1}{2}\frac{d}{dt}\int\limits_{\Omega } p\left(
\varphi \right) \mu ^{2}dx\\
&=\int\limits_{\Omega } \chi _{\varphi }\sigma _{t}\varphi _{t}dx+\frac{1}{2}\int\limits_{\Omega } p^{\prime
}\left( \varphi \right) \varphi _{t}\mu ^{2}dx+\left\langle  p\left( \varphi \right\rangle 
N_{\sigma },\mu _{t}\right) .  \tag{4.42}
\end{align*}%
Then, testing (1.3) by $N_{\sigma t}$ in $L^{2}\left( \Omega \right) $, we have%
\begin{align*}
&\chi _{\sigma }\left\Vert \sigma _{t}\right\Vert_{L^{2}\left( \Omega \right)} ^{2}+\frac{1}{2}\frac{d}{dt}%
\left\Vert \nabla N_{\sigma }\right\Vert_{L^{2}\left( \Omega \right)} ^{2}+\frac{1}{2}\frac{d}{dt}\int\limits_{\Omega }
p\left( \varphi \right) \left( N_{\sigma }\right) ^{2}dx\\
&=\int\limits_{\Omega } \chi _{\varphi }\sigma _{t}\varphi _{t}dx+\frac{1}{2}\int\limits_{\Omega } p^{\prime
}\left( \varphi \right) \varphi _{t}\left( N_{\sigma }\right) ^{2}dx+\left\langle 
p\left( \varphi \right) \mu ,N_{\sigma t}\right\rangle   \tag{4.43}.
\end{align*}%
Summing (4.42) with (4.43) and applying Young inequality, we get%
\begin{align*}
&\frac{1}{2}\frac{d}{dt}\left( \left\Vert \nabla \mu \right\Vert_{L^{2}\left( \Omega \right)}
^{2}+\left\Vert \nabla N_{\sigma }\right\Vert_{L^{2}\left( \Omega \right)} ^{2}+\int \limits_{\Omega } p\left( \varphi
\right) \left( N_{\sigma }-\mu \right) ^{2}dx\right)\\
&+\left\Vert \nabla \varphi _{t}\right\Vert_{L^{2}\left( \Omega \right)} ^{2}+\dfrac{\chi _{\sigma}}{2}\left\Vert
\sigma _{t}\right\Vert _{L^{2}\left( \Omega \right)}^{2}+\int \limits_{\Omega } \Psi ^{\prime \prime }\left( \varphi
\right) \varphi _{t}^{2}dx\\
&\leq c\left\Vert
\varphi _{t}\right\Vert_{L^{2}\left( \Omega \right)}^{2}+\frac{1}{2}\int \limits_{\Omega } p^{\prime
}\left( \varphi \right) \varphi _{t}\left( N_{\sigma }-\mu \right) ^{2}dx. 
\tag{4.44}
\end{align*}%
Following the same procedure in \cite{frigieri}, from (4.44) it follows that
\begin{align*}
&\frac{1}{2}\frac{d}{dt}\left( \left\Vert \nabla \mu \right\Vert_{L^{2}\left( \Omega \right)}
^{2}+\left\Vert \nabla N_{\sigma }\right\Vert_{L^{2}\left( \Omega \right)} ^{2}+\int  \limits_{\Omega }p\left( \varphi
\right) \left( N_{\sigma }-\mu \right) ^{2}dx\right)+\frac{1}{2}\left\Vert \nabla \varphi _{t}\right\Vert_{L^{2}\left( \Omega \right)} ^{2}+\frac{\chi
	_{\sigma }}{2}\left\Vert \sigma _{t}\right\Vert_{L^{2}\left( \Omega \right)} ^{2}\\
&\leq\widehat{c}\left\Vert \varphi _{t}\right\Vert_{L^{2}\left( \Omega \right)} ^{2}+\Lambda \left( \left\Vert
\nabla N_{\sigma }\right\Vert_{L^{2}\left( \Omega \right)} ^{2}\left\Vert N_{\sigma }\right\Vert
_{H^{1}\left( \Omega \right) }^{2}+\left\Vert \nabla \mu \right\Vert
^{2}\left\Vert \mu \right\Vert _{H^{1}\left( \Omega \right) }^{2}\right)\\
&+\Lambda \left( \left\Vert N_{\sigma }\right\Vert _{H^{1}\left( \Omega
	\right) }^{2}+\left\Vert \mu \right\Vert _{H^{1}\left( \Omega \right)
}^{2}+\left\Vert p\left( \varphi \right) \right\Vert _{L^{6/5}\left( \Omega
	\right) }^{4}\right) ,  \tag{4.45}
\end{align*}%
where $ \Lambda $  is a positive constant depending on the norm of the initial data and $ \Psi, p, \Omega $.
Then, testing (1.1) by $\varphi _{t}$ in $L^{2}\left(\Omega \right) $, we get
\begin{align*}
&\frac{1}{2}\left\Vert \varphi _{t}\right\Vert_{L^{2}\left( \Omega \right)} ^{2}+\frac{1}{2}\frac{d}{dt}%
\left\Vert \Delta \varphi \right\Vert_{L^{2}\left( \Omega \right)} ^{2}\leq \frac{1}{8\widehat{c}}\left\Vert \nabla \varphi _{t}\right\Vert_{L^{2}\left( \Omega \right)}
^{2}+\chi _{\varphi }^{2}\left\Vert \sigma \right\Vert _{H^{1}\left( \Omega
	\right) }^{2}\\
&+c\left\Vert \Psi ^{\prime \prime }\left( \varphi \right)
\right\Vert _{L^{7/2}}^{2}\left\Vert \nabla \varphi \right\Vert
_{L^{14/3}}^{2}+c\left\Vert p\left( \varphi \right) \right\Vert
_{L^{3}\left( \Omega \right) }^{2}\left( \left\Vert \mu \right\Vert
_{H^{1}\left( \Omega \right) }^{2}+\left\Vert N_{\sigma }\right\Vert
_{H^{1}\left( \Omega \right) }^{2}\right) .  \tag{4.46}
\end{align*}%
Summing (4.45) and (4.46), using the same arguments in \cite{frigieri}, we infer
\begin{align*}
&\frac{1}{2}\frac{d}{dt}\left( \left\Vert \nabla \mu \right\Vert_{L^{2}\left( \Omega \right)}
^{2}+\left\Vert \nabla N_{\sigma }\right\Vert _{L^{2}\left( \Omega \right)}^{2}+c\left\| \Delta\varphi\right\|_{L^{2}\left( \Omega \right)} ^{2}+
\int \limits_{\Omega } p\left( \varphi
\right) \left( N_{\sigma }-\mu \right) ^{2}dx\right) \\
&+\frac{1}{4}\left\Vert
\nabla \varphi _{t}\right\Vert_{L^{2}\left( \Omega \right)} ^{2}+\frac{\chi _{\sigma }}{2}\left\Vert
\sigma _{t}\right\Vert_{L^{2}\left( \Omega \right)} ^{2}\leq \pi _{1}\left( \left\Vert \nabla \mu \right\Vert_{L^{2}\left( \Omega \right)} ^{2}+\left\Vert \nabla
N_{\sigma }\right\Vert_{L^{2}\left( \Omega \right)} ^{2}\right) +\pi _{2},  \tag{4.47}
\end{align*}%
where
\begin{align*}
&\pi _{1}:=c \left\Vert p\left( \varphi \right) \right\Vert
_{L^{3}\left( \Omega \right) }^{2}+\Lambda\left( \left\Vert \nabla \mu \right\Vert_{L^{2}\left( \Omega \right)} ^{2}+\left\Vert \nabla
N_{\sigma }\right\Vert_{L^{2}\left( \Omega \right)} ^{2}\right) +\Lambda \text{,}\\
 &\pi _{2}:=c\left\Vert \Psi ^{\prime \prime }\left( \varphi \right)
\right\Vert _{L^{7/2}\left( \Omega \right)}^{2}\left\Vert \nabla \varphi \right\Vert
_{L^{14/3}\left( \Omega \right)}^{2}+\Lambda \left\Vert P\left( \varphi \right) \right\Vert
_{L^{3}\left( \Omega \right) }^{2}+\Lambda.
\end{align*}% 
Moreover, we have%
\begin{equation*}
\left\Vert \pi _{1}\right\Vert _{L^{1}\left( 0,T\right) }\leq \Lambda(1+T)  \text{, }\left\Vert \pi _{2}\right\Vert _{L^{1}\left( 0,T\right)
}\leq \Lambda(1+T)  .
\end{equation*}%
Since $\varphi _{0}\in H^{3}\left( \Omega \right) $ and $ \partial _{\nu }\varphi_{0}=0 $ also the initial data in the Galerkin approximation are uniformly bounded in $ H^{3}\left( \Omega \right) $ as well as  $\mu \in
H^{1}\left( \Omega \right) $. Then, we apply Gronwall's lemma and from (4.47) it follows that
\begin{align}
&\nabla \mu, \Delta \varphi \in L^{\infty }\left( 0,T;L^{2}\left( \Omega \right)
\right),  & \sigma \in L^{\infty }\left( 0,T;H^{1}\left( \Omega \right)\right)\nonumber\\
&\nabla \varphi _{t}, \sigma _{t} \in L^{2}\left(
0,T;L^{2}\left( \Omega \right) \right), &\Psi ^{\prime }\left( \varphi \right) \in L^{\infty }\left(
0,T;L^{2}\left( \Omega \right) \right) \nonumber .
\end{align}
Hence, $\mu \in L^{\infty
}\left( 0,T;L^{2}\left( \Omega \right) \right) $ which yields%
\begin{equation}
\mu \in L^{\infty }\left( 0,T;H^{1}\left( \Omega \right) \right) . 
\tag{4.48}
\end{equation}%
In addition, using the elliptic regularity theory, we infer $\varphi \in L^{\infty }\left( 0,T;H^{2}\left( \Omega
\right) \right) $. Then, we obtain  $\Psi ^{\prime }\left( \varphi \right) \in L^{\infty
}\left( 0,T;H^{1}\left( \Omega \right) \right) $. Thus, together with (4.52), it follows that $\Delta \varphi \in L^{\infty }\left( 0,T;H^{1}\left( \Omega \right) \right) $. Therefore, by the elliptic regularity theory, we deduce
\begin{equation*}
\varphi \in L^{\infty }\left( 0,T;H^{3}\left( \Omega \right) \right) . 
\end{equation*}%
On the other hand, by integrating (4.50) from $0$ to $t$, we infer $\varphi _{t}\in L^{2}\left( 0,T;L^{2}\left(
\Omega \right) \right) $. Since $\nabla\varphi _{t}\in L^{2}\left( 0,T;L^{2}\left(
\Omega \right) \right) $, we deduce $\varphi _{t}\in L^{2}\left(
0,T;H^{1}\left( \Omega \right) \right) $.
Considering obtained controls in the problem (1.1)-(1.3), we also infer $ \sigma\in L^{2}\left( 0,T;H^{2}\left(
\Omega \right) \right) $ and $ \mu\in L^{2}\left( 0,T;H^{3}\left(
\Omega \right) \right) $.
\end{proof}
\begin{remark}
	If $ \Psi \in C^{4}(\mathbb{R}) $ and  the boundary $ \Gamma $ is smooth enough, it can be easily seen by using elliptic regularity that the strong solutions of the problem (1.1)-(1.5) have the further regularity $ \varphi \in L^{2}(0,T;H^{5}(\Omega)) $.
\end{remark}
Hence, as a consequence of Theorem 3.1 and Theorem 3.2, problem (1.1)-(1.5), by
the formula $S\left( t\right) \left( \varphi _{0},\sigma _{0}\right) =\left(
\varphi \left( t\right) ,\sigma \left( t\right) \right) $, generates a
weakly continuous semigroup $\left\{ S\left( t\right) \right\} _{t\geq 0}$
in $H^{1}\left( \Omega \right) \times L^{2}\left( \Omega \right) $, where $\left(
\varphi \left( t\right) ,\sigma \left( t\right) \right) $ is a weak solution determined by Theorem 3.1.

\section{ Long-time Dynamics of weak solutions}

\subsection{Asymptotic compactness}
We start with the following asymptotic compactness lemma:
\begin{lemma}
Let the conditions ($ M_{2} $), ($\Psi $) and ($P_{2}$) hold and $B$ be a bounded subset
of $H^{1}\left( \Omega \right) \times L^{2}\left( \Omega \right) $. Then
every sequence of the form $\left\{ S\left( t_{k}\right) \left( \xi
_{k},\eta _{k}\right) \right\} _{k=1}^{\infty },$ where $\left\{ \left( \xi
_{k},\eta _{k}\right) \right\} _{k=1}^{\infty }\subset B$, $t_{k}\rightarrow
\infty $, has a convergent subsequence in $H^{1}\left( \Omega \right) \times
L^{2}\left( \Omega \right) $.
\end{lemma}

\begin{proof}
\textbf{Step 1:} First of all, we will obtain the desired weak and strong
convergences, which are required in the second step to establish sequential
limit estimates. \newline
Let us consider the sequence $\left( \varphi _{k}\left( t\right) ,\sigma _{k}\left(
t\right) \right) =S\left( t+t_{k}-T_{0}\right) \left( \xi _{k},\eta
_{k}\right) $ for any $T_{0}>1$ and $t_{k}>T_{0}  $. Considering (3.5) in energy identity
(3.19), it follows that%
\begin{equation}
\sup_{0\leq t\leq \infty }\left\Vert \left( \varphi _{k}\left( t\right)
,\sigma _{k}\left( t\right) \right) \right\Vert _{H^{1}\left( \Omega \right)
\times L^{2}\left( \Omega \right) }\leq M  \tag{5.1}
\end{equation}%
where $M$ only depends on $B.$ Moreover, from (3.19) we also know that%
\begin{equation}
\left\{ 
\begin{array}{r}
\int\limits_{0}^{\infty }\left\Vert \nabla \mu _{k}\left( t\right)
\right\Vert _{L^{2}\left( \Omega \right) }^{2}dt\leq M, \\ 
\int\limits_{0}^{\infty }\left\Vert \nabla N_{\sigma k}\left( t\right)
\right\Vert _{L^{2}\left( \Omega \right) }^{2}dt\leq M, \\ 
\int\limits_{0}^{\infty }\left\Vert \sqrt{P\left( \varphi _{k}\left(
t\right) \right) }\left( N_{\sigma k}\left( t\right) -\mu _{k}\left(
t\right) \right) \right\Vert _{L^{2}\left( \Omega \right) }^{2}dt\leq M.%
\end{array}%
\right.  \tag{5.2}
\end{equation}%
Furthermore, in the existence part we have already proved that 
\begin{equation}
\left\{ 
\begin{array}{r}
\int\limits_{0}^{\infty }\left\Vert \left( \varphi _{kt}\left( t\right)
,\sigma _{kt}\left( t\right) \right) \right\Vert _{\overset{\ast }{H^{1}}\left( \Omega \right) \times \overset{\ast }{H^{1}}\left( \Omega \right) }^{2}dt\leq c\left( M\right), \\ 
\int\limits_{T}^{T+1}\left\Vert \varphi _{k}\left( t\right) \right\Vert
_{H^{3}\left( \Omega \right) }^{2}dt\leq c\left( M\right) , \\ 
\int\limits_{T}^{T+1}\left\Vert \sigma _{k}\left( t\right) \right\Vert
_{H^{1}\left( \Omega \right) }^{2}dt\leq c\left( M\right),%
\end{array}
\right. \tag{5.3}
\end{equation}%
$ \forall T\geq 0 $, where $c\left( M\right) $ is a nondecreasing, nonnegative function of $M$.
Hence, from (5.1) and (5.3)$_{1}$, there exists a subsequence $\left\{
k_{m}\right\} _{m=1}^{\infty }$ and $\left( \varphi _{m}\left( t\right)
,\sigma _{m}\left( t\right) \right) =S\left( t+t_{k_{m}}-T_{0}\right) \left(
\xi _{k_{m}},\eta _{k_{m}}\right) $ such that $t_{k_{m}}\geq T_{0}$ and%
\begin{equation}
\left\{ 
\begin{array}{lll}
\varphi _{m}\rightarrow \varphi& \text{ weakly star in }&L^{\infty }\left(
0,\infty ;H^{1}\left( \Omega \right) \right) , \\ 
\varphi _{mt}\rightarrow \varphi _{t}&\text{ weakly in }&L^{2}(0,\infty
;\overset{\ast }{H^{1}}\left( \Omega \right)), \\ 
\sigma _{m}\rightarrow \sigma &\text{ weakly star in }&L^{\infty }\left(
0,\infty ;L^{2}\left( \Omega \right) \right) , \\ 
\sigma _{mt}\rightarrow \sigma _{t}&\text{ weakly in }&L^{2}(0,\infty
;\overset{\ast }{H^{1}}\left( \Omega \right)),%
\end{array}%
\right.  \tag{5.4}
\end{equation}%
for some $\left( \varphi ,\sigma \right) \in L^{\infty }\left( 0,\infty
;H^{1}\left( \Omega \right) \times L^{2}\left( \Omega \right) \right) $ such that $\left( \varphi_{t} ,\sigma_{t} \right) \in
L^{2}(0,\infty ;\overset{\ast }{H^{1}}\left( \Omega \right) \times \overset{\ast }{H^{1}}\left(
\Omega \right)) $. Using Aubin-Lions lemma, it follows from (5.3) and (5.4) that%
\begin{equation}
\left\{ 
\begin{array}{llll}
\varphi _{m}\rightarrow \varphi& \text{ strongly in }&C\left( \left[ 0,T\right]
;L^{\kappa }\left( \Omega \right) \right) &\forall T\geq 0, \\ 
\varphi _{m}\rightarrow \varphi &\text{ strongly in }&L^{2}\left(
T,T+1;H^{2}\left( \Omega \right) \right) &\forall T\geq 0, \\ 
\sigma _{m}\rightarrow \sigma &\text{ strongly in }&L^{2}\left(
T,T+1;L^{2}\left( \Omega \right) \right) &\forall T\geq 0,%
\end{array}%
\right.  \tag{5.5}
\end{equation}%
where $2\leq \kappa <6  $. Moreover, from (5.5)$_{2}$ and (5.5)$_{3}$, we have%
\begin{equation}
N_{\sigma _{m}}\rightarrow N_{\sigma }\text{, strongly in }L^{2}\left(
T,T+1;L^{2}\left( \Omega \right) \right) \text{ \ }\forall T\geq 0. 
\tag{5.6}
\end{equation}%
Furthermore, from the Gagliardo-Nirenberg inequality we obtain%
\begin{equation*}
\left\Vert \varphi _{m}-\varphi \right\Vert _{L^{6}\left( \left(
T,T+1\right) \times \Omega \right) }\leq c_{1}\left\Vert \varphi
_{m}-\varphi \right\Vert _{L^{\infty }\left( T,T+1;L^{2}\left( \Omega
\right) \right) }^{2/3}\left\Vert \varphi _{m}-\varphi \right\Vert
_{L^{2}\left( T,T+1;H^{3}\left( \Omega \right) \right) }^{1/3}.
\end{equation*}%
together with (5.5)$_{1}$, which yields%
\begin{equation}
\varphi _{m}\rightarrow \varphi \text{ strongly in }L^{6}\left( \left(
T,T+1\right) \times \Omega \right) \text{ \ }\forall T\geq 0.  \tag{5.7}
\end{equation}%
On the other hand, from (3.4), we have%
\begin{equation*}
\left\vert \Psi \left( \varphi _{m}\right) \right\vert \leq c_{2}\left(
1+\left\vert \varphi _{m}\right\vert ^{\rho}\right) \text{, \  }2\leq \rho<6.
\end{equation*}%
Also from (5.5)$_{1}$, it follows that
\begin{equation}
\varphi _{m} \rightarrow \varphi  \text{
almost everywhere in }(0,\infty)\times\Omega .  \tag{5.8}
\end{equation}
Since $\Psi $ is continuous (5.8) yields that
\begin{equation*}
\Psi \left( \varphi _{m} \right) \rightarrow \Psi \left(
\varphi  \right) \text{ almost everywhere in }(0,\infty)\times\Omega .
\end{equation*}%
Moreover from (5.7), we have%
\begin{equation*}
\lim_{m\rightarrow \infty }\overset{T+1}{\underset{T}{\int }}\underset{%
\Omega }{\int }c_{2}\left( 1+\left\vert \varphi _{m}\right\vert ^{\rho}\right)
dxdt=\overset{T+1}{\underset{T}{\int }}\underset{\Omega }{\int }%
c_{2}\left( 1+\left\vert \varphi \right\vert ^{\rho}\right) dxdt.
\end{equation*}%
Then by the generalized Lebesgue theorem, we infer%
\begin{equation}
\lim_{m\rightarrow \infty }\overset{T+1}{\underset{T}{\int }}\underset{%
\Omega }{\int }\Psi \left( \varphi _{m}\left( t\right) \right)dxdt =\overset{T+1}{\underset{T}{\int }}\underset{%
\Omega }{\int }\Psi \left(
\varphi \left( t\right) \right)dxdt \text{ \ }\forall T\geq 0,  \tag{5.9}
\end{equation}%
and similarly recalling (5.5)$_{1}$, we also deduce%
\begin{equation}
\lim_{m\rightarrow \infty }\underset{\Omega }{\int }\Psi \left( \varphi
_{m}\left( T\right) \right)dx =\underset{\Omega }{\int }\Psi \left( \varphi \left( T\right) \right)dx 
\text{ \ }\forall T\geq 0\text{.}  \tag{5.10}
\end{equation}%
On the other hand, recalling that $\left\{ \Psi ^{\prime }\left( \varphi
_{m}\right) \right\} _{m=1}^{\infty }$ is uniformly bounded in $L^{2}\left(
T,T+1;H^{1}\left( \Omega \right) \right) $, $\forall T\geq 0$, from (5.8) there holds%
\begin{equation}
\Psi ^{\prime }\left( \varphi _{m}\right) \rightarrow \Psi ^{\prime }\left(
\varphi \right) \text{ weakly in }L^{2}\left( T,T+1;H^{1}\left( \Omega
\right) \right).  \tag{5.11}
\end{equation}%
By using the Gagliardo-Nirenberg inequality, we have the embedding%
\begin{equation*}
L^{\infty }\left( T,T+1;H^{1}\left( \Omega \right) \right) \cap L^{2}\left(
T,T+1;H^{3}\left( \Omega \right) \right) \hookrightarrow L^{8}\left(
T,T+1;L^{\infty }\left( \Omega \right) \right)  \text{ \ }\forall T\geq 0\text{,}
\end{equation*}%
which yields%
\begin{equation*}
\left\Vert \varphi _{n}\right\Vert _{L^{8}\left( T,T+1;L^{\infty }\left(
\Omega \right) \right) }\leq c\left( M\right) \text{  \ }\forall T\geq 0.
\end{equation*}%
Then, since $\rho<6$, we deduce 
\begin{align*}
&\overset{T+1}{\underset{T}{\int }}\underset{\Omega }{\int }\left( \Psi
^{\prime }\left( \varphi _{n}\right) -\Psi ^{\prime }\left( \varphi
_{m}\right) \right) ^{2}dxdt\\
&\leq c_{1}\overset{T+1}{\underset{T}{\int }}%
\underset{\Omega }{\int }\left\vert \varphi _{n}-\varphi _{m}\right\vert
^{2}\left( 1+\left\vert \varphi _{n}\right\vert ^{2\left( \rho-2\right)
}+\left\vert \varphi _{m}\right\vert ^{2\left( \rho-2\right) }\right) dxdt\\
&\leq c_{2}\overset{T+1}{\underset{T}{\int }}\left( 1+\left\Vert \varphi
_{n}\right\Vert _{L^{\infty }\left( \Omega \right)}^{2\left( \rho-2\right) }+\left\Vert \varphi
_{m}\right\Vert _{L^{\infty }\left( \Omega \right)}^{2\left( \rho-2\right) }\right) \left\Vert
\varphi _{n}-\varphi _{m}\right\Vert _{L^{2}\left( \Omega \right)}^{2}dt\\
&\leq c_{3}\left\Vert \varphi _{n}-\varphi _{m}\right\Vert _{C\left(
	0,T;L^{2}\left( \Omega \right) \right) }^{2}\overset{T+1}{\underset{T}{\int 
}}\left( 1+\left\Vert \varphi _{n}\right\Vert _{L^{\infty }\left( \Omega
	\right) }^{2\left( \rho-2\right) }+\left\Vert \varphi _{m}\right\Vert
_{L^{\infty }\left( \Omega \right) }^{2\left( \rho-2\right) }\right) dt\\
&\leq c_{3}\left\Vert \varphi _{n}-\varphi _{m}\right\Vert _{C\left(
	0,T;L^{2}\left( \Omega \right) \right) }^{2}\text{ \ }\forall T\geq 0. 
\tag{5.12}
\end{align*}%
Taking into account (5.5)$_{1}$ in (5.12), we have%
\begin{equation*}
\Psi ^{\prime }\left( \varphi _{m}\right) \rightarrow \Psi ^{\prime }\left(
\varphi \right) \text{ strongly in }L^{2}\left( T,T+1;L^{2}\left( \Omega
\right) \right) \text{ \ }\forall T\geq 0,
\end{equation*}%
which together with (5.5)$_{2}$ and (5.5)$_{3}$ yields 
\begin{equation}
\mu _{m}\rightarrow \mu \text{ strongly in }L^{2}\left( T,T+1;L^{2}\left(
\Omega \right) \right) \text{ \ }\forall T\geq 0.  \tag{5.13}
\end{equation}%
\textbf{Step 2: }Now, from the problem (1.1)-(1.3), we have%
\begin{align*}
&\left( \varphi _{n}-\varphi _{m}\right) _{t}-\Delta \left( \mu _{n}-\mu _{m}\right)\\
&=\left( p\left( \varphi _{n}\right)
-p\left( \varphi _{m}\right) \right) \left( N_{\sigma _{n}}-\mu _{n}\right)
+p\left( \varphi _{m}\right) \left( N_{\sigma _{n}}-N_{\sigma _{m}}\right)
-p\left( \varphi _{m}\right) \left( \mu _{n}-\mu _{m}\right) ,  \tag{5.14}\\
&\mu _{n}-\mu _{m}=-\Delta \left( \varphi _{n}-\varphi _{m}\right) +\Psi
^{\prime }\left( \varphi _{n}\right) -\Psi ^{\prime }\left( \varphi
_{m}\right) -\chi _{\varphi }\left( \sigma _{n}-\sigma _{m}\right),\tag{5.15}\\
&\left( \sigma _{n}-\sigma _{m}\right) _{t}-\Delta \left( N_{\sigma _{n}}-N_{\sigma _{m}}\right)\\
&=-\left( p\left( \varphi
_{n}\right) -p\left( \varphi _{m}\right) \right) \left( N_{\sigma _{n}}-\mu
_{n}\right) -p\left( \varphi _{m}\right) \left( N_{\sigma _{n}}-N_{\sigma
	_{m}}\right) +p\left( \varphi _{m}\right) \left( \mu _{n}-\mu _{m}\right) . 
\tag{5.16}
\end{align*}%
Testing (5.14) with $\left( \mu _{n}-\mu _{m}\right) $, we obtain%
\begin{align*}
&\left\langle \left\langle \left( \varphi _{n}-\varphi _{m}\right) _{t},\left(
\mu _{n}-\mu _{m}\right)\right\rangle \right\rangle  +\left\Vert \nabla \left( \mu _{n}-\mu
_{m}\right) \right\Vert _{L^{2}\left( \Omega\right) }^{2}+\underset{\Omega }{\int }p\left(
\varphi _{m}\right) \left( \mu _{n}-\mu _{m}\right) ^{2}dx\\
&-\underset{\Omega }{\int }p\left( \varphi _{m}\right) \left( N_{\sigma
	_{n}}-N_{\sigma _{m}}\right) \left( \mu _{n}-\mu _{m}\right) dx=\underset{%
	\Omega }{\int }\left( p\left( \varphi _{n}\right) -p\left( \varphi
_{m}\right) \right) \left( N_{\sigma _{n}}-\mu _{n}\right) \left( \mu
_{n}-\mu _{m}\right) dx.  \tag{5.17}
\end{align*}%
Testing (5.15) with $\left( \varphi _{n}-\varphi _{m}\right) _{t}$, it follows that%
\begin{align*}
&-\left\langle \left\langle \left( \varphi _{n}-\varphi _{m}\right) _{t},\left(
\mu _{n}-\mu _{m}\right)\right\rangle \right\rangle +\frac{1}{2}\frac{d}{dt}\left\Vert
\nabla \left( \varphi _{n}-\varphi _{m}\right) \right\Vert _{L^{2}\left( \Omega\right) }^{2}\\
&+\underset{\Omega }{\int }\left( \Psi ^{\prime }\left( \varphi _{n}\right)
-\Psi ^{\prime }\left( \varphi _{m}\right) \right) \left( \varphi
_{n}-\varphi _{m}\right) _{t}dx-\underset{\Omega }{\int }\chi _{\varphi
}\left( \sigma _{n}-\sigma _{m}\right) \left( \varphi _{n}-\varphi
_{m}\right) _{t}dx=0.  \tag{5.18}
\end{align*}
Then, testing (5.16) with $\left( N_{\sigma _{n}}-N_{\sigma _{m}}\right) 
$, we infer%
\begin{align*}
&\frac{\chi _{\sigma }}{2}\frac{d}{dt}\left\Vert \sigma _{n}-\sigma
_{m}\right\Vert _{L^{2}}^{2}-\underset{\Omega }{\int }\chi _{\varphi
}\left( \varphi _{n}-\varphi _{m}\right) \left( \sigma _{n}-\sigma
_{m}\right) _{t}dx+\left\Vert \nabla \left( N_{\sigma _{n}}-N_{\sigma
_{m}}\right) \right\Vert _{L^{2}}^{2}\\
&+\underset{\Omega }{\int }p\left( \varphi _{m}\right) \left( N_{\sigma
	_{n}}-N_{\sigma _{m}}\right) ^{2}dx-\underset{\Omega }{\int }p\left(
\varphi _{m}\right) \left( \mu _{n}-\mu _{m}\right) \left( N_{\sigma
	_{n}}-N_{\sigma _{m}}\right) dx\\
&=-\underset{\Omega }{\int }\left( p\left( \varphi _{n}\right) -p\left(
\varphi _{m}\right) \right) \left( N_{\sigma _{n}}-\mu _{n}\right) \left(
N_{\sigma _{n}}-N_{\sigma _{m}}\right) dx.  \tag{5.19}
\end{align*}
Summing (5.17)-(5.19), it follows that%
\begin{align*}
&\frac{d}{dt}\left( \frac{1}{2}\left\Vert \nabla \left( \varphi _{n}-\varphi
_{m}\right) \right\Vert _{L^{2}}^{2}+\frac{\chi _{\sigma }}{2}\left\Vert
\sigma _{n}-\sigma _{m}\right\Vert _{L^{2}}^{2}+\int \chi _{\varphi }\left(
\sigma _{n}-\sigma _{m}\right) \left( \varphi _{n}-\varphi _{m}\right)
\right)\\
&+\left\Vert \nabla \left( \mu _{n}-\mu _{m}\right) \right\Vert
_{L^{2}}^{2}+\left\Vert \nabla \left( N_{\sigma _{n}}-N_{\sigma _{m}}\right)
\right\Vert _{L^{2}}^{2}+\left\Vert \sqrt{p\left( \varphi _{m}\right) }%
\left( \left( N_{\sigma _{n}}-\mu _{n}\right) -\left( N_{\sigma _{m}}-\mu
_{m}\right) \right) \right\Vert _{L^{2}}^{2}\\
&\leq \underset{\Omega }{\int }\left( \Psi ^{\prime }\left( \varphi
_{m}\right) -\Psi \left( \varphi _{n}\right) \right) \left( \varphi
_{n}-\varphi _{m}\right) _{t}dx\\
&+ \underset{\Omega }{\int }\left\vert\left( p\left( \varphi _{n}\right)
-p\left( \varphi _{m}\right) \right) \left( N_{\sigma _{n}}-\mu _{n}\right)
\left( \left( \mu _{n}-\mu _{m}\right) +\left( N_{\sigma _{n}}-N_{\sigma
	_{m}}\right) \right)\right\vert dx.
\end{align*}
From the above equality, by using Young inequality, we obtain that 
\begin{align*}
&\frac{d}{dt}\left( \frac{1}{2}\left\Vert \nabla \left( \varphi _{n}-\varphi
_{m}\right) \right\Vert _{L^{2}\left( \Omega\right) }^{2}+\frac{\chi _{\sigma }}{2}\left\Vert
\sigma _{n}-\sigma _{m}\right\Vert _{L^{2}\left( \Omega\right) }^{2}+\underset{\Omega }{\int } \chi _{\varphi }\left(
\sigma _{n}-\sigma _{m}\right) \left( \varphi _{n}-\varphi _{m}\right)
\right)\\
&+c_{1}\left\Vert \nabla \left( \mu _{n}-\mu _{m}\right) \right\Vert
_{L^{2}\left( \Omega\right)}^{2}+c_{1}\left\Vert \nabla \left( N_{\sigma _{n}}-N_{\sigma
	_{m}}\right) \right\Vert _{L^{2}\left( \Omega\right)}^{2}\\
&\leq 2\left\Vert \left( p\left( \varphi _{n}\right) -p\left( \varphi
_{m}\right) \right) \left( N_{\sigma _{n}}-\mu _{n}\right) \right\Vert
_{\overset{\ast }{H^{1}}\left( \Omega \right)}^{2}+c_{2}\left\Vert  \mu _{n}-\mu _{m} \right\Vert
_{L^{2}\left( \Omega\right)}^{2}+c_{2}\left\Vert N_{\sigma _{n}}-N_{\sigma _{m}}\right\Vert
_{L^{2}\left( \Omega\right)}^{2}\\
&+\underset{\Omega }{\int }\left( \Psi ^{\prime }\left( \varphi _{m}\right)
-\Psi^{\prime } \left( \varphi _{n}\right) \right) \left( \varphi _{n}-\varphi
_{m}\right) _{t}dx.  \tag{5.20}
\end{align*}%
For the first term on the right hand side of (5.20), we have%
\begin{align*}
\left\Vert \left( p\left( \varphi _{n}\right) -p\left( \varphi _{m}\right)
\right) \left( N_{\sigma _{n}}-\mu _{n}\right) \right\Vert _{\overset{\ast }{H^{1}}\left( \Omega \right)}^{2}&\leq c_{3}\left\Vert p\left( \varphi _{n}\right) -p\left( \varphi
_{m}\right) \right\Vert _{L^{3/2}\left( \Omega \right) }^{2}\left\Vert
\left( N_{\sigma _{n}}-\mu _{n}\right) \right\Vert _{L^{6}\left( \Omega
	\right) }^{2}\\
&\leq c_{3}\left\Vert \varphi _{n}-\varphi _{m}\right\Vert _{L^{6}\left( \Omega \right)}^{2}\left(
\left\Vert N_{\sigma _{n}}\right\Vert _{L^{6}\left( \Omega \right)
}^{2}+\left\Vert \mu _{n}\right\Vert _{L^{6}\left( \Omega \right)
}^{2}\right).  \tag{5.21}
\end{align*}%
Considering (5.21) in (5.20), we infer%
\begin{align*}
&\frac{d}{dt}\left( \mathcal{E}\left( \varphi _{n}\left( t\right) -\varphi
_{m}\left( t\right) ,\sigma _{n}\left( t\right) -\sigma _{m}\left( t\right)
\right) \right)\\
&\leq c_{3}\left( \left\Vert N_{\sigma _{n}}\right\Vert _{L^{6}\left( \Omega
	\right) }^{2}+\left\Vert \mu _{n}\right\Vert _{L^{6}\left( \Omega \right)
}^{2}\right) \left\Vert \varphi _{n}-\varphi _{m}\right\Vert
_{L^{6}\left( \Omega \right)}^{2}+K^{n,m}\left( t\right) ,  \tag{5.22}
\end{align*}%
where $\mathcal{E}\left( \varphi ,\sigma \right) :=\frac{1}{2}\left\Vert
\nabla \varphi \right\Vert _{L^{2}\left( \Omega \right)}^{2}+\frac{\chi _{\sigma }}{2}\left\Vert
\sigma \right\Vert _{L^{2}\left( \Omega \right)}^{2}+\underset{\Omega }{\int } \chi _{\varphi }\sigma \varphi $ and $%
K^{n,m}\left( t\right) =c_{2}\left\Vert \mu _{n}-\mu _{m}
\right\Vert _{L^{2}\left( \Omega \right)}^{2}+c_{2}\left\Vert N_{\sigma _{n}}-N_{\sigma
_{m}}\right\Vert _{L^{2}\left( \Omega \right)}^{2}+\underset{\Omega }{\int }\left( \Psi ^{\prime
}\left( \varphi _{m}\right) -\Psi^{\prime
} \left( \varphi _{n}\right) \right) \left(
\varphi _{n}-\varphi _{m}\right) _{t}dx$.\newline
On the other hand, it can be easily obtained that%
\begin{equation*}
\left\Vert \varphi \right\Vert _{H^{1}\left( \Omega \right) }^{2}+\left\Vert
\sigma \right\Vert _{L^{2}\left( \Omega \right) }^{2}\leq K_{1}\mathcal{E}%
\left( \varphi ,\sigma \right) +K_{2}\left\Vert \varphi \right\Vert
_{L^{2}\left( \Omega \right) }^{2}
\end{equation*}%
for some constants $K_{1}$, $K_{2}>0$. Then, from (5.22), it follows that%
\begin{align*}
&\frac{d}{dt}\left( \mathcal{E}\left( \varphi _{n}\left( t\right) -\varphi
_{m}\left( t\right) ,\sigma _{n}\left( t\right) -\sigma _{m}\left( t\right)
\right) \right)\\
&\leq c_{4}\left( \left\Vert N_{\sigma _{n}}\right\Vert _{L^{6}\left( \Omega
	\right) }^{2}+\left\Vert \mu _{n}\right\Vert _{L^{6}\left( \Omega \right)
}^{2}\right) \mathcal{E}\left( \varphi _{n}\left( t\right) -\varphi
_{m}\left( t\right) ,\sigma _{n}\left( t\right) -\sigma _{m}\left( t\right)
\right) +\widetilde{K}^{n,m}\left( t\right) ,  \tag{5.23}
\end{align*}
where $\widetilde{K}^{n,m}\left( t\right) :=K^{n,m}\left( t\right)
+c_{5}\left( \left\Vert N_{\sigma _{n}}\right\Vert _{L^{6}\left( \Omega
\right) }^{2}+\left\Vert \mu _{n}\right\Vert _{L^{6}\left( \Omega \right)
}^{2}\right) \left\Vert \varphi _{n}-\varphi _{m}\right\Vert _{L^{2}.}^{2}$%
\newline
Now let us take $T\leq t\leq T+1$ and multiply (5.23) with $%
e^{-c_{4}\int\limits_{T}^{t}\left( \left\Vert N_{\sigma _{n}}\left( \tau
\right) \right\Vert _{L^{6}\left( \Omega \right) }^{2}+\left\Vert \mu
_{n}\right\Vert \left( \tau \right) _{L^{6}\left( \Omega \right)
}^{2}\right) d\tau }.$ Then, we obtain%
\begin{align*}
&\frac{d}{dt}\left( \mathcal{E}\left( \varphi _{n}\left( t\right) -\varphi
_{m}\left( t\right) ,\sigma _{n}\left( t\right) -\sigma _{m}\left( t\right)
\right) e^{-c_{4}\int\limits_{T}^{t}\left( \left\Vert N_{\sigma _{n}}\left(
\tau \right) \right\Vert _{L^{6}\left( \Omega \right) }^{2}+\left\Vert \mu
_{n}\right\Vert \left( \tau \right) _{L^{6}\left( \Omega \right)
}^{2}\right) d\tau }\right)\\
&\leq \widetilde{K}^{n,m}\left( t\right) e^{-c_{4}\int\limits_{T}^{t}\left(
	\left\Vert N_{\sigma _{n}}\left( \tau \right) \right\Vert _{L^{6}\left(
		\Omega \right) }^{2}+\left\Vert \mu _{n}\right\Vert \left( \tau \right)
	_{L^{6}\left( \Omega \right) }^{2}\right) d\tau }\leq \widetilde{K}%
^{n,m}\left( t\right) . \tag{5.24}
\end{align*}
Integrating (5.24) between $t$ and $T+1,$ by using the uniform boundedness
of $\left\{ N_{\sigma _{n}}\right\} _{n=1}^{\infty }$ and $\left\{ \mu
_{n}\right\} _{n=1}^{\infty }$ in $L^{2}\left( T,T+1\right) ;H^{1}\left(
\Omega \right) $, we infer%
\begin{align*}
&\mathcal{E}\left( \varphi _{n}\left( T+1\right) -\varphi _{m}\left(
T+1\right) ,\sigma _{n}\left( T+1\right) -\sigma _{m}\left( T+1\right)
\right)\\
&\leq \mathcal{E}\left( \varphi _{n}\left( t\right) -\varphi _{m}\left(
t\right) ,\sigma _{n}\left( t\right) -\sigma _{m}\left( t\right) \right)
e^{c_{4}\int\limits_{t}^{T+1}\left( \left\Vert N_{\sigma _{n}}\left( \tau
	\right) \right\Vert _{L^{6}\left( \Omega \right) }^{2}+\left\Vert \mu
	_{n}\right\Vert \left( \tau \right) _{L^{6}\left( \Omega \right)
	}^{2}\right) d\tau }\\
&+\int\limits_{t}^{T+1}\widetilde{K}^{n,m}\left( \tau \right) d\tau
e^{c_{4}\int\limits_{T}^{T+1}\left( \left\Vert N_{\sigma _{n}}\left( \tau
	\right) \right\Vert _{L^{6}\left( \Omega \right) }^{2}+\left\Vert \mu
	_{n}\right\Vert \left( \tau \right) _{L^{6}\left( \Omega \right)
	}^{2}\right) d\tau }\\
&\leq c_{6}\mathcal{E}\left( \varphi _{n}\left( t\right) -\varphi _{m}\left(
t\right) ,\sigma _{n}\left( t\right) -\sigma _{m}\left( t\right) \right)
+c_{6}\int\limits_{t}^{T+1}\widetilde{K}^{n,m}\left( \tau \right) d\tau .
\end{align*}
Finally, integrating the last inequality with respect to $t$ between $T$ and 
$T+1$, we deduce%
\begin{align*}
&\mathcal{E}\left( \varphi _{n}\left( T+1\right) -\varphi _{m}\left(
T+1\right) ,\sigma _{n}\left( T+1\right) -\sigma _{m}\left( T+1\right)
\right)\\
&\leq c_{6}\int\limits_{T}^{T+1}\mathcal{E}\left( \varphi _{n}\left( t\right)
-\varphi _{m}\left( t\right) ,\sigma _{n}\left( t\right) -\sigma _{m}\left(
t\right) \right) dt+c_{6}\int\limits_{T}^{T+1}\int\limits_{t}^{T+1}\widetilde{K%
}^{n,m}\left( \tau \right) d\tau dt. \tag{5.25}
\end{align*}
Recalling (5.5)$_{2}$ and (5.5)$_{3}$, it follows that%
\begin{equation}
\underset{n\rightarrow \infty }{\lim \sup }\text{ }\underset{m\rightarrow \infty }{%
\lim \sup }\underset{T}{\overset{T+1}{\int }}\mathcal{E}\left( \varphi
_{n}\left( t\right) -\varphi _{m}\left( t\right) ,\sigma _{n}\left( t\right)
-\sigma _{m}\left( t\right) \right) dt=0  \tag{5.26}
\end{equation}%
Now, let us estimate the last term of $\widetilde{K}^{n,m}\left( t\right) $.
Considering (5.4)$_{2}$ and (5.11), we have 
\begin{align*}
&\underset{n\rightarrow \infty }{\lim \sup }\text{ }\underset{m\rightarrow
\infty }{\lim \sup }\int\limits_{T}^{T+1}\int\limits_{t}^{T+1}\int\limits_{%
\Omega }\left( \Psi ^{\prime }\left( \varphi _{m}\right) -\Psi ^{\prime
}\left( \varphi _{n}\right) \right) \left( \varphi _{n}-\varphi _{m}\right)
_{t}dxd\tau dt\\
&=\underset{n\rightarrow \infty }{\lim \sup }\text{ }\underset{m\rightarrow
	\infty }{\lim \sup }\int\limits_{T}^{T+1}\int\limits_{t}^{T+1}\int\limits_{%
	\Omega }\left( -\frac{d}{dt}\left( \Psi \left( \varphi _{n}\right) \right) -%
\frac{d}{dt}\left( \Psi \left( \varphi _{m}\right) \right) +\Psi ^{\prime
}\left( \varphi _{n}\right) \varphi _{mt}+\Psi ^{\prime }\left( \varphi
_{m}\right) \varphi _{nt}\right) dxd\tau dt\\
&=\underset{n\rightarrow \infty }{\lim \sup }\text{ }\underset{m\rightarrow
	\infty }{\lim \sup }\int\limits_{T}^{T+1}\int\limits_{t}^{T+1}\int\limits_{%
	\Omega }\left( -\frac{d}{dt}\left( \Psi \left( \varphi _{n}\right) \right) -%
\frac{d}{dt}\left( \Psi \left( \varphi _{m}\right) \right) +2\frac{d}{dt}%
\left( \Psi \left( \varphi \right) \right) \right) dxd\tau dt\\
&\leq \underset{n\rightarrow \infty }{\lim \sup }\text{ }\underset{%
	m\rightarrow \infty }{\lim \sup }\int\limits_{\Omega }\left( -\Psi \left(
\varphi _{n}\left( T+1\right) \right) -\Psi \left( \varphi _{m}\left(
T+1\right) \right) +2\Psi \left( \varphi \left( T+1\right) \right) \right) dx\\
&+\underset{n\rightarrow \infty }{\lim \sup }\text{ }\underset{m\rightarrow
	\infty }{\lim \sup }\int\limits_{T}^{T+1}\int\limits_{\Omega }\left( 
\Psi \left( \varphi _{n}\left( t\right) \right) +\Psi \left( \varphi
_{m}\left( t\right) \right) -2\Psi \left( \varphi \left( t\right) \right)
\right)  dxdt.
\end{align*}%
Then taking into account (5.9) and (5.10) in the above estimate, we deduce%
\begin{equation}
\underset{n\rightarrow \infty }{\lim \sup }\text{ }\underset{m\rightarrow
	\infty }{\lim \sup }\int\limits_{T}^{T+1}\int\limits_{t}^{T+1}\int\limits_{\Omega
}\left( \Psi ^{\prime }\left( \varphi _{m}\right) -\Psi ^{\prime }\left(
\varphi _{n}\right) \right) \left( \varphi _{n}-\varphi _{m}\right)
_{t}dxd\tau dt=0\text{ \  }\forall T\geq 0\text{.}  \tag{5.27}
\end{equation}%
and considering (5.5)$_{1}$, (5.6), (5.13) and (5.27), we get
\begin{equation}
\underset{n\rightarrow \infty }{\lim \sup }\text{ }\underset{m\rightarrow \infty }{%
\lim \sup }\int\limits_{T}^{T+1}\int\limits_{t}^{T+1}\widetilde{K}%
^{n,m}\left( t\right) dt=0.  \tag{5.28}
\end{equation}%
Considering (5.26), (5.28) in (5.25), we have
\begin{equation*}
\underset{n\rightarrow \infty }{\lim \sup }\text{ }\underset{m\rightarrow \infty }{%
\lim \sup } \ \ \mathcal{E}\left( \varphi _{n}\left( T+1\right) -\varphi
_{m}\left( T+1\right) ,\sigma _{n}\left( T+1\right) -\sigma _{m}\left(
T+1\right) \right) =0\text{ \ }\forall T\geq 0
\end{equation*}%
which \ yields%
\begin{equation*}
\underset{n\rightarrow \infty }{\lim \sup }\text{ }\underset{m\rightarrow \infty }{%
\lim \sup }\left( \left\Vert S\left( T+1+t_{k_{n}}-T_{0}\right) \left( \xi
_{k_{n}},\eta _{k_{n}}\right) -S\left( T+1+t_{k_{m}}-T_{0}\right) \left( \xi
_{k_{m}},\eta _{k_{m}}\right) \right\Vert _{H^{1}\left( \Omega \right)
\times L^{2}\left( \Omega \right) }^{2}\right) =0
\end{equation*}%
for all $T\geq 0$. Choosing $T=T_{0}-1$ in the last inequality, we deduce%
\begin{equation*}
\underset{n\rightarrow \infty }{\lim \sup }\text{ }\underset{m\rightarrow \infty }{%
\lim \sup }\left\Vert S\left( t_{k_{n}}\right) \left( \xi _{k_{n}},\eta
_{k_{n}}\right) -S\left( t_{k_{m}}\right) \left( \xi _{k_{m}},\eta
_{k_{m}}\right) \right\Vert _{H^{1}\left( \Omega \right) \times L^{2}\left(
\Omega \right) }^{2}=0\text{,}
\end{equation*}%
and consequently%
\begin{equation*}
\underset{n\rightarrow \infty }{\lim \inf }\underset{m\rightarrow \infty }{%
\text{ }\lim \inf }\left\Vert S\left( t_{n}\right) \left( \xi _{n},\eta
_{n}\right) -S\left( t_{m}\right) \left( \xi _{m},\eta _{m}\right)
\right\Vert _{H^{1}\left( \Omega \right) \times L^{2}\left( \Omega \right)
}^{2}=0.
\end{equation*}%
The last equality, as shown in the proof of  \cite[Lemma 3.4]{khan2},
gives us the desired result.
\end{proof}
\subsection{Gradient System}
\begin{lemma}
Under conditions ($\Psi $) and ($P_{2}$), the dynamical system $\left(
H^{1}\left( \Omega \right) \times L^{2}\left( \Omega \right) ,S\left(
t\right) \right) \mathcal{\ }$\ is a gradient system, i.e. there exists a strict Lyapunov function for $\left(
H^{1}\left( \Omega \right) \times L^{2}\left( \Omega \right) ,S\left(
t\right) \right) $ on the whole phase space. 
\end{lemma}

\begin{proof}
We consider the energy functional%
\begin{equation*}
E\left( \varphi ,\sigma \right) :=\frac{1}{2}\left\Vert \nabla \varphi
\right\Vert _{L_{2}\left( \Omega \right) }^{2}+\underset{\Omega }{\int }%
\Psi \left( \varphi \right) +\frac{\chi _{\sigma }}{2}\left\Vert \sigma
\right\Vert _{L_{2}\left( \Omega \right) }^{2}+\underset{\Omega }{\int }%
\chi _{\varphi }\sigma \left( 1-\varphi \right) .
\end{equation*}%
Then from (3.12), we obtain that%
\begin{align*}
&E\left( \varphi \left( t\right) ,\sigma \left( t\right) \right)
+\int\limits_{s}^{t}\left( \left\Vert \nabla \mu\left( \tau\right)  \right\Vert _{L_{2}\left(
\Omega \right) }^{2}+\left\Vert \nabla N_{\sigma }\left( \tau\right)\right\Vert _{L_{2}\left(
\Omega \right) }^{2}\right)d\tau \\
&+\int\limits_{s}^{t}\int\limits_{\Omega
}p\left( \varphi\left( \tau,x\right)  \right) \left( N_{\sigma }\left( \tau,x\right)-\mu\left( \tau,x\right) \right) ^{2}dxd\tau=E\left( \varphi
\left( s\right) ,\sigma \left( s\right) \right) \text{ \ \ }\forall t\geq s%
\text{.}  \tag{5.29}
\end{align*}%
From (5.29), it readily follows that $E\left( \varphi \left( t\right)
,\sigma \left( t\right) \right) $ is a nonincreasing function with respect
to $t.$ So, $E\left( \varphi \left( t\right) ,\sigma \left( t\right) \right) 
$ is a Lyapunov function for $\left( H^{1}\left( \Omega \right) \times
L^{2}\left( \Omega \right) ,S\left( t\right) \right) $. \newline
Now, we will show that $E\left( \varphi \left( t\right) ,\sigma \left(
t\right) \right) $ is a strict Lyapunov function. Assume that for $\left(
\varphi \left( t\right) ,\sigma \left( t\right) \right) =S\left( t\right)
\left( \varphi _{0},\sigma _{0}\right) $, the following equality holds:%
\begin{equation*}
E\left( \varphi \left( t\right) ,\sigma \left( t\right) \right) =E\left(
\varphi _{0},\sigma _{0}\right) \text{ \ }\forall t\geq 0\text{.}
\end{equation*}%
Then, from (5.29) we have%
\begin{equation}
\left\{ 
\begin{array}{r}
\nabla \mu \left( t,\cdot \right) =0\text{ \ a.e. in }\Omega , \\ 
\nabla N_{\sigma }\left( t,\cdot \right) =0\text{ \ a.e. in }\Omega , \\ 
P\left( \varphi \right) \left( N_{\sigma }\left( t,\cdot \right) -\mu \left(
t,\cdot \right) \right) =0\text{ \ a.e. in }\Omega, %
\end{array}%
\right.  \tag{5.30}
\end{equation}
for $t\geq 0.$ Considering (5.30), in the problem (1.1)-(1.3), we obtain%
\begin{equation*}
\left\{ 
\begin{array}{c}
\varphi _{t}\left( t,\cdot \right) =0\text{ \ a.e. in }\Omega ,\text{ }%
\forall t\geq 0, \\ 
\sigma _{t}\left( t,\cdot \right) =0\text{ \ a.e. in }\Omega ,\text{ }%
\forall t\geq 0,%
\end{array}%
\right.
\end{equation*}%
which yields%
\begin{equation}
S\left( t\right) \left( \varphi _{0},\sigma _{0}\right) =\left( \varphi
_{0},\sigma _{0}\right) \text{ \ }\forall t\geq 0.  \tag{5.31}
\end{equation}%
Thus, the energy functional $E\left( \varphi ,\sigma \right) $ is a strict Lyapunov function on the dynamical system $\left( H^{1}\left( \Omega \right) \times
L^{2}\left( \Omega \right) ,S\left( t\right) \right) $.
\end{proof}

\subsection{Stationary points set}
In this part, we analyze the structure of the stationary points of the problem (1.1)-(1.4). Our aim is to show that the set of stationary points is nonempty and bounded. \newline
If we add (1.1) to (1.3) and test the obtained equation with $\xi =1$,
thanks to (1.4), we have%
\begin{equation*}
\underset{\Omega }{\int }\left( \varphi \left( t\right) +\sigma \left(
t\right) \right) dx=\underset{\Omega }{\int }\left( \varphi _{0}+\sigma
_{0}\right) dx\text{ \ }\forall t\geq 0\text{.}
\end{equation*}%
Now, for any arbitrary constant $M\in \mathbb{R}$, define the following subspace of $H^{1}\left( \Omega \right) \times
L^{2}\left( \Omega \right) $:%
\begin{equation}
Z_{M}:=\left\{ \left( \varphi ,\sigma \right) \in H^{1}\left( \Omega \right)
\times L^{2}\left( \Omega \right) :\text{ }\underset{\Omega }{\int }\left(
\varphi +\sigma \right) dx=\left\vert \Omega \right\vert M\right\} \text{.} 
\tag{5.32}
\end{equation}%
The set $Z_{M}$ is endowed with the usual norm in $H^{1}\left( \Omega \right) \times
L^{2}\left( \Omega \right) $, and it is a closed affine subspace of $H^{1}\left(
\Omega \right) \times L^{2}\left( \Omega \right) $. Then, for $\left(
\varphi _{0},\sigma _{0}\right) \in Z_{M}$, we can define a semigroup $%
\left\{ S_{M}\left( t\right) \right\} _{t\geq 0}$ of weakly continuous
operators on $Z_{M}$ where $\left\{ S_{M}\left( t\right) \right\} _{t\geq 0}$
is the restriction of $\left\{ S\left( t\right) \right\} _{t\geq 0}$ on $
Z_{M}.$\\ 
	Let us denote the set of stationary points of the problem
(1.1)-(1.4) in $Z_{M}$ as follows:

	\begin{align}
\mathcal{N}_{M}=\left\lbrace \left( \varphi ,\sigma \right) \in Z_{M}:S_{M}\left( t\right)\left( \varphi ,\sigma \right)=\left( \varphi ,\sigma \right),  \forall t\geq 0\right\rbrace.\nonumber  
\end{align}
If $ \left( \varphi ,\sigma \right) \in \mathcal{N}_{M}$, from (3.19) (with equality sign), it follows that%
\begin{equation*}
\int\limits_{s}^{t}\left( \left\Vert \nabla \mu \right\Vert _{L_{2}\left(
	\Omega \right) }^{2}+\left\Vert \nabla N_{\sigma }\right\Vert _{L_{2}\left(
	\Omega \right) }^{2}\right) d\tau +\int\limits_{s}^{t}\int\limits_{\Omega
}p\left( \varphi \right) \left( N_{\sigma }-\mu \right) ^{2}dxd\tau =0, 
\end{equation*}%
which yields%
\begin{equation*}
\left\{ 
\begin{array}{rlr}
\mu =&\mu ^{0}&\text{ a.e. in }\Omega , \\ 
N_{\sigma }=&N_{\sigma }^{0}&\text{ a.e. in }\Omega , \\ 
\sqrt{p\left( \varphi \right) }\left( N_{\sigma }-\mu \right) =&0&\text{ a.e.
	in }\Omega ,%
\end{array}%
\right.
\end{equation*}%
where $\mu ^{0}$ are $N_{\sigma }^{0}$ are given by the following identities:%
\begin{equation*}
\mu ^{0}=N_{\sigma }^{0}=\frac{1}{\left\vert \Omega \right\vert }%
\int\limits_{\Omega }\Psi ^{\prime }\left( \varphi \right) dx-\frac{\chi
	_{\varphi }}{\left\vert \Omega \right\vert }\int\limits_{\Omega }\sigma dx%
\text{.}\tag{5.33}
\end{equation*}%
Therefore, we obtain that $\mathcal{N}_{M}$ consists of solutions of the following stationary problem.%
\begin{equation}
\left\{ 
\begin{array}{rl}
-\Delta \varphi +\Psi ^{\prime }\left( \varphi \right) -\chi _{\varphi
}\sigma =&\mu ^{0}, \\ 
\chi _{\sigma }\sigma +\chi _{\varphi }\left( 1-\varphi \right) =&\mu ^{0},
\\ 
\underset{\Omega }{\int }\left( \varphi +\sigma \right) dx=&\left\vert
\Omega \right\vert M.%
\end{array}%
\right.  \tag{5.34}
\end{equation}

Hence, to prove that $\mathcal{N}_{M}$ is nonempty, it is enough to show that (5.34) has at least one solution.\\
 We start with the following lemma regarding the energy functional
\begin{equation*}
E\left( \varphi ,\sigma \right) :=\frac{1}{2}\left\Vert \nabla \varphi
\right\Vert _{L_{2}\left( \Omega \right) }^{2}+\underset{\Omega }{\int }%
\Psi \left( \varphi \right) +\frac{\chi _{\sigma }}{2}\left\Vert \sigma
\right\Vert _{L_{2}\left( \Omega \right) }^{2}+\underset{\Omega }{\int }%
\chi _{\varphi }\sigma \left( 1-\varphi \right) .
\end{equation*}%
\begin{lemma}
	Assume that the conditions ($\Psi $) and ($P_{2}$) are satisfied. Then, energy functional $ {E}\left( \varphi ,\sigma\right) $ has at least one minimizer $ \left( \varphi^{*} ,\sigma^{*}\right)\in Z_{M} $ such that
	\begin{equation*}
	{E}\left( \varphi^{*} ,\sigma^{*}\right)=\inf_{\left( \varphi ,\sigma\right)\in Z_{M}} {E}\left( \varphi ,\sigma\right).
	\end{equation*}
\end{lemma}

\begin{proof}
	Using Young inequality, it is easy to see that ${E}$ is bounded from below in $Z_{M}$ by%
	\begin{equation}
	{E}\left( \varphi ,\sigma \right) \geq \frac{1}{2}%
	\int\limits_{\Omega }\left\vert \nabla \varphi \right\vert ^{2}dx+\left(
	R_{1}-\frac{2\chi _{\varphi }^{2}}{\chi _{\sigma }}\right)
	\int\limits_{\Omega }\left\vert \varphi \right\vert ^{2}dx+\frac{\chi
		_{\sigma }}{4}\int\limits_{\Omega }\left\vert \sigma \right\vert
	^{2}dx-R_{2}-\frac{2\chi _{\varphi }^{2}}{\chi _{\sigma }}\left\vert \Omega
	\right\vert  \tag{5.35}
	\end{equation}
	which means that due to $\Psi $ we have that ${E}$ has an infimum ${E}^{*}:=\underset%
	{\left( \varphi ,\sigma \right) \in Z_{M}}{\inf }{E}\left( \varphi,\sigma \right) $. Then, there exists a minimizing sequence $\left\{ \left(
	\varphi _{k},\sigma _{k}\right) \right\} _{k=1}^{\infty }\subset $ $Z_{M}$
	such that%
	\begin{equation*}
	 \lim_{k\rightarrow \infty }{E}\left( \varphi _{k},\sigma _{k}\right)={E}^{*}\text{ \ \ and \ \ }{E}\left( \varphi	_{k},\sigma _{k}\right) \leq {E}^{*}+1\text{, \ }\forall k\in 	\mathbb{N}
	\text{.} \tag{5.36}
	\end{equation*}
	Thanks to $ R_{1}>\frac{2\chi _{\varphi }^{2}}{\chi _{\sigma }} $, from (5.35) and (5.36), we obtain that $\left\{ \left( \varphi _{k},\sigma
	_{k}\right) \right\} _{k=1}^{\infty }$ is bounded in $H^{1}\left( \Omega
	\right) \times L^{2}\left( \Omega \right) $. Then, by the Banach-Alaoglu theorem
	there exists a subsequence $\left\{ \left( \varphi _{k_{m}},\sigma
	_{k_{m}}\right) \right\} _{m=1}^{\infty }$ such that%
	\begin{equation}
	\left\{ 
	\begin{array}{ll}
	\varphi _{k_{m}}\rightarrow {\varphi }^{*}&\text{ \ weakly in \ }%
	H^{1}\left( \Omega \right) , \\ 
	\sigma _{k_{m}}\rightarrow {\sigma }^{*}&\text{ \  weakly in \ }%
	L^{2}\left( \Omega \right) .%
	\end{array}%
	\right.  \tag{5.37}
	\end{equation}%
	From (5.37) and the compact embedding of $ H^{1}\left( \Omega \right) $ to $ L^{\rho}\left( \Omega \right) $, $ \rho<6 $, we also infer
	\begin{equation}
	\varphi _{k_{m}}\rightarrow {\varphi }^{*}\text{ strongly in }%
	L^{\rho}\left( \Omega \right) \text{ for all }\rho<6.  \tag{5.38}
	\end{equation}%
	Then, from (5.37)$_{2}$ and (5.38), it easily follows that%
	\begin{equation*}
	\lim_{m\rightarrow \infty }\int\limits_{\Omega }\chi _{\varphi }\left(
	1-\varphi _{k_{m}}\right) \sigma _{k_{m}}dx=\int\limits_{\Omega }\chi
	_{\varphi }\left( 1-\varphi \right) \sigma dx.
	\end{equation*}%
	Moreover, up to subsequence, (5.38) implies,%
	\begin{equation}
	\varphi _{k_{m}}\rightarrow {\varphi }^{*}\text{ \ \ a.e. in }\Omega . 
	\tag{5.39}
	\end{equation}
	Together with the continuity of $\Psi $, (5.39) yields%
	\begin{equation*}
	\Psi \left( \varphi _{k_{m}}\right) \rightarrow \Psi \left( {\varphi }^{*}\right) \text{ \ \ a.e. in }\Omega .
	\end{equation*}%
	Recalling the growth condition (3.4) on $\Psi $, we obtain that%
	\begin{equation*}
	\left\vert \Psi \left( \varphi _{k_{m}}\right) \right\vert \leq C\left(
	1+\left\vert \varphi _{k_{m}}\right\vert ^{\rho}\right)
	\end{equation*}%
	where $2\leq \rho<6$. Therefore, by using the generalized Lebesgue dominated
	convergence theorem, we deduce%
	\begin{equation*}
	\lim_{m\rightarrow \infty }\int\limits_{\Omega }\Psi \left( \varphi
	_{k_{m}}\right) =\int\limits_{\Omega }\Psi \left( {\varphi }^{*}
	\right) .
	\end{equation*}%
	On the other hand, the terms $\frac{1}{2}\int\limits_{\Omega }\left\vert \nabla \varphi
	\right\vert ^{2}dx$ and $\frac{\chi _{\sigma }}{2}\int\limits_{\Omega
	}\left\vert \sigma \right\vert ^{2}dx$ in the functional ${E}$
	are convex with respect to $\nabla \varphi $ and $\sigma $, so they are lower
	semicontinuous. Thus, we infer
	\begin{equation*}
	{E}\left( \varphi^{*} ,\sigma^{*}\right) \leq 
	\underset{m\rightarrow \infty }{\lim \inf }\text{ }{E}\left( \varphi
	_{k_{m}},\sigma _{k_{m}}\right) ={E}^{*}
	\end{equation*}%
	which proves that $\left( \varphi^{*} ,\sigma^{*}\right) $ is a minimizing point of ${E}$ on $Z_{M}.$
\end{proof}

\begin{lemma}
	Let the conditions ($\Psi $) and ($P_{2}$) hold and $\left( \varphi^{*} ,\sigma^{*}\right) $ be a minimizer of $	{E}\left( \varphi ,\sigma\right) $ in  $Z_{M}$. Then, $\left( \varphi^{*} ,\sigma^{*}\right) \in H^{2}\left( \Omega\right) \times H^{2}\left( \Omega\right)$ is a strong solution of the problem (5.33)-(5.34).
\end{lemma}
\begin{proof}
 Since $\left( \varphi ^{\ast },\sigma
^{\ast }\right)$ is a minimizer of ${E}$ in $Z_{M}$, we obtain for any $\left( \eta ,\xi \right) \in H^{1}\left(
\Omega \right) \times L^{2}\left( \Omega \right) $ such that $\underset{%
	\Omega }{\int }\left( \eta +\xi \right) dx=0$ that%
\begin{align*}
0=\left. \frac{d}{dt}\right\vert _{t=0}{E}\left( \varphi ^{\ast
}+t\eta ,\sigma ^{\ast }+t\xi \right) =\lim_{t\rightarrow 0}\frac{{E}%
	\left( \varphi ^{\ast }+t\eta ,\sigma ^{\ast }+t\xi \right) -{E}%
	\left( \varphi ^{\ast },\sigma ^{\ast }\right) }{t}.\tag{5.40}
\end{align*}%
Then, evaluating the limit in (5.40), for all $ \left( \eta ,\xi \right)\in Z_{0} $  we infer that%
\begin{equation*}
\int\limits_{\Omega }\left( \nabla \varphi ^{\ast }\cdot\nabla \eta +\Psi
^{\prime }\left( \varphi ^{\ast }\right) \eta +\chi _{\varphi }\eta \sigma
^{\ast } \right) dx+\int\limits_{\Omega }\left( \chi
_{\sigma }\sigma ^{\ast
}\xi +\chi _{\varphi }\xi \left( 1-\varphi ^{\ast }\right) 
\right) dx=0. \tag{5.41}
\end{equation*}%
Now, let us take $\left( \overline{\eta} ,\overline{\xi} \right) \in H^{1}\left(
\Omega \right) \times L^{2}\left( \Omega \right) $. Then, for $B=\frac{1}{2\left| \Omega\right|  }\int\limits_{\Omega } \left( \overline{\eta} +\overline{\xi} \right)dx $, consider
\begin{equation*}
 \eta=\overline{\eta}-B,\text{ \ }  \xi=\overline{\xi}-B
\end{equation*}
which is allowed as a test function in (5.41). Hence, we obtain
\begin{align}
&\int\limits_{\Omega }\left( \nabla \varphi ^{\ast }\cdot\nabla\overline{\eta} +\Psi
^{\prime }\left( \varphi ^{\ast }\right) \overline{\eta} +\chi _{\varphi }\overline{\eta} \sigma
^{\ast } \right) dx+\int\limits_{\Omega }\left( \chi
_{\sigma }\sigma ^{\ast
}\overline{\xi} +\chi _{\varphi }\overline{\xi} \left( 1-\varphi ^{\ast }\right) 
\right) dx\nonumber \\
&=\int\limits_{\Omega }B\left( \Psi
^{\prime }\left( \varphi ^{\ast }\right)+\chi _{\varphi }\sigma
^{\ast }+\chi
_{\sigma }\sigma
^{\ast }+\chi _{\varphi }\left( 1-\varphi ^{\ast }\right)\right)  dx .  \tag{5.42}
\end{align}
Choosing $\overline{\eta}=0 $ and $\overline{\xi} $ arbitrary in (5.42), we get
\begin{equation}
\frac{1}{\left\vert \Omega \right\vert }%
\int\limits_{\Omega }\Psi ^{\prime }\left( \varphi^{\ast } \right) dx-\frac{\chi
	_{\varphi }}{\left\vert \Omega \right\vert }\int\limits_{\Omega }\sigma^{\ast } dx=\chi _{\sigma }\sigma^{\ast } +\chi _{\varphi }\left( 1-\varphi^{\ast } \right).\tag{5.43}
\end{equation}
Hence, choosing $\overline{\xi}=0 $ and $\overline{\eta} $ arbitrary in (5.42), we obtain
\begin{equation*}
-\Delta \varphi^{\ast } +\Psi ^{\prime }\left( \varphi^{\ast } \right) -\chi _{\varphi
}\sigma^{\ast } =\mu ^{0},
\end{equation*}
where
\begin{equation*}
\mu ^{0}=\frac{1}{\left\vert \Omega \right\vert }%
\int\limits_{\Omega }\Psi ^{\prime }\left( \varphi^{\ast } \right) dx-\frac{\chi
	_{\varphi }}{\left\vert \Omega \right\vert }\int\limits_{\Omega }\sigma^{\ast } dx.
\end{equation*}
Moreover, it also follows from this and (5.43)
\begin{equation*} 
\chi _{\sigma }\sigma^{\ast } +\chi _{\varphi }\left( 1-\varphi^{\ast } \right) =\mu ^{0}.
\end{equation*}%
Therefore, the minimizing point of $E\left( \varphi ,\sigma \right) $ \
in $Z_{m}$, denoted by $\left( \varphi ^{\ast },\sigma ^{\ast }\right) $, is
a weak solution of the problem (5.33)-(5.34) in $Z_{m}$.\newline
Now, let us define the following function
\begin{equation*}
F_{N}\left( z\right) =\left\lbrace  
\begin{array}{ll}
\Psi_{0}^{\prime}\left(  z\right),  &\left| z\right| \leq N,\\
\Psi_{0}^{\prime}\left( N\right),   &z > N,\\
\Psi_{0}^{\prime}\left(-N\right),   &z < -N.
\end{array}
 \right.  
\end{equation*}
It is easy to see that, as $ N\rightarrow\infty $
\begin{equation}
F_{N}\left( \varphi\left( x\right) \right)\rightarrow\Psi_{0}^{\prime}\left(  \varphi\left( x\right) \right) \ \ a.e. \ in \ \Omega. \tag{5.44}
\end{equation}
Thus, replacing $ \left( \varphi,\sigma\right)  $ with $ \left( \varphi^{*},\sigma^{*}\right)  $ in the problem (5.33)-(5.34) and testing $ (5.34)_{1} $ with $ F_{N}\left( \varphi^{*}\right) $, we get
\begin{align*}
&\int\limits_{\Omega }\left| \nabla\varphi^{*}\right| ^{2}F_{N}^{\prime}\left( \varphi^{*}\right)dx+\int\limits_{\Omega }\Psi_{0}^{\prime}\left(  \varphi^{*}\right)F_{N}\left( \varphi^{*}\right)dx\\
& =-\int\limits_{\Omega }\lambda\left( \varphi^{*}\right)F_{N}\left( \varphi^{*}\right)dx+\chi_{\varphi^{*}}\int\limits_{\Omega }\sigma^{*} F_{N}\left( \varphi^{*}\right)dx+\mu_{0}\int\limits_{\Omega } F_{N}\left( \varphi^{*}\right)dx.
\end{align*}
Since $F_{N}^{\prime}\left( \varphi^{*}\right)\geq0 $ and $ F_{N}\left( \varphi^{*}\right)\leq\Psi_{0}^{\prime}\left(  \varphi^{*}\right) $, from the last equality we obtain
\begin{align*}
\int\limits_{\Omega }\left| F_{N} \left( \varphi^{*}\right)\right| ^{2}dx\leq C\left( \chi_{\varphi^{*}},\mu_{0}\right) \left( \left\| \sigma^{*}\right\| _{L^{2}\left( \Omega\right) }^{2}+\left\| \lambda\left( \varphi^{*}\right)\right\| _{L^{2}\left( \Omega\right) }^{2}+1\right) .
\end{align*}
Hence, considering (5.44) and applying Fatou's lemma in the last estimate we deduce
 $ \Psi_{0}^{\prime}\left(  \varphi^{*}\right) \in L^{2}\left( \Omega\right)  $. Hence, by elliptic regularity theory,  $\left( \varphi ^{\ast },\sigma
^{\ast }\right)\in H^{2}\left( \Omega\right) \times H^{2}\left( \Omega\right)$ is the strong solution of the problem (5.33)-(5.34).
\end{proof}
\begin{remark}
	In the case without chemotaxis stationary solutions $ \left( \varphi ^{\ast },\sigma^{\ast }\right) $ have the property that $ \sigma^{\ast } $ is constant (see \cite{cav}). In our case, $ \mathcal{N}_{M} $ consists of pairs $ \left( \varphi ^{\ast },\sigma^{\ast }\right) $ in which $ \sigma^{\ast } $ is not necessarily constant.
\end{remark}
\begin{lemma}
	Assume that the conditions ($\Psi $) and ($P_{2}$) are satisfied. Then, the set of stationary points $ \mathcal{N}_{M} $ is nonempty and bounded in $ Z_{M} $.
\end{lemma}
\begin{proof}
	From Lemma 5.3 and Lemma 5.4, we obtain that $ \mathcal{N}_{M} $ is nonempty. Now, testing $ (5.34)_{1} $ with $\varphi $, we get
\begin{equation} 
  \left\| \nabla\varphi\right\|_{L^{2}\left( \Omega\right) }  ^{2}+\int\limits_{\Omega }\Psi ^{\prime }\left( \varphi \right)\varphi dx=\mu ^{0}\int\limits_{\Omega }\varphi dx+\int\limits_{\Omega }{\chi}
  	_{\varphi }\sigma\varphi dx. \tag{5.45}
\end{equation}

Then, testing $ (5.34)_{2} $ with $\sigma $, we obtain
\begin{equation} 
{\chi}_{\sigma }\left\| \sigma\right\|_{L^{2}\left( \Omega\right)}  ^{2}=\mu ^{0}\int\limits_{\Omega }\sigma dx-\int\limits_{\Omega }{\chi}
_{\varphi }\sigma\left( 1-\varphi\right)  dx \tag{5.46}.
\end{equation}
From the assumptions (3.3) and (3.4) imposed on $ \Psi $, we can find a number $ K>M $ such that	$ \Psi^{\prime\prime }\left( s\right) $ is positive  for all $ \left| s\right| \geq K $. Then, it is easy to see that
\begin{equation*}
\left( s-M\right) \Psi^{\prime }\left( s\right)\geq\Psi\left( s\right)-\Psi\left( 0\right)-M\Psi^{\prime}\left( 0\right)\ \ \forall\left|  s\right| \geq K,
\end{equation*}
which together with (3.5), yields
\begin{equation}
\left( s-M\right) \Psi^{\prime }\left( s\right)\geq R_{1}\left| s\right| ^{2}-R_{3}  \ \  \forall \left| s\right| \geq K. \tag{5.47}
\end{equation}
Summing (5.45) and (5.46), we have
\begin{align*}
\left\| \nabla\varphi\right\|_{L^{2}\left( \Omega\right) }  ^{2}+{\chi}_{\sigma }\left\| \sigma\right\|_{L^{2}\left( \Omega\right)}  ^{2}+\int\limits_{\Omega}\Psi ^{\prime }\left( \varphi \right)\varphi dx=\mu^{0}\left| \Omega\right| M +2\int\limits_{\Omega}{\chi}
_{\varphi }\sigma\varphi dx-\int\limits_{\Omega}{\chi}
_{\varphi }\sigma dx.
\end{align*}
Using (5.33) in the the last inequality, we get
\begin{align*}
\left\| \nabla\varphi\right\|_{L^{2}\left( \Omega\right) }  ^{2}+{\chi}_{\sigma }\left\| \sigma\right\|_{L^{2}\left( \Omega\right)}  ^{2}+\int\limits_{\Omega}\Psi ^{\prime }\left( \varphi \right)\left(  \varphi-M\right)  dx=2\int\limits_{\Omega}{\chi}
_{\varphi }\sigma\varphi dx-\left( M+1\right) \int\limits_{\Omega}{\chi}
_{\varphi }\sigma dx
\end{align*}
which yields
\begin{align*}
\left\| \nabla\varphi\right\|_{L^{2}\left( \Omega\right) }  ^{2}+{\chi}_{\sigma }\left\| \sigma\right\|_{L^{2}\left( \Omega\right)}  ^{2}+\int\limits_{\left| \varphi\right| \geq K  }\Psi ^{\prime }\left( \varphi \right)\left( \varphi-M\right)  dx\nonumber\\
\leq C\left( M,K,{\chi}_{\varphi }\right) \left( 1+\int\limits_{\Omega }\left| \sigma\right|  dx\right) +2\int\limits_{\left| \varphi\right| \geq K  }{\chi}
_{\varphi }\left| \sigma\varphi\right|  dx
\end{align*} 
where $ C\left( M,K,{\chi}_{\varphi }\right) $ is a constant depending on $ M,K,{\chi}_{\varphi } $.
Hence, recalling $ R_{1}>\frac{2{\chi}_{\varphi }^{2}}{{\chi}_{\sigma }} $ and applying Young inequality in the last estimate, with the help of (5.47), we deduce
\begin{equation*}
\left\| \varphi\right\|_{H^{1}\left( \Omega\right) }  ^{2}+\left\| \sigma\right\|_{L^{2}\left( \Omega\right)}  ^{2}\leq
C\left( M,K,{\chi}_{\sigma },{\chi}_{\varphi }\right),
\end{equation*}
which completes the proof of lemma.
\end{proof}
\subsection{Results on the long-time behaviour }
\subsubsection{On the whole phase space $ H^{1}\left(\Omega \right) \times L^{2}\left(\Omega \right)$}
\begin{definition}
	Let $X$ be a Banach space. The $\omega -limit$ set of a set $A\subset X$ is defined as%
\begin{equation*}
	\omega \left( A\right) =\underset{t\geq 0}{\cap }\overline{\underset{\tau
			\geq t}{\cup }S\left( \tau \right) A}.
\end{equation*}
\end{definition}
Then, from Lemma 5.1, we get the following result (see \cite{chueshov}).
\begin{lemma}
		Let the conditions ($\Psi $) and ($P_{2}$) hold and $ B $ be a bounded set in $ H^{1}\left( \Omega \right) \times
	L^{2}\left( \Omega \right) $, then $  \omega\left( B\right)$ is nonempty, compact and invariant.
\end{lemma}
\subsubsection{On $ Z_{M} $}
\begin{definition}
	Let $\left\{ S\left( t\right) \right\} _{t\geq 0}$ be a semigroup on a
	metric space $\left( X,d\right) .$ A set $\mathcal{A}\subset X$ is called a
	global attractor for the semigroup $\left\{ S\left( t\right) \right\}
	_{t\geq 0}$, iff\newline
	$\bullet $ $\mathcal{A}$ is a compact set.\newline
	$\bullet $ $\mathcal{A}$ is invariant, i.e. $S\left( t\right) \mathcal{A}=%
	\mathcal{A}$, $\forall t\geq 0.$\newline
	$\bullet $ $\underset{t\rightarrow \infty }{\lim }\underset{v\in B}{\sup }%
	\underset{w\in \mathcal{A}}{\inf }d\left( S\left( t\right) v,w\right) =0,$
	for each bounded set $B\subset X.$
\end{definition}
\begin{definition}
	Let $\mathcal{N}$ be the set of stationary points of the dynamical system $ \left(X,S\left( t\right)  \right)  $.
	We define the unstable manifold $\mathcal{M}^{u}\left( \mathcal{N}\right) $
	emanating from the set $\mathcal{N}$ as a set of all $y\in X$ such that
	there exists a full trajectory $\gamma =\{u(t):t\in R\}$ with the properties%
	\[
	u(0)=y\text{ and }\lim_{t\rightarrow -\infty }dist_{X}(u(t),\mathcal{N})=0.
	\]
\end{definition}
From Lemma 5.1, we easily obtained the following asymptotic compactness lemma for the dynamical system $\left(Z_{M},S_{M}\left( t\right) \right) $.
\begin{lemma}
	Let the conditions ($\Psi $) and ($P_{2}$) hold and $B$ be a bounded subset
	of $Z_{M} $. Then
	every sequence of the form $\left\{ S\left( t_{k}\right) \left( \xi
	_{k},\eta _{k}\right) \right\} _{k=1}^{\infty },$ where $\left\{ \left( \xi
	_{k},\eta _{k}\right) \right\} _{k=1}^{\infty }\subset B$, $t_{k}\rightarrow
	\infty $, has a convergent subsequence in $Z_{M} $.
\end{lemma}
Moreover, from Lemma 5.2, we have gradient property also for $\left(Z_{M},S_{M}\left( t\right) \right) $ which is stated as follows.
\begin{lemma}
	Under conditions ($\Psi $) and ($P_{2}$), the dynamical system  $\left(Z_{M},S_{M}\left( t\right) \right) $ is a gradient system.
\end{lemma}
Hence, from Lemma 5.1, Lemma 5.6 and Lemma 5.7, applying the proof of \cite[Theorem 4.1]{khan}, we deduce the existence of the global attractor for the dynamical system $\left(Z_{M},S_{M}\left( t\right) \right) $. For the convenience of the reader, we state the theorem with its proof as follows.
\begin{theorem}
	Assume that the assumptions ($\Psi $) and ($P_{2}$) are satisfied. Then, the
	semigroup $\left\{ S_{M}\left( t\right) \right\} _{t\geq 0}$ generated by
	the weak solutions of the problem (1.1)-(1.4) possesses a global attractor $%
	\mathcal{A}_{M}$ in $Z_{M}$, and $\mathcal{A_{M}}=\mathcal{M}^{u}\left( \mathcal{
		N}_{M}\right) $.
\end{theorem}
\begin{proof}
	Let $B$ be a bounded subset \ of  $Z_{M}$. Then, Lemma 5.7 entails that $ \omega\left( B\right)  $ is compact, invariant and attracts $B.$ Let $\theta \in \omega \left(
	B\right) .$ By the invariance of $\omega \left( B\right) $, there exists a
	full trajectory $\left\{ \left( \varphi\left( t\right) ,\sigma\left( t\right)
	\right) ,\text{ }t\in
	%TCIMACRO{\U{211d} }%
	%BeginExpansion
	\mathbb{R}
	%EndExpansion
	\right\} \subset \omega \left( B\right) $ such that $\left( \varphi\left( 0\right)
	,\sigma\left( 0\right) \right) =\theta $.\\
	On the other hand, from lemma 5.8, energy functional\begin{equation*}
	E\left( \varphi ,\sigma \right) :=\frac{1}{2}\left\Vert \nabla \varphi
	\right\Vert _{L_{2}\left( \Omega \right) }^{2}+\underset{\Omega }{\int }%
	\Psi \left( \varphi \right) +\frac{\chi _{\sigma }}{2}\left\Vert \sigma
	\right\Vert _{L_{2}\left( \Omega \right) }^{2}+\underset{\Omega }{\int }%
	\chi _{\varphi }\sigma \left( 1-\varphi \right) .
	\end{equation*}%
	is the strict Lyapunov functional. Then, also considering the boundedness of \ $E\left( \left( \varphi\left(
	t\right) ,\sigma\left( t\right) \right) \right) ,$ we have%
	\begin{equation}
	\lim_{t\rightarrow -\infty }E\left( \left( \varphi\left( t\right) ,\sigma\left(
	t\right) \right) \right) =l.  \tag{5.48}
	\end{equation}%
	Now, let us define $\alpha \left( \theta \right) :=\left\{ \vartheta \in
	\omega \left( B\right) ,\text{ there exists a sequence }\left\{ \left(
	\varphi\left( t_{k}\right) ,\sigma\left( t_{k}\right) \right) \right\}
	_{k=1}^{\infty },\text{ such that}\right. $ \newline
	$\left. t_{k}\searrow -\infty \text{ and }\left( \varphi\left( t_{k}\right)
	,\sigma\left( t_{k}\right) \right) \rightarrow \vartheta \text{ strongly in } Z_{M}\right\} .$ It is easy to see that $\alpha \left( \theta \right) $
	is compact and invariant subset of $\omega \left( B\right) .$ Also, by
	(5.48), it follows that%
	\begin{equation*}
	E\left( \vartheta \right) =l\text{ \ }\forall \vartheta \in \alpha \left(
	\theta \right) \text{,}
	\end{equation*}%
	which, together with invariance of $\alpha \left( \theta \right) $, yields%
	\begin{equation*}
	E\left( S\left( t\right) \vartheta \right) =l\text{ \ }\forall \vartheta \in
	\alpha \left( \theta \right) ,\text{ \ }\forall t\geq 0.
	\end{equation*}%
	Since $ E $ is strict Lyapunov functional, from the last two equality we obtain $\vartheta \in \mathcal{N}_{M}$. Hence, we have that $\alpha \left( \theta \right) \subset \mathcal{N}_{M}$
	and consequently%
	\begin{equation}
	\omega \left( B\right) \subset \mathcal{M}^{u}\left( \mathcal{N}_{M}\right) .
	\tag{5.49}
	\end{equation}%
	By using the weak continuity and the asymptotic compactness of $\left\{
	S_{M}\left( t\right) \right\} _{t\geq 0}$ and the boundedness of $\mathcal{N}_{M}$
	in $Z_{M}$, it is easy to show that $\mathcal{M}^{u}\left( \mathcal{N_{M}}%
	\right) $ is invariant and compact. Thus, in virtue of (5.49), we obtain that
	$\mathcal{A}_{M}:=\mathcal{M}^{u}\left( \mathcal{N}\right) $ is a global
	attractor.
\end{proof}
\subsection{Regularity of the attractor}
In the following theorem, we will the regularity of the global
attractor.
\begin{theorem}
	Under the assumptions of Theorem 5.9, global attractor $\mathcal{A}_{M}$
	for the problem (1.1)-(1.4) is bounded in $\left( H^{3}\left( \Omega \right)\times
	H^{1}\left( \Omega \right)\right) \cap Z_{M} $.
\end{theorem}

\begin{proof}
	Let $\left( \varphi _{0},\sigma _{0}\right) \in \mathcal{A}_{M}$. Since $%
	\mathcal{A}_{M}$ is invariant, we obtain for any $t>0$%
	\begin{equation*}
	S_{M}\left( t\right) \mathcal{A}_{M}=\mathcal{A}_{M}
	\end{equation*}%
	which means that for any $\widetilde{t}>2$ there exists $\left( \widetilde{%
		\varphi }_{0},\widetilde{\sigma }_{0}\right) \in \mathcal{A}_{M}$ such that%
	\begin{equation*}
	\left( \varphi _{0},\sigma _{0}\right) =S_{M}\left( \widetilde{t}\right) \left( 
	\widetilde{\varphi }_{0},\widetilde{\sigma }_{0}\right) .
	\end{equation*}%
	On the other hand, since $\left( \widetilde{\varphi }_{0},\widetilde{\sigma }%
	_{0}\right) \in \mathcal{A}_{M}$ there exists an invariant trajectory \[ %
	\gamma :=\left\{ \left( \varphi \left( t\right) ,\sigma \left( t\right)
	\right) :t\in 
	%TCIMACRO{\U{211d} }%
	%BeginExpansion
	\mathbb{R}
	%EndExpansion
	\right\} \subset \mathcal{A}_{M} \] such that $\left( \varphi \left( 0\right)
	,\sigma \left( 0\right) \right) =\left( \widetilde{\varphi }_{0},\widetilde{%
		\sigma }_{0}\right) .$ Then,recalling the regularization of weak solutions, for all $t>0$, we have,%
	\begin{equation*}
	\int\limits_{t}^{t+1}\left( \left\Vert \varphi \left( \tau \right)
	\right\Vert _{H^{3}\left( \Omega \right) }^{2}+\left\Vert \sigma \left( \tau
	\right) \right\Vert _{H^{1}\left( \Omega \right) }^{2}\right) d\tau \leq
	c_{1}
	\end{equation*}%
	where $c_{1}$ is independent of $t$. Hence, using the mean value theorem
	for integrals, for all $t>0$, we deduce that%
	\begin{equation}
	\text{ there exists }t^{\ast }\in \left( t,t+1\right) ,\text{ such that }%
	\left\Vert \varphi \left( t^{\ast }\right) \right\Vert _{H^{3}\left( \Omega
		\right) }^{2}+\left\Vert \sigma \left( t^{\ast }\right) \right\Vert
	_{H^{1}\left( \Omega \right) }^{2}\leq c_{1}.  \tag{5.50}
	\end{equation}%
	If we choose $0<t<1$ in $\left( 5.50\right) $, we obtain that $\widetilde{t}%
	>t^{\ast }$, and 
	\begin{equation*}
	\left( \varphi _{0},\sigma _{0}\right) =S_{M}\left( \widetilde{t}-t^{\ast
	}\right) \left( \varphi \left( t^{\ast }\right) ,\sigma \left( t^{\ast
	}\right) \right) 
	\end{equation*}%
	which yields%
	\begin{equation*}
	\left\Vert \varphi _{0}\right\Vert _{H^{3}\left( \Omega \right)
	}^{2}+\left\Vert \sigma _{0}\right\Vert _{H^{1}\left( \Omega \right)
	}^{2}\leq c_{2}.
	\end{equation*}
\end{proof}
\section{Finite dimensionality of the global attractor}
\begin{definition}
	Let $ A $ be a compact set in a metric space $ X $. The fractal (box-counting) dimension $ dim_{f}A $ of $ A $ is defined by
	\begin{equation*}
	 dim_{f}A=\limsup_{\epsilon\rightarrow 0^{+}}\dfrac{lnN_{\epsilon}\left(A\right) }{ln\left(1/\epsilon \right) }
	\end{equation*}
	where $ N_{\epsilon}\left(A\right) $ is the minimal number of closed balls of the radius $ \epsilon $ which cover the set $ A $.
\end{definition}
In this section we benefit from the ideas of \cite{gal}, \cite{khan1} and \cite{MZ}). Let us start with the
following lemma.

\begin{lemma}
Assume that the assumptions of Theorem 5.9 hold. Also, let $\left\{ T\left(
t,\tau \right) \right\} _{t\geq \tau }$ be the process generated by the
problem%
\begin{equation}
\left\{ 
\begin{array}{rr}
\widehat{\varphi }_{t}(t,x)+A^{2}\widehat{\varphi }(t,x)-\chi _{\varphi }A%
\widehat{\sigma }(t,x)+R_{1}A\widehat{\varphi }(t,x)=0&\text{ \ \ }t\geq \tau \text{, }%
x\in \Omega, \\ 
\widehat{\sigma }_{t}(t,x)+\chi _{\sigma }A\widehat{\sigma }(t,x)-\chi _{\varphi }A%
\widehat{\varphi }(t,x)=0&\text{ \ \ }t\geq \tau \text{, }x\in \Omega,  \\ 
\partial _{\nu }\widehat{\varphi }(t,x)=\partial _{\nu }\widehat{\sigma }(t,x)=0&\text{%
 \ \ }t\geq \tau \text{, }x\in \Gamma,  \\ 
\widehat{\varphi }\left( \tau \right) =\widehat{\varphi }_{0}\text{, \ \ }%
\widehat{\sigma }\left( \tau \right) =\widehat{\sigma }_{0},&%
\end{array}%
\right.   \tag{6.1}
\end{equation}%
in $H^{1}\left( \Omega \right) \times L^{2}\left( \Omega \right) $. Then, 
\newline
i) There exist $\omega _{1}>0$ such that%
\begin{equation}
\left\Vert T\left( t,\tau \right) \right\Vert _{L\left( H^{1}\left( \Omega
\right) \times L^{2}\left( \Omega \right),H^{1}\left( \Omega
\right) \times L^{2}\left( \Omega \right) \right) }\leq 2e^{-\omega
_{1}\left( t-\tau \right) }\text{ \ }\forall t\geq \tau,   \tag{6.2}
\end{equation}%
where $L\left( X,Y\right) $ is the space of linear bounded operators from $X$ to $Y$.%
\newline
ii) There exist $\Lambda>1$ and $\omega _{2}>0$ such that%
\begin{equation}
\left\Vert T\left( t,\tau \right) \right\Vert _{L( \overset{\ast }{H^{1}%
	}\left( \Omega \right) \times \overset{\ast }{H^{1}}\left( \Omega \right)
	,H^{1}\left( \Omega \right) \times L^{2}\left( \Omega \right)  )}\leq \Lambda%
\frac{e^{-\omega _{2}\left( t-\tau \right) }}{\sqrt{t-\tau }}\text{ \ }%
\forall t\geq \tau .  \tag{6.3}
\end{equation}%
\end{lemma}

\begin{proof}
The proof consists of two parts. At first, we prove i) and then ii).\newline
\textbf{i)} Multiplying (6.1)$_{1}$ with $\widehat{\varphi }$ and (6.1)$_{2}$
with $\widehat{\sigma }$, we obtain \ 
\begin{align*}
&\frac{1}{2}\frac{d}{dt}\left( \left\Vert \widehat{\varphi }\left( t\right)
\right\Vert _{L^{2}\left( \Omega \right) }^{2}+\left\Vert \widehat{\sigma }%
\left( t\right) \right\Vert _{L^{2}\left( \Omega \right) }^{2}\right)
+\left\Vert \widehat{\varphi }\left( t\right) \right\Vert _{H^{2}\left(
\Omega \right) }^{2}+\chi _{\sigma }\left\Vert \widehat{\sigma }\left(
t\right) \right\Vert _{H^{1}\left( \Omega \right) }^{2}\\
&-2\chi _{\varphi }\int\limits_{\Omega }A^{1/2}\widehat{\sigma }\left(
t,x\right) A^{1/2}\widehat{\varphi }\left( t,x\right) dx+R_{1}\left\Vert 
\widehat{\varphi }\right\Vert _{H^{1}\left( \Omega \right) }^{2}=0. 
\tag{6.4}
\end{align*}
Since $R_{1}>\frac{2\chi _{\varphi }^{2}}{\chi _{\sigma }}$ and%
\begin{equation*}
\left\vert 2\chi _{\varphi }\int\limits_{\Omega }A^{1/2}\widehat{\sigma }%
\left( t,x\right) A^{1/2}\widehat{\varphi }\left( t,x\right) dx\right\vert
\leq \frac{\chi _{\sigma }}{2}\left\Vert \widehat{\sigma }\left( t\right)
\right\Vert _{H^{1}\left( \Omega \right) }^{2}+\frac{2\chi _{\varphi }^{2}}{%
\chi _{\sigma }}\left\Vert \widehat{\varphi }\right\Vert _{H^{1}\left(
\Omega \right) }^{2},
\end{equation*}%
from (6.4) it follows that%
\begin{equation}
\frac{1}{2}\frac{d}{dt}\left( \left\Vert \widehat{\varphi }\left( t\right)
\right\Vert _{L^{2}\left( \Omega \right) }^{2}+\left\Vert \widehat{\sigma }%
\left( t\right) \right\Vert _{L^{2}\left( \Omega \right) }^{2}\right)
+\left\Vert \widehat{\varphi }\left( t\right) \right\Vert _{H^{2}\left(
\Omega \right) }^{2}+\frac{\chi _{\sigma }}{2}\left\Vert \widehat{\sigma }%
\left( t\right) \right\Vert _{H^{1}\left( \Omega \right) }^{2}\leq 0. 
\tag{6.5}
\end{equation}%
Now, multiplying (6.1)$_{1}$ with $A\widehat{\varphi }$, we have%
\begin{equation*}
\frac{1}{2}\frac{d}{dt}\left( \left\Vert \widehat{\varphi }\left( t\right)
\right\Vert _{H^{1}\left( \Omega \right) }^{2}\right) +\left\Vert \widehat{%
\varphi }\left( t\right) \right\Vert _{H^{3}\left( \Omega \right)
}^{2}+R_{1}\left\Vert \widehat{\varphi }\right\Vert _{H^{2}\left( \Omega
\right) }^{2}=\chi _{\varphi }\int\limits_{\Omega }A^{1/2}\widehat{\sigma }%
\left( t,x\right) A^{3/2}\widehat{\varphi }\left( t,x\right) dx.
\end{equation*}%
Applying Young inequality in the right hand side of the las estimate, we
infer%
\begin{equation}
\frac{1}{2}\frac{d}{dt}\left( \left\Vert \widehat{\varphi }\left( t\right)
\right\Vert _{H^{1}\left( \Omega \right) }^{2}\right) +c_{1}\left\Vert 
\widehat{\varphi }\left( t\right) \right\Vert _{H^{3}\left( \Omega \right)
}^{2}\leq c_{2}\left\Vert \widehat{\sigma }\left( t\right) \right\Vert
_{H^{1}\left( \Omega \right) }^{2}.  \tag{6.6}
\end{equation}%
Then, after multiplying (6.5) with a big enough constant and adding the
obtained inequality to (6.6), we obtain%
\begin{equation*}
\frac{d}{dt}\left( \Theta \left( t\right) \right) +c_{3}\left( \left\Vert 
\widehat{\varphi }\left( t\right) \right\Vert _{H^{3}\left( \Omega \right)
}^{2}+\left\Vert \widehat{\sigma }\left( t\right) \right\Vert _{H^{1}\left(
\Omega \right) }^{2}\right) \leq 0,
\end{equation*}%
where $\Theta \left( t\right) =\left\Vert \widehat{\varphi }\left( t\right)
\right\Vert _{H^{1}\left( \Omega \right) }^{2}+\left\Vert \widehat{\varphi }%
\left( t\right) \right\Vert _{L^{2}\left( \Omega \right) }^{2}+\left\Vert 
\widehat{\sigma }\left( t\right) \right\Vert _{L^{2}\left( \Omega \right)
}^{2}$. Since $\left\Vert \widehat{\varphi }\left( t\right) \right\Vert
_{H^{3}\left( \Omega \right) }^{2}+\left\Vert \widehat{\sigma }\left(
t\right) \right\Vert _{H^{1}\left( \Omega \right) }^{2}\geq c\Theta \left(
t\right) $, from the last inequality, we get%
\begin{equation*}
\frac{d}{dt}\left( \Theta \left( t\right) \right) +c_{4}\Theta \left(
t\right) \leq 0,
\end{equation*}%
which yields%
\begin{equation}
\Theta \left( t\right) \leq e^{-c_{4}\left( t-\tau \right) }\Theta \left(
\tau \right) .  \tag{6.7}
\end{equation}%
Therefore, recalling the definition of $\Theta \left( t\right) $, from (6.7)
we deduce%
\begin{equation*}
\left\Vert \widehat{\varphi }\left( t\right) \right\Vert _{H^{1}\left(
\Omega \right) }^{2}+\left\Vert \widehat{\sigma }\left( t\right) \right\Vert
_{L^{2}\left( \Omega \right) }^{2}\leq 2e^{-c_{4}\left( t-\tau \right)
}\left( \left\Vert \widehat{\varphi }\left( \tau \right) \right\Vert
_{H^{1}\left( \Omega \right) }^{2}+\left\Vert \widehat{\sigma }\left( \tau
\right) \right\Vert _{L^{2}\left( \Omega \right) }^{2}\right) 
\end{equation*}%
which proves the desired estimate (6.2).\newline
\textbf{ii)} Now, we will prove (6.3). Firstly, testing (6.1)$_{1}$ with $%
A^{-1}\varphi $ and (6.1)$_{2}$ with $A^{-1}\sigma $, we obtain%
\begin{align*}
&\frac{1}{2}\frac{d}{dt}\left( \left\Vert A^{-1/2}\widehat{\varphi }\left(
t\right) \right\Vert _{L^{2}\left( \Omega \right) }^{2}+\left\Vert A^{-1/2}%
\widehat{\sigma }\left( t\right) \right\Vert _{L^{2}\left( \Omega \right)
}^{2}\right) +\left\Vert \widehat{\varphi }\left( t\right) \right\Vert
_{H^{1}\left( \Omega \right) }^{2}+\chi _{\sigma }\left\Vert \widehat{\sigma 
}\left( t\right) \right\Vert _{L^{2}\left( \Omega \right) }^{2}\\
&-2\chi _{\varphi }\int\limits_{\Omega }\widehat{\sigma }\left( t,x\right) 
\widehat{\varphi }\left( t,x\right) dx+R_{1}\left\Vert \widehat{\varphi }%
\right\Vert _{L^{2}\left( \Omega \right) }^{2}=0.  \tag{6.8}
\end{align*}
Then, by using the similar estimates obtained in the first step, (6.8)
yields that%
\begin{equation}
\frac{1}{2}\frac{d}{dt}\left( \left\Vert \widehat{\varphi }\left( t\right)
\right\Vert _{\overset{\ast }{H^{1}}\left( \Omega \right) }^{2}+\left\Vert 
\widehat{\sigma }\left( t\right) \right\Vert _{\overset{\ast }{H^{1}}\left(
\Omega \right) }^{2}\right) +\left\Vert \widehat{\varphi }\left( t\right)
\right\Vert _{H^{1}\left( \Omega \right) }^{2}+\frac{\chi _{\sigma }}{2}%
\left\Vert \widehat{\sigma }\left( t\right) \right\Vert _{L^{2}\left( \Omega
\right) }^{2}\leq 0.  \tag{6.9}
\end{equation}%
Since $\left\Vert \widehat{\varphi }\left( t\right) \right\Vert
_{H^{1}\left( \Omega \right) }^{2}+\left\Vert \widehat{\sigma }\left(
t\right) \right\Vert _{L^{2}\left( \Omega \right) }^{2}\geq c\left(
\left\Vert \widehat{\varphi }\left( t\right) \right\Vert _{\overset{\ast }{%
H^{1}}\left( \Omega \right) }^{2}+\left\Vert \widehat{\sigma }\left(
t\right) \right\Vert _{\overset{\ast }{H^{1}}\left( \Omega \right)
}^{2}\right) $ for some $c>0$, from the last inequality there holds%
\begin{align*}
&\frac{d}{dt}\left( \left\Vert \widehat{\varphi }\left( t\right) \right\Vert
_{\overset{\ast }{H^{1}}\left( \Omega \right) }^{2}+\left\Vert \widehat{%
\sigma }\left( t\right) \right\Vert _{\overset{\ast }{H^{1}}\left( \Omega
\right) }^{2}\right) +c_{5}\left( \left\Vert \widehat{\varphi }\left(
t\right) \right\Vert _{\overset{\ast }{H^{1}}\left( \Omega \right)
}^{2}+\left\Vert \widehat{\sigma }\left( t\right) \right\Vert _{\overset{%
\ast }{H^{1}}\left( \Omega \right) }^{2}\right)\\
&+c_{6}\left( \left\Vert \widehat{\varphi }\left( t\right) \right\Vert
_{H^{1}\left( \Omega \right) }^{2}+\frac{\chi _{\sigma }}{2}\left\Vert 
\widehat{\sigma }\left( t\right) \right\Vert _{L^{2}\left( \Omega \right)
}^{2}\right) \leq 0. 
\end{align*}
From the last estimate, we deduce%
\begin{align*}
&\left\Vert \widehat{\varphi }\left( t\right) \right\Vert _{\overset{\ast }{%
H^{1}}\left( \Omega \right) }^{2}+\left\Vert \widehat{\sigma }\left(
t\right) \right\Vert _{\overset{\ast }{H^{1}}\left( \Omega \right)
}^{2}+\int\limits_{\tau }^{t}\left( \left\Vert \widehat{\varphi }\left(
t\right) \right\Vert _{H^{1}\left( \Omega \right) }^{2}+\frac{\chi _{\sigma }%
}{2}\left\Vert \widehat{\sigma }\left( t\right) \right\Vert _{L^{2}\left(
\Omega \right) }^{2}\right)ds\\
&\leq c_{7}e^{-c_{5}\left( t-\tau \right) }\left( \left\Vert \widehat{\varphi 
}\left( \tau \right) \right\Vert _{\overset{\ast }{H^{1}}\left( \Omega
	\right) }^{2}+\left\Vert \widehat{\sigma }\left( \tau \right) \right\Vert _{%
	\overset{\ast }{H^{1}}\left( \Omega \right) }^{2}\right),
\end{align*}%
and in particular%
\begin{equation}
\int\limits_{\tau }^{t}\left( \left\Vert \widehat{\varphi }\left( t\right)
\right\Vert _{H^{1}\left( \Omega \right) }^{2}+\left\Vert \widehat{\sigma }%
\left( t\right) \right\Vert _{L^{2}\left( \Omega \right) }^{2}\right) ds\leq
c_{7}e^{-c_{5}\left( t-\tau \right) }\left( \left\Vert \widehat{\varphi }%
\left( \tau \right) \right\Vert _{\overset{\ast }{H^{1}}\left( \Omega
\right) }^{2}+\left\Vert \widehat{\sigma }\left( \tau \right) \right\Vert _{%
\overset{\ast }{H^{1}}\left( \Omega \right) }^{2}\right) .  \tag{6.10}
\end{equation}%
Next, testing (6.1)$_{1}$ with $\left( t-\tau \right) A^{-1}\widehat{\varphi 
}_{t}$ and (6.1)$_{2}$ with $\left( t-\tau \right) A^{-1}\widehat{\sigma }%
_{t}$, we have%
\begin{equation}
\frac{d}{dt}\left( \left( t-\tau \right) \Phi \left( t\right) \right)
+\left( t-\tau \right) \left\Vert A^{-1/2}\widehat{\varphi }_{t}\left(
t\right) \right\Vert _{L^{2}\left( \Omega \right) }^{2}+\left( t-\tau
\right) \left\Vert A^{-1/2}\widehat{\sigma }_{t}\left( t\right) \right\Vert
_{L^{2}\left( \Omega \right) }^{2}=\Phi \left( t\right)   \tag{6.11}
\end{equation}%
where $\Phi \left( t\right) =\frac{1}{2}\left\Vert \widehat{\varphi }\left(
t\right) \right\Vert _{H^{1}\left( \Omega \right) }^{2}+\frac{\chi _{\sigma }%
}{2}\left\Vert \widehat{\sigma }\left( t\right) \right\Vert _{L^{2}\left(
\Omega \right) }^{2}+\frac{R_{1}}{2}\left\Vert \widehat{\varphi }\left(
t\right) \right\Vert _{L^{2}\left( \Omega \right) }^{2}-\chi _{\varphi
}\int\limits_{\Omega }\widehat{\varphi }\left( t\right) \widehat{\sigma }%
\left( t\right) $. Hence, integrating (6.11) from $\tau $ to $t$, we obtain%
\begin{equation}
\left( t-\tau \right) \Phi \left( t\right) \leq \int\limits_{\tau }^{t}\Phi
\left( s\right) ds.  \tag{6.12}
\end{equation}%
Thanks to $R_{1}>\frac{2\chi _{\varphi }^{2}}{\chi _{\sigma }}$, it can be
readily obtained that%
\begin{equation}
M_{1}\left( \left\Vert \widehat{\varphi }\left( t\right) \right\Vert
_{H^{1}\left( \Omega \right) }^{2}+\left\Vert \widehat{\sigma }\left(
t\right) \right\Vert _{L^{2}\left( \Omega \right) }^{2}\right) \leq \Phi
\left( t\right) \leq M_{2}\left( \left\Vert \widehat{\varphi }\left(
t\right) \right\Vert _{H^{1}\left( \Omega \right) }^{2}+\left\Vert \widehat{%
\sigma }\left( t\right) \right\Vert _{L^{2}\left( \Omega \right)
}^{2}\right)   \tag{6.13}
\end{equation}%
for some constants $M_{1}$, $M_{2}>0$. Thus, considering (6.13) in (6.12),
we have%
\begin{equation*}
\left( t-\tau \right) \left( \left\Vert \widehat{\varphi }\left( t\right)
\right\Vert _{H^{1}\left( \Omega \right) }^{2}+\left\Vert \widehat{\sigma }%
\left( t\right) \right\Vert _{L^{2}\left( \Omega \right) }^{2}\right) \leq
c_{8}\int\limits_{\tau }^{t}\left( \left\Vert \widehat{\varphi }\left(
t\right) \right\Vert _{H^{1}\left( \Omega \right) }^{2}+\left\Vert \widehat{%
\sigma }\left( t\right) \right\Vert _{L^{2}\left( \Omega \right)
}^{2}\right) ds.
\end{equation*}%
Consequently, the last estimate and (6.10) proves (6.3).
\end{proof}

\begin{lemma}
The fractal dimension of $\mathcal{A}_{M}$ is finite.
\end{lemma}

\begin{proof}
Let $\theta _{1},$ $\theta _{2}\in \mathcal{A}_{M}.$ Then
there exist full trajectories such that $\left( \varphi _{1}\left( t\right)
,\sigma _{1}\left( t\right) \right) =S_{M}\left( t\right) \theta _{1},$ $\left(
\varphi _{2}\left( t\right) ,\sigma _{2}\left( t\right) \right) =S_{M}\left(
t\right) \theta _{2}$. Define 
\begin{equation*}
\varphi  =\varphi _{1} -\varphi _{2} \text{, \ }\sigma =\sigma _{1}
-\sigma _{2} \text{, \ }\mu  =\mu _{1} -\mu _{2} \text{, \ \ }N_{\sigma }
=N_{\sigma _{1}} -N_{\sigma _{2}}
\end{equation*}%
where $N_{\sigma _{1}}=\chi _{\sigma }\sigma _{1}+\chi _{\varphi }\left(
1-\varphi _{1}\right) $ and $N_{\sigma _{2}}=\chi _{\sigma }\sigma _{2}+\chi
_{\varphi }\left( 1-\varphi _{2}\right) $. Then $\left( \varphi \left(
t\right) ,\sigma \left( t\right) \right) $ is the solution of the following
problem.%
\begin{equation*}
\left\{ 
\begin{array}{r}
\varphi _{t}=\Delta \mu +\left( p\left( \varphi _{1}\right) -p\left( \varphi
_{2}\right) \right) \left( N_{\sigma _{1}}-\mu _{1}\right) +p\left( \varphi
_{2}\right) \left( N_{\sigma }-\mu \right) \ \ \text{\ in }\left( 0,T\right) \times \Omega 
, \\ 
\mu =-\Delta \varphi +\Psi ^{\prime }\left( \varphi _{1}\right) -\Psi
^{\prime }\left( \varphi _{2}\right) -\chi _{\varphi }\sigma \ \ \text{\ in }%
\left( 0,T\right)\times\Omega   , \\ 
\sigma _{t}=\chi _{\sigma }\Delta \sigma -\chi _{\varphi }\Delta \varphi
-\left( p\left( \varphi _{1}\right) -p\left( \varphi _{2}\right) \right)
\left( N_{\sigma _{1}}-\mu _{1}\right) -p\left( \varphi _{2}\right) \left(
N_{\sigma }-\mu \right) \ \ \text{\ in }\left( 0,T\right) \times \Omega  , \\ 
\partial _{\nu }\mu =\partial _{\nu }\varphi =\partial _{\nu }\sigma =0\text{\ \
on }\left( 0,T\right)\times\Gamma.%
\end{array}%
\right. 
\end{equation*}%
Then, let us rewrite the above problem as follows.%
\begin{equation}
\left\{ 
\begin{array}{r}
\varphi _{t}+A\mu =\mu +B\ \ \text{\ in }\left( 0,T\right) \times \Omega  ,
\\ 
\mu =A\varphi +\Psi ^{\prime }\left( \varphi _{1}\right) -\Psi ^{\prime
}\left( \varphi _{2}\right) -\chi _{\varphi }\sigma -\varphi \ \ \text{\ in }%
\left( 0,T\right) \times \Omega , \\ 
\sigma _{t}+\chi _{\sigma }A\sigma -\chi _{\varphi }A\varphi =\chi _{\sigma
}\sigma -\chi _{\varphi }\varphi -B\ \ \text{\ in }\left( 0,T\right) \times \Omega ,%
\end{array}%
\right.   \tag{6.14}
\end{equation}%
where $B=\left( p\left( \varphi _{1}\right) -p\left( \varphi _{2}\right)
\right) \left( N_{\sigma _{1}}-\mu _{1}\right) +p\left( \varphi _{2}\right)
\left( N_{\sigma }-\mu \right) .$ Considering (6.14)$_{2}$ in (6.14)$_{1}$,
we get the following problem. 
\begin{equation*}
\left\{ 
\begin{array}{l}
\varphi _{t}+A^{2}\varphi -\chi _{\varphi }A\sigma +R_{1}A\varphi =\left(
R_{1}-1\right) A\varphi  \\
-A\left( \Psi ^{\prime }\left( \varphi _{1}\right) -\Psi ^{\prime }\left(
\varphi _{2}\right) \right) -\Delta \varphi +\Psi ^{\prime }\left( \varphi
_{1}\right) -\Psi ^{\prime }\left( \varphi _{2}\right) -\chi _{\varphi
}\sigma +B, \\ 
\sigma _{t}+\chi _{\sigma }A\sigma -\chi _{\varphi }A\varphi =\chi _{\sigma
}\sigma -\chi _{\varphi }\varphi -B.%
\end{array}%
\right. 
\end{equation*}%
Then, by the variation of the parameters formula, we find%
\begin{equation}
\left( \varphi \left( t\right) ,\sigma \left( t\right) \right) =T\left(
t,0\right) \left( \varphi \left( 0\right) ,\sigma \left( 0\right) \right)
+\int\limits_{0}^{t}T\left( t,s\right) F\left( s\right) ds  \tag{6.15}
\end{equation}%
where
\begin{align*}
F =\left( \left( R_{1}-1\right) A\varphi -A\left( \Psi
^{\prime }\left( \varphi _{1}\right) -\Psi ^{\prime }\left( \varphi
_{2}\right) \right) -\Delta \varphi +\Psi ^{\prime }\left( \varphi
_{1}\right) -\Psi ^{\prime }\left( \varphi _{2}\right) -\chi _{\varphi
}\sigma +B,\chi _{\sigma }\sigma -\chi _{\varphi }\varphi -B\right) .
\end{align*}
Hence, using previous lemma we obtain
\begin{align*}
&\left\Vert \left( \varphi \left( t\right) ,\sigma \left( t\right) \right)
\right\Vert _{H^{1}\left( \Omega \right) \times L^{2}\left( \Omega \right) }\\
&\leq \left\Vert T\left( t,0\right) \left( \varphi \left( 0\right) ,\sigma
\left( 0\right) \right) \right\Vert _{H^{1}\left( \Omega \right) \times
	L^{2}\left( \Omega \right) }+\int\limits_{0}^{t}\left\Vert T\left(
t,s\right) F\left( s\right) \right\Vert _{H^{1}\left( \Omega \right) \times
	L^{2}\left( \Omega \right) }ds\text{ }\\
&\leq 2e^{-\omega _{1}t}\left\Vert \theta _{2}-\theta _{1}\right\Vert
_{H^{1}\left( \Omega \right) \times L^{2}\left( \Omega \right)
}+c_{1}\int\limits_{0}^{t}\frac{e^{-\omega _{2}\left( t-s\right) }}{\sqrt{%
		t-s}}\left\Vert F\left( s\right) \right\Vert _{\overset{\ast }{H^{1}}\left(
	\Omega \right) \times \overset{\ast }{H^{1}}\left( \Omega \right) }ds. 
\tag{6.16}
\end{align*}%
On the other hand, since the attractor $A_{m}$ is bounded in $H^{3}\left(
\Omega \right) \times H^{1}\left( \Omega \right) $, we have $\varphi \in
H^{3}\left( \Omega \right) \hookrightarrow L^{\infty }\left( \Omega \right) $
which yields%
\begin{equation*}
\left\Vert F\left( s\right) \right\Vert _{\overset{\ast }{H^{1}}\left(
\Omega \right) \times \overset{\ast }{H^{1}}\left( \Omega \right) }\leq
c_{2}\left\Vert \left( \varphi \left( t\right) ,\sigma \left( t\right)
\right) \right\Vert _{H^{1}\left( \Omega \right) \times L^{2}\left( \Omega
\right) }.
\end{equation*}%
Taking into account the last estimate in (6.16), we get%
\begin{align*}
&\left\Vert \left( \varphi \left( t\right) ,\sigma \left( t\right) \right)
\right\Vert _{H^{1}\left( \Omega \right) \times L^{2}\left( \Omega \right)
}\leq 2e^{-\omega _{1}t}\left\Vert \theta _{2}-\theta _{1}\right\Vert
_{H^{1}\left( \Omega \right) \times L^{2}\left( \Omega \right) }\\
&+c_{1}\int\limits_{0}^{t}\frac{1}{\sqrt{t-s}}\left\Vert \left( \varphi
\left( t\right) ,\sigma \left( t\right) \right) \right\Vert _{H^{1}\left(
	\Omega \right) \times L^{2}\left( \Omega \right) }ds. \tag{6.17}
\end{align*}%
Then, by the Gronwall inequality, we deduce%
\begin{equation}
\left\Vert S\left( t\right) \theta _{2}-S\left( t\right) \theta
_{1}\right\Vert _{H^{1}\left( \Omega \right) \times L^{2}\left( \Omega
\right) }\leq 2e^{-\omega _{1}t+c_{1}2\sqrt{t}}\left\Vert \theta _{2}-\theta
_{1}\right\Vert _{H^{1}\left( \Omega \right) \times L^{2}\left( \Omega
\right) }\text{.}  \tag{6.18}
\end{equation}%
Moreover, recalling boundedness of the $A_{m}$ in $H^{3}\left( \Omega
\right) \times H^{1}\left( \Omega \right) $ once more and using
interpolation in the last term of (6.17), we obtain%
\begin{align*}
&\left\Vert \left( \varphi \left( t\right) ,\sigma \left( t\right) \right)
\right\Vert _{H^{1}\left( \Omega \right) \times L^{2}\left( \Omega \right)
}\leq 2e^{-\omega _{1}t}\left\Vert \theta _{2}-\theta _{1}\right\Vert
_{H^{1}\left( \Omega \right) \times L^{2}\left( \Omega \right) }\\
&+c_{2}\sqrt{t}\int\limits_{0}^{t}\frac{1}{\sqrt{t-s}}\left\Vert \left(
\varphi \left( s\right) ,\sigma \left( s\right) \right) \right\Vert _{%
	\overset{\ast }{H^{1}}\left( \Omega \right) \times \overset{\ast }{H^{1}}%
	\left( \Omega \right) }ds,
\end{align*}
which yields 
\begin{align*}
&\left\Vert S\left( t\right) \theta _{2}-S\left( t\right) \theta
_{1}\right\Vert _{H^{1}\left( \Omega \right) \times L^{2}\left( \Omega
\right) }\\
&\leq 2e^{-\omega _{1}t}\left\Vert \theta _{2}-\theta _{1}\right\Vert
_{H^{1}\left( \Omega \right) \times L^{2}\left( \Omega \right)
}+c_{3}t\sup_{0\leq s\leq t}\left\Vert S\left( s\right) \theta _{2}-S\left(
s\right) \theta _{1}\right\Vert _{\overset{\ast }{H^{1}}\left( \Omega
	\right) \times \overset{\ast }{H^{1}}\left( \Omega \right) } . \tag{6.19}
\end{align*}%
Consequently, exploiting \cite[Theorem 7.9.6]{chueshov}, (6.18) and (6.19) yield the
finite dimensionality of the attractor.
\end{proof}
\section*{Acknowledgements}

This work was supported by the Scientific and Technological Research Council of Turkey (TUBITAK) and by the RTG 2339 “Interfaces, Complex Structures, and Singular Limits” of the
German Science Foundation (DFG). The support is
gratefully acknowledged.


\begin{thebibliography}{9}
\bibitem{cav} Cavaterra, C., Rocca, E., Wu, H., Long-Time Dynamics and Optimal Control of a Diffuse Interface Model for Tumor Growth. Appl. Math. Optim., 1--49 (2019)
	
\bibitem{chueshov} Chueshov, I., Lasiecka, I., Von Karman Evolution Equations: Well-posedness and long-time dynamics. Springer, Berlin (2010)

\bibitem{colli1} Colli, P., Gilardi, G., Hilhorst, D., On a
Cahn--Hilliard type phase field model related to tumor growth. Discrete Contin. Dyn. Syst. \textbf{35}(6), 2423--2442 (2015)

\bibitem{colli2} Colli, P., Gilardi, G., Rocca, E. Sprekels, J., Vanishing viscosities and error estimate for a Cahn-Hilliard type phase
field system related to tumor growth. Nonlinear Anal. Real World Appl. \textbf{26},
93--108 (2015)

\bibitem{frigieri} Frigieri, S., Graselli, M., Rocca, E., On a diffuse
interface model of tumor growth. European Journal of Applied Mathematics \textbf{26}(2), 215--243 (2015)

\bibitem{gal} Gal, C. G., Exponential attractors for a Cahn-Hilliard model in bounded domains with permeable walls. Electronic Journal of Differential Equations \textbf{143}, (2006)

\bibitem{garcke1} Garcke, H., Lam, K. F., Sitka, E., Styles, V., A
Cahn--Hilliard--Darcy model for tumor growth with chemotaxis and active-transport. Math. Models Methods Appl. Sci. \textbf{26}(6), 1095--1148 (2016)

\bibitem{garcke2} Garcke, H., Lam, K.F., Global weak solutions and asymptotic limits of a Cahn-Hilliard-Darcy system modelling tumor growth. AIMS Mathematics \textbf{1}(3), 318--360 (2016)

\bibitem{garcke3} Garcke, H., Lam, K.F., Well-posedness of a
Cahn--Hilliard system modelling tumor growth with chemotaxis and active transport. European J. Appl. Math. \textbf{28}, 284--316 (2017)

\bibitem{garcke4} Garcke, H., Lam, K. F., Analysis of a Cahn--Hilliard system with non zero Dirichlet conditions modelling tumour growth with
chemotaxis. Discrete and Continuous Dynamical Systems \textbf{37}(8), 4277--4308 (2017)

\bibitem{hawkins-daarud} Hawkins-Daarud, A., Van der Zee, K.G., Oden, J. T., Numerical simulation of a thermodynamically consistent four-species tumor growth model. Int. J. Numer. Methods Biomed. Eng. \textbf{28}, 3--24 (2012)

\bibitem{HKNZ} Hilhorst, D., Kampmann, J., Nguyen, T. N., Van Der Zee, K. G., Formal asymptotic limit of a diffuse-interface tumor-growth model. Math. Models Methods Appl. Sci. \textbf{25}(6), 1011--1043 (2015)

\bibitem{khan} Khanmamedov, A., Yayla, S., Global attractors for the 2D hyperbolic Cahn–Hilliard equations. Z. Angew. Math. Phys. \textbf{69}, 14 (2018)

\bibitem{khan1} Khanmamedov, A., Global attractors for the plate equation with a localized damping and
a critical exponent in an unbounded domain. J.Differential Equations \textbf{225}, 528--548 (2006)

\bibitem{khan2} Khanmamedov, A., Global attractors for 2-D wave equations with displacement dependent damping. Math. Methods
Appl. Sci. \textbf{33}, 177--187 (2010)
 
\bibitem{MZ} Miranville, A. and Zelik, S., Exponential attractors for the Cahn-Hilliard equation with
dynamic boundary conditions. Math. Methods Appl. Sci. \textbf{28}(6), 709--735 (2005)

\bibitem{MRS} Miranville, A. Rocca, E. and Schimperna, G., On the long time behavior of a tumor growth model. J. Differential Equations \textbf{267}(4), 2616--2642 (2019)

\bibitem{miranville} Miranville, A., The Cahn-Hilliard Equation: Recent Advances and Applications. SIAM, (2019)







\end{thebibliography}
\end{document}